\documentclass[11pt]{article}

\usepackage{amsfonts}
\usepackage{amsthm}
\usepackage{latexsym, amssymb, bbm}
\usepackage{xcolor}
\usepackage{graphicx}
\usepackage{amsmath}
\usepackage{enumerate}
\usepackage{enumitem}
\usepackage[all]{xy}
\usepackage[margin=3cm]{geometry}
\usepackage[mathscr]{eucal}
\usepackage{latexsym, amssymb, bbm, mathrsfs,mathabx}
 \usepackage{caption} \usepackage{subcaption}
\usepackage{amscd}

\usepackage{amsmath,amsthm,amsfonts,amssymb,graphicx,psfrag}
\usepackage{color}

\usepackage{hyperref}
\hypersetup{
	backref,colorlinks=true,linkcolor=blue,citecolor=blue,urlcolor=blue,citebordercolor={0 0 1},urlbordercolor={0 0 1},linkbordercolor={0 0 1}}

\usepackage{tikz}

\newtheorem{thm}{Theorem}[section]
\newtheorem{theorem}[thm]{Theorem}

\newtheorem{prop}[thm]{Proposition}
\newtheorem{proposition}[thm]{Proposition}
\newtheorem{lemma}[thm]{Lemma}
\newtheorem{remark}[thm]{Remark}
\newtheorem{rmk}[thm]{Remark}

\newtheorem{defn}[thm]{Definition}
\newtheorem{definition}[thm]{Definition}
\newtheorem{corr}[thm]{Corollary}
\newtheorem{corollary}[thm]{Corollary}

\newtheorem{fact}[thm]{Fact}
\newtheorem{example}[thm]{Example}
\newtheorem{conjecture}[thm]{Conjecture}

\theoremstyle{definition}

\DeclareMathOperator{\aut}{Aut}

\DeclareMathOperator{\Diff}{Diff}

\newcommand{\mb}[1]{\mathbb{#1}}
\newcommand{\mc}[1]{\mathcal{#1}}

\newcommand{\bdf}{\begin{defn}}
\newcommand{\edf}{\end{defn}}

\newcommand{\bthm}{\begin{thm}}
\newcommand{\ethm}{\end{thm}}

\newcommand{\blem}{\begin{lemma}}
\newcommand{\elem}{\end{lemma}}

\newcommand{\bcor}{\begin{corr}}
\newcommand{\ecor}{\end{corr}}

\newcommand{\bprop}{\begin{prop}}
\newcommand{\eprop}{\end{prop}}

\newcommand{\brmk}{\begin{rmk}}
\newcommand{\ermk}{\end{rmk}}

\newcommand{\bpf}{\begin{proof}}
\newcommand{\epf}{\end{proof}}

\newcommand{\bex}{\begin{example}}
\newcommand{\eex}{\end{example}}

\numberwithin{equation}{section}

\def\R{\mathbb{R}}
\def\Z{\mathbb{Z}}

\def\CP{\mathbb{CP}}

\def\mb{\mathbf}

\def\Symp{\mbox{Symp}}

\def\xkm2{\overline{X}_{k-2}}

\usepackage{todonotes}

\usetikzlibrary{cd}
\RequirePackage{tikz-cd}
\RequirePackage{amssymb}
\usetikzlibrary{calc}
\usetikzlibrary{decorations.pathmorphing}

\tikzset{curve/.style={settings={#1},to path={(\tikztostart)
			.. controls ($(\tikztostart)!\pv{pos}!(\tikztotarget)!\pv{height}!270:(\tikztotarget)$)
			and ($(\tikztostart)!1-\pv{pos}!(\tikztotarget)!\pv{height}!270:(\tikztotarget)$)
			.. (\tikztotarget)\tikztonodes}},
	settings/.code={\tikzset{quiver/.cd,#1}
		\def\pv##1{\pgfkeysvalueof{/tikz/quiver/##1}}},
	quiver/.cd,pos/.initial=0.35,height/.initial=0}

\tikzset{tail reversed/.code={\pgfsetarrowsstart{tikzcd to}}}
\tikzset{2tail/.code={\pgfsetarrowsstart{Implies[reversed]}}}
\tikzset{2tail reversed/.code={\pgfsetarrowsstart{Implies}}}
\tikzset{no body/.style={/tikz/dash pattern=on 0 off 1mm}}

\newcommand{\ZZ}{\mathbb{Z}}

\newcommand{\CC}{\mathbb{C}}
\newcommand{\RR}{\mathbb{R}}
\newcommand{\PP}{\mathbb{P}}

\title{Enumerative aspect of symplectic log Calabi-Yau divisors and almost toric fibrations}
\date{\today}
\author{Tian-Jun Li, Jie Min, Shengzhen Ning}

\AtEndDocument{\bigskip{\footnotesize
		\textsc{School of Mathematics, University of Minnesota, Minneapolis, MN, US} \par
		\textit{E-mail address}: \texttt{tjli@math.umn.edu} \par
		\addvspace{\medskipamount}
		\textsc{Max Planck Institute for Mathematics, Bonn, Germany} \par
		\textit{E-mail address}: \texttt{minxx127@umn.edu} \par
		\addvspace{\medskipamount}
		\textsc{School of Mathematics, University of Minnesota, Minneapolis, MN, US} \par
		\textit{E-mail address}: \texttt{ning0040@umn.edu} \par
		
}}

\begin{document}
	\maketitle
	\begin{abstract}
		In this paper we are interested in the isotopy classes of symplectic log Calabi-Yau divisors in a fixed symplectic rational surface. We give several equivalent definitions and prove the stability, finiteness and rigidity results. Motivated by the problem of counting toric actions, we obtain a general counting formula of symplectic log Calabi-Yau divisors in a restrictive region of $c_1$-nef cone. A detailed count in the case of $ 2 $- and $ 3 $-point blow-ups of complex projective space for all symplectic forms is also given. In our framework the complexity of the combinatorics of analyzing Delzant polygons is reduced to the arrangement of homology classes. Then we study its relation with almost toric fibrations. We raise the problem of realizing all symplectic log Calabi-Yau divisors by some almost toric fibrations and verify it together with another conjecture of Symington in a special region.
	\end{abstract}
	
	\tableofcontents
	
	\section{Introduction}
	
	Let X be a smooth rational surface and let $ D\subset X  $ be an effective reduced anti-canonical divisor. Such pairs $ (X,D) $ are called anti-canonical pairs or log Calabi-Yau surfaces and has been extensively studied since Looijenga (\cite{Lo81}). The mirror symmetry aspects of anti-canonical pairs were first studied by Gross, Hacking and Keel in \cite{GrHaKe11} and \cite{GrHaKe12}. Then it was studied by many people in for example \cite{Pas14-thesis}, \cite{Keating2018} and \cite{HaKe20-HMS}. The moduli space of anti-canonical pairs have been studied in \cite{Lo81} and \cite{GrHaKe12}, where Torelli type results were proved. See \cite{Fr} for an excellent survey.
	As an symplectic analogue of anti-canonical pairs, the notion of symplectic log Calabi-Yau pairs was introduced and studied in \cite{LiMa16-deformation}, \cite{LiMi-ICCM} and \cite{LiMaMi-logCYcontact}, with applications to contact structures and symplectic fillings.
	
	Instead of considering the pair of a symplectic manifold and a divisor, in this paper we are interested in the space of isotopy classes of symplectic log Calabi-Yau divisors in a fixed symplectic rational surface. We will give several different but equivalent definitions, prove stability, rigidity and finiteness properties. Making use of those results we count the number of elements in it and investigate its relation to toric actions and almost toric fibrations. Note that in this paper although the word ``space" is frequently used, we will only study its properties as a set. In other word, no additionally topological or algebraic structures are added. Hopefully this will not cause confusion. 
	
	\subsection{Space of symplectic log Calabi-Yau divisors}\label{section:intro moduli}
	
	A symplectic divisor refers to a connected configuration of finitely many closed embedded symplectic surfaces $ D=\cup C_i $ in a symplectic 4-manifold such that all intersections are positively transversal and there are no triple intersections. For a more detailed discussion on symplectic divisors, see \cite{LiMa14-divisorcap} or \cite{LiMi-ICCM}.
	
	A \textbf{symplectic log Calabi-Yau pair} $(X,D,\omega)$ is a closed symplectic 4-manifold $(X,\omega)$ together with a nonempty symplectic divisor $D=\cup C_i$ representing the  Poincare dual of $c_1(X,\omega )$.
	It's an easy observation (\cite{LiMa19-survey}) that $ D $ is either a torus or a cycle of spheres. In the former case, $ (X,D,\omega ) $ is called an \textbf{elliptic log Calabi-Yau pair}. In the later case, it's called a \textbf{symplectic Looijenga pair} and can only happen when $ (X,\omega ) $ is rational. As a consequence, we have $ b^+(Q_D)=0 $ or $ 1 $, where $Q_D$ is the intersection matrix of $D$. The symplectic divisor $ D $ is then called a {\bf symplectic log Calabi-Yau divisor} in $ (X,\omega) $.
	
	Now fixing a symplectic rational surface $ (X,\omega) $, we denote by $ p\mathcal{LCY}(X,\omega) $ the set of symplectic log Calabi-Yau divisors in $ (X,\omega) $, where the prefix `$p$' stands for `pre'. The symplectomorphism group $ Symp(X,\omega) $ naturally acts on $ p\mathcal{LCY}(X,\omega) $. The following space we mainly study in this paper can be viewed as the set of path components of $ p\mathcal{LCY}(X,\omega) $ modulo the action of $ Symp(X,\omega) $.
	
	\begin{definition}
		The {\bf isotopy class space} of symplectic log Calabi-Yau divisors is
		\[
		\mc{LCY}(X,\omega)= (p\mathcal{LCY}(X,\omega)/\Symp(X,\omega))/\sim,
		\]
		where $ [D]\sim [D'] $ with $[D],[D']\in p\mathcal{LCY}(X,\omega)/\Symp(X,\omega)$ if there exist representatives $D,D'\in p\mathcal{LCY}(X,\omega)$ and an symplectic isotopy $ D_t $ in $ (X,\omega) $ with $ D_0=D $ and $ D_1=D' $.
	\end{definition}
	
	Denote by $ p\mathcal{LCY}^\perp(X,\omega) $ the subset of $ \omega  $-orthogonal divisors in $ p\mathcal{LCY}(X,\omega) $, that is, the divisors whose irreducible components all intersect $ \omega  $-orthogonally.
	\begin{thm}\label{thm:all equivalent}
		The following sets are different but equivalent ways to describe the isotopy class space of symplectic log Calabi-Yau divisors in $ (X,\omega ) $.
		\begin{align*}
		\mc{LCY}(X,\omega) = &p\mathcal{LCY}(X,\omega)/\text{strict symplectic deformation equivalence} \\
		\cong &p\mathcal{LCY}(X,\omega)/ \text{strict homological equivalence}\\
		\cong &p\mathcal{LCY}(X,\omega)/ \text{strict $ H^2 $-automorphism}\\
		\cong & p\mathcal{LCY}^\perp(X,\omega)/ \text{symplectomorphism}
		\end{align*}
	\end{thm}
	Here the notions of strict homological equivalence and strict symplectic deformation equivalence were defined in \cite{LiMa16-deformation} and will be recalled in Definition \ref{def:deformation}. By a strict $ H^2 $-automorphism we mean an integral isometry of the lattice $ H^2(X;\ZZ) $ preserving the Poincare dual of the homology classes of components of $ D $ and the symplectic class $ [\omega ] $ (see Proposition \ref{prop:torelli}).  And by symplectomorphism we mean the restriction of the action of $Symp(X,\omega)$ on the subset $p\mathcal{LCY}^\perp(X,\omega)\subset p\mathcal{LCY}(X,\omega)$.
	
	Theorem \ref{thm:all equivalent} is a meta-theorem, made up of three equivalences. The first equivalence is the main result of \cite{LiMa16-deformation} (see also Theorem \ref{thm: symplectic deformation class=homology classes}). The second equivalence is proved in Proposition \ref{prop:torelli}, which can be viewed as an improvement of the first equivalence. The last equivalence is established in Proposition \ref{prop:def to symp}, where we prove a version of symplectic isotopy extension theorem for symplectic divisors, adapting an argument of Siebert and Tian (\cite{SibTian05}). This theorem can be seen as a symplectic version of the Torelli theorem in \cite{Fr}.

	For two symplectic log Calabi-Yau pairs $(X,\omega_1,D_1)$ and $(X,\omega_2,D_2)$ with the same underlying manifold but different symplectic forms, there is also a notion of (non-strict) symplectic deformation equivalence introduced in \cite{LiMa16-deformation}, see also Definition \ref{def:deformation}. In this paper we also introduce the following space of (non-strict) deformation classes associated to a manifold $X$.
	\begin{definition}
		For a closed smooth 4-manifold $X$, the {\bf deformation class space} of symplectic log Calabi-Yau divisors is
		\[ \widetilde{\mc{LCY}}(X)= \big{(}\bigsqcup\limits_{\omega \text{ symplectic forms on }X } p\mathcal{LCY}(X,\omega) \big{)} / \sim,  \] where $ \sim $ is the symplectic deformation equivalence. 
	\end{definition}
	
	Note that if $D_1,D_2\in p\mathcal{LCY}(X,\omega)$ are strict deformation equivalent, then the pairs $(X,\omega,D_1)$ and $(X,\omega,D_2)$ are deformation equivalent. Therefore there is a natural map (not always injective) $\mathcal{LCY}(X,\omega)\rightarrow \widetilde{\mathcal{LCY}}(X)$ for each symplectic form $\omega$.
	
	Recall that a toric action on a symplectic 4-manifold $ (X,\omega) $ is a group homomorphism $ \rho:T\to Ham(X,\omega) $, where $ T\cong (S^1)^2 $ is a 2-torus. We define the space of toric actions  $$ \mc{T}(X,\omega)=\{ \rho:T^2\to Ham(X,\omega ) \} / \sim^{t}, $$ where $ \sim^t $ is the equivalence of toric actions on $ (X,\omega ) $ (Definition \ref{def:toric equivalence}).
	For any toric action on a closed symplectic 4-manifold, we obtain a natural symplectic log Calabi-Yau divisor called the boundary divisor, as  preimage of the boundary of moment polygon.
	To further inverstigate this relation between symplectic log Calabi-Yau divisors and toric actions on symplectic rational surfaces, we introduce the notion of toric symplectic log Calabi-Yau divisors. For this purpose, given $D\in p\mathcal{LCY}(X,\omega)$ we denote the length of $D$ by $k(D)$ and define $q(D)=12-k(D)-D^2$ (this is called the ``charge" in \cite{Fr}).
	\begin{definition} \label{def: toric lcy}
		$ D\in p\mathcal{LCY}(X,\omega) $ is called a {\bf toric symplectic log Calabi-Yau divisor} if $q(D)=0$. Denote by $$ t\mc{LCY}(X,\omega)=\{ D\in p\mathcal{LCY}(X,\omega)\,|\, q(D)=0 \}/\sim\,\,\subset \mathcal{LCY}(X,\omega) $$ the isotopy class space of toric symplectic log Calabi-Yau divisors, where $\sim$ is the strict symplectic deformation equivalence. Moreover, let $\widetilde{t\mc{LCY}}(X)\subset \widetilde{\mc{LCY}}(X) $ be the subset consisting only deformation classes of toric symplectic log Calabi-Yau divisors.
	\end{definition}
	
	The boundary divisors of a toric action is easily seen to be a toric symplectic log Calabi-Yau divisor. And in fact, the converse is also true up to strict symplectic deformation equivalence. In Section \ref{section:toric action}, we prove the following equivalence.
	\begin{thm}\label{thm:toric=tLCY}
		Let $ (X,\omega ) $ be a symplectic rational surface. The natural map \[
		\mc{T}(X,\omega)\to t\mc{LCY}(X,\omega )
		\]
		taking a toric action to its boundary divisor is a bijection.
		
	\end{thm}

	\subsection{Stability and rigidity}
	To give a better description of the isotopy class space $\mathcal{LCY}(X,\omega)$, we also introduce a more combinatorial space $\mathcal{HLCY}(X,\omega)$ consisting of {\bf homological log Calabi-Yau divisors} (see Definition \ref{def:hlcy}). There is a natural map 
	\[\mathcal{LCY}(X,\omega)\rightarrow\mathcal{HLCY}(X,\omega)\]
	by sending a geometric divisor to its underlying homological configuration. In Proposition \ref{prop:realization} we prove this map is bijective. So we can actually identify $\mathcal{LCY}(X,\omega)$ and $\mathcal{HLCY}(X,\omega)$. It turns out working over $\mathcal{HLCY}(X,\omega)$ will simpify our later discussions on finiteness, counting and almost toric fibrations. 
	
	As an immediate consequence, we obeserve the following stability phenomenon. For a symplectic 4-manifold $(X,\omega)$, denote by $ \mc{S}_\omega  $ the set of elements of $H_2(X;\ZZ)$ that can be represented by $\omega $-symplectic spheres. It turns out the isotopy class space of symplectic log Calabi-Yau divisors is determined by $ \omega  $ through $ \mc{S}_\omega  $.
	\begin{thm}\label{thm:stability}
		Let $ X $ be a rational surface. Suppose $ \omega  $ and $ \omega' $ are two symplectic forms with $ c_1(X,\omega)\cdot [\omega]>0 $, $ c_1(X,\omega')\cdot [\omega']>0 $ and $ \mc{S}_{\omega}=\mc{S}_{\omega'} $. Then \[
		\mc{LCY}(X,\omega )=\mc{LCY}(X,\omega').
		\]
	\end{thm}

	
	
	For each symplectic log Calabi-Yau divisor $ D\subset (X,\omega ) $ of length $ k $, we can associate two invariants of $ D $, its self-intersection sequence $ s(D)=([C_i]^2)_{i=1}^k\in \ZZ^k $ and its symplectic area sequence $ a(D)=(\omega([C_i]))_{i=1}^k\in (\RR^+)^k $.
	In the holomorphic category, a labeled toric anticanonical pair (defined in \cite{Fr}) is called {\bf taut} if its isomorphism class is determined by its self-intersection sequence.
	Similarly we define a labeled symplectic log Calabi-Yau divisor to be {\bf def-taut} (resp. {\bf iso-taut}) if its symplectic deformation class (resp. isotopy class) is determined by $ s(D) $ (resp. $ (s(D),a(D)) $). Such notion can be seen as a type of rigidity for symplectic log Calabi-Yau divisors.
	
	
	
	We show in Lemma \ref{lemma:taut} that toric symplectic log Calabi-Yau divisors are both def-taut and iso-taut, which will be used to prove Proposition \ref{prop:CY to toric}. However general symplectic log Calabi-Yau divisors might lose the tautness. In the Appendix \ref{section:appendixtaut} we give a criterion of def-tautness for 
	several families of $b^+=1$ divisors discussed in \cite{LiMaMi-logCYcontact}.
	These resutls are not used in the current paper but are related to the problems of classifying symplectic fillings of torus bundles. We hope there will be applications in the future.
	
	
	

	\subsection{Finiteness and counting}
	
	Karshon, Kessler and Pinsonnault proved the finiteness of toric actions in \cite{KaKePi07-finite} and gave an upper bound for the number of toric actions in \cite{KaKePi14-count} and \cite{KKP}. Inspired by their results and the identification of $ \mc{T}(X,\omega) $ with $ t\mc{LCY}(X,\omega) $ discussed above, we study the enumerative aspect of symplectic log Calabi-Yau divisors. The strategy of \cite{KaKePi14-count} and \cite{KKP} is to understand the combinatorics of Delzant polygons. In our framework, the combinatorics is reduced to the arrangement of homology classes by investigating the space $\mathcal{HLCY}(X,\omega)$.
	
	First we observe that the isotopy class space of symplectic log Calabi-Yau divisors \[
	|\mc{LCY}(X,\omega)|<\infty
	\] is finite for any $(X, \omega)$ in Corollary \ref{lemma:finiteness}. So it makes sense to count the number $ |\mc{LCY}(X,\omega)| $.
	
	Since we are only interested in symplectic log Calabi-Yau divisors up to strict homological equivalence, it suffices to count in the normalized reduced symplectic cone, which is the fundamental domain inside the symplectic cone under the action of orientation-preserving diffeomorphisms. A reduced symplectic class $ [\omega ] $ on $ M_l=\CC\PP^2\# l\overline{\CC\PP}^2 $ is determined by a reduced vector $ (\delta_1,\dots,\delta_l) $ and write $ \mc{LCY}(l;\delta_1,\dots,\delta_l):=\mc{LCY}(M_l,\omega) $. Note that the existence of symplectic log Calabi-Yau divisors implies $c_1\cdot\omega>0$. So we may focus on a subregion satisfying $\delta_1+\cdots+\delta_l<3$, which is called the $c_1$-nef cone. (see Section \ref{section:c1-nef}).
	
	By enumerating all homological log Calabi-Yau divisors, we get a detailed count for $ M_2=\CC\PP^2\# 2\overline{\CC\PP}^2 $ with every reduced symplectic form in Proposition \ref{prop:general count M2} and Corollary \ref{cor:toric count M2}. For general symplectic log Calabi-Yau divisors, we have
	\[
	|\mc{LCY}(2;\delta_1,\delta_2)|=7(\lceil\dfrac{\delta_1}{1-\delta_1}\rceil +\lceil \dfrac{\delta_1-\delta_2}{1-\delta_1}\rceil )+12,
	\]
	except $ |\mc{LCY}(2;\delta_1,\delta_2)|=13 $ when $ \delta_1=\delta_2<\dfrac{1}{2} $. For toric symplectic log Calabi-Yau divisors, we have \[
	|t\mc{LCY}(2;\delta_1,\delta_2)|=\lceil\dfrac{\delta_1}{1-\delta_1}\rceil +\lceil \dfrac{\delta_1-\delta_2}{1-\delta_1}\rceil.
	\]
	
	For the case of $ M_3=\CC\PP^2\# 3\overline{\CC\PP}^2 $, we also give a detailed counting result for toric symplectic log Calabi-Yau divisors in Proposition \ref{prop:M3}.
	
	However in general it's quite difficult to give a detailed count for $M_l$ when $l$ is getting large, even for toric ones. Therefore we explore the region in the $c_1$-nef cone which admits at least one toric action. This is already a very intricate and interesting question. In section \ref{section:toric cone} we give the complete answer of the toric regions for $l\leq 5$ and some geometric descriptions of the toric regions for general $l$.
	
	In section \ref{section:general counting}, we introduce the notion of {\bf restrictive} region, which is a smaller region than the toric region. It turns out that in this region, the counting problem is simplified dramatically and indeed we are able to give a general counting formula (Proposition \ref{prop:general count}), though quite involved, in this special region. As a by-product we recover the upper bound in \cite{KaKePi14-count}.

	\subsection{Relation with almost toric fibrations}
	An almost toric fibration of a symplectic 4-manifold $ (X,\omega) $ is a Lagrangian fibration $ \pi:X \to B $ with only nodal and elliptic singularities, which can be viewed as a generalization of toric actions. 
	Almost toric fibrations play an important role in the study of exotic Lagrangians (\cite{Vianna-CP2}, \cite{Vianna-CP2-infinite}) and symplectic embeddings (\cite{CasVia-embedding}, \cite{CGHMP-toric-staircase}).
	Similar to the case of toric fibrations, the preimage $\pi^{-1}(\partial B)$ of the base is a symplectic divisor representing the Poincare dual of $ c_1(X,\omega) $ (Proposition 8.2 of \cite{Sym02}), i.e. a symplectic log Calabi-Yau divisor. Moreover, the operations (nodal trade, almost toric blow up, toric blow up) on almost toric fibrations have their analogues (smoothing, non-toric blow up, toric blow up) on symplectic log Calabi-Yau divisors. Therefore we naturally want to understand their relations. Let $\mathcal{ATF}(X,\omega)$ denote all the almost toric fibrations on $(X,\omega)$. The first issue is the following realization problem:
	
	\begin{conjecture}[Realization]\label{conj:atfrealization}
		The map $ \Phi:\mc{ATF}(X,\omega)\to \mc{LCY}(X,\omega) $ is surjective.
	\end{conjecture}
	
	Unlike the space of toric actions $\mathcal{T}(X,\omega)$ which contains the equivalent classes, the space $\mathcal{ATF}(X,\omega)$ here denotes all the almost toric fibrations. Thus the next step is to find a suitable equivalence relation $\sim^{atf}$ such that a bijection between $\mathcal{ATF}(X,\omega)/\sim^{atf}$ and $\mathcal{LCY}(X,\omega)$ can be established.
	
	Note that an affirmative answer to Conjecture \ref{conj:atfrealization} will imply every symplectic rational manifold with $\omega\cdot c_1>0$ admits almost toric fibrations, which is also a question we don't know the answer to. Based on the detailed counting results for small rational manifolds in the previous sections, we are able to prove Conjecture \ref{conj:atfrealization} holds for 
	$ M=S^2\times S^2 $ and $ \CP^2\# l\overline{\CP^2} $ with $ l\le 3 $
	with arbitrary symplectic forms (Proposition \ref{prop:surjectivity conjecture minimal} and \ref{prop:surjective M3}).
	\\
	
	In \cite{Sym02}, Symington observed that different toric fibrations on $ (S^2\times S^2,\omega) $ can be changed to each other via nodal trades, nodal slides and mutations, which gives a path of almost toric fibrations connecting toric fibrations. This observation led to the following conjecture.
	\begin{conjecture}[Connectedness]
		Any two toric fibrations of $ (M,\omega ) $ can be connected by a path of almost toric fibrations of $ (M,\omega ) $.
	\end{conjecture}

	Based on Symington's observation and the detailed counting results for toric divisors in small rational manifolds, we can prove it for $ M=S^2\times S^2 $ and $ \CP^2\# l\CP^2 $ with $ l\le 3 $
	(Proposition \ref{prop:surjectivity conjecture minimal}).
	
	It's hard to generalize the approaches of proving these two conjectures for small rational manifolds to the cases for arbitrary $(X,\omega)$ due to the difficulty of enumerating $\mathcal{LCY}(X,\omega)$ and $t\mathcal{LCY}(X,\omega)$. However, if $[\omega]$ is restrictive in the sense of section \ref{section:general counting}, we are fortune enough to verify these two conjectures. In other words, besides the cases for small rational manifolds ($M_l$ for $l\leq3$), there exists an open region inside the $c_1$-nef cone supporting these two conjectures. The proofs are shown in Section \ref{section:verify}.
	
	\textbf{Acknowledgments}:
	We are grateful to Cheuk Yu Mak, Liya Ouyang and Shuo Zhang for helpful conversations. We also thank Yael Karshon for sharing with us a draft of their unpublished paper.
	All authors are supported by NSF grant DMS 1611680.

	\section{Isotopy class space of symplectic log Calabi-Yau divisors}\label{section:properties}

	\subsection{Symplectomorphism and symplectic Torelli theorem}
	The goal of this section is to prove Theorem \ref{thm:all equivalent}. It will be a combination of Theorem \ref{thm: symplectic deformation class=homology classes}, Proposition \ref{prop:def to symp} and Proposition \ref{prop:torelli}. We begin with some notions and results in \cite{LiMa16-deformation}.
	\begin{definition}\label{def:deformation}
		Let $ X $ be a closed oriented 4-manifold. Symplectic divisors $ (D^0,\omega^0) $ and $ (D^1,\omega^1) $ in $ X $ with the same number of components, are {\bf symplectic deformation equivalent} if they are connected by a family of symplectic divisors $ (D^t,\omega^t) $, up to an orientation-preserving diffeomorphism. They are called {\bf homologically equivalent} if there is an orientation-preserving diffeomorphism $ \Phi:X\to X $ such that $ \Phi_*[C_j^0]=[C_j^1] $ for all $ j=1,\dots,k $.
		
		Fix a closed symplectic 4-manifold $ (X,\omega ) $. Symplectic divisors $ D^0 $ and $ D^1 $ in $ (X,\omega ) $ are {\bf strictly symplectic deformation equivalent} if they are connected by a family of symplectic divisor $ D^t $, up to a symplectomorphism of $ (X,\omega ) $. They are {\bf strictly homologically equivalent}  if there is a homological equivalence fixing $ [\omega] $.
		
		%
		%
	\end{definition}
	
	The main result of \cite{LiMa16-deformation} states that homological equivalence and symplectic deformation equivalence introduced above are equivalent for symplectic log Calabi-Yau divisors.
	
	\begin{theorem}[\cite{LiMa16-deformation},\cite{LiMa19-survey}] \label{thm: symplectic deformation class=homology classes}
		Two symplectic log Calabi-Yau divisors $ (D^0,\omega^0) $ and $ (D^1,\omega^1) $ in $ X $ are symplectic deformation equivalent if and only if they are   homologically equivalent. In particular, each symplectic deformation class contains a Kahler pair. Moreover,  two symplectic log Calabi-Yau divisors $ D^0 $ and $ D^1 $ in $ (X,\omega) $ are strictly symplectic deformation equivalent if and only if they are strictly homologically equivalent.
		
	\end{theorem}
	By the above theorem, we have the following equivalent definition for the isotopy class space of symplectic log Calabi-Yau divisors \begin{align*}
	\mc{LCY}(X,\omega)=&p\mathcal{LCY}(X,\omega)/\text{strict symplectic deformation equivalence}\\
	=& p\mathcal{LCY}(X,\omega)/\text{strict homological equivalence}
	\end{align*}
	
	\begin{rmk}
		Definition \ref{def:deformation} and Theorem \ref{thm: symplectic deformation class=homology classes} were originally stated for symplectic log Calabi-Yau pairs $ (X,\omega,D) $ in \cite{LiMa16-deformation}. The original statements restrict to the current ones by fixing $ X $ or $ (X,\omega ) $.
	\end{rmk}
	
	For symplectic orthogonal divisors, the equivalences introduced in the previous subsection can be shown to be equivalent to symplectomorphism. This is a singular analogue of the following result of Siebert-Tian, saying that in the case of a single symplectic surface, a symplectic isotopy of the surface extends to a Hamiltonian isotopy of the ambient symplectic 4-manifold.
	\begin{prop}[Proposition 0.3 of \cite{SibTian05}]
		A smooth isotopy $ \Sigma_t\subset (X^4,\omega ) $ of a symplectic surface can be generated by a Hamiltonian isotopy, i.e. $ \Sigma_t=\phi^t_{H_t}(\Sigma_0) $. Given an open subset $ U\subset X^4 $ in which $ \Sigma_t\cap U=\Sigma_0\cap U $ is fixed, we may moreover assume that $ H_t|_V=0 $ holds in any relatively compact subset $ V $ in $ U $.
	\end{prop}
	It is observed in \cite{RiGoIv16} that the argument of Siebert-Tian works for symplectic divisors when the symplectic isotopy is constant near the intersection points. Such an isotopy is an instance of a strong $ D $-symplectic isotopy of marked symplectic divisors introduced in \cite{LiMa16-deformation}.
	
	\begin{definition}
		A {\bf marked symplectic divisor} consists of a 5-tuple\[
		\Theta=(X,\omega,D,\{p_j\},\{I_j\})
		\]
		such that \begin{itemize}
			\item $ D\subset (X,\omega ) $ is an $ \omega  $-orthogonal divisor,
			\item $ p_j $, called centers of marking, are points on $ D $ (intersection points of $ D $ allowed)
			\item $ I_j:(B(\delta_j),\omega_{std})\to (X,\omega ) $, called coordinates of marking, are symplectic embeddings sending the origin to $ p_j $ and with $ I_j^{-1}(D)=\{ z=0 \} \cap B(\delta_j) $ (resp. $ I_j^{-1}(D)=(\{ z=0 \}\cup \{ w=0 \})\cap B(\delta_j) $) if $ p_j $ is a smooth (resp. an intersection) point of $ D $, where $ (z,w) $ is the complex coordinate for $ B(\delta_j) $. Moreover, we require that the image of $ I_j $ are disjoint.
		\end{itemize}
	\end{definition}
	Every $ \omega $-orthogonal symplectic divisor can be regarded as a marked divisor where $ \{ p_j \} $ consists of all intersection points and $ I_j $ exists since $ D $ is $ \omega  $-orthogonal. In the rest of this subsection, we will always take an $ \omega $-orthogonal symplectic divisor as a marked divisor in such way.
	
	\begin{definition}
		Let $ \Theta=(X,\omega,D,\{p_j\},\{I_j\}) $ be a marked divisor. A {\bf $ D $-symplectic isotopy} of $ \Theta $ is a 4-tuple $ (X,\omega_t,D,\{ p_j \}) $ such that $ \omega_t  $ is a smooth family of cohomologous symlectic forms on $ X $ with $ \omega_0=\omega  $ and $ D $ being $ \omega_t  $-symplectic for all $ t $.
		
		A {\bf strong $ D $-symplectic isotopy} is a 5-tuple $ (X,\omega_t,D,\{p_j\},\{I_{j,t}\}) $ such that \begin{itemize}
			\item the 4-tuple $ (X,\omega_t,D,p_j) $ is a $ D $-symplectic isotopy of $ \Theta  $,
			\item $ D $ is $ \omega_t  $-orthogonal,
			\item $ I_{j,t}:(B(\epsilon_j),\omega_{std})\to (X,\omega_t) $ are symplectic embeddings sending the origin to $ p_j $, $ I_{j,0}=I_j $ and $ I_{j,t}^{-1}(D)=\{ z=0 \}\cap B(\epsilon_j) $ (resp. $ I_{j,t}^{-1}(D)=(\{z=0 \}\cup \{ w=0 \})\cap B(\epsilon_j) $) if $ p_j $ is a smooth point (resp. an intersection point) of $ D $, for some $ \epsilon_j<\delta_j $, where $ (z,w) $ is the complex coordinate for $ B(\delta_j) $.
		\end{itemize}
	\end{definition}
	
	\begin{lemma}[\cite{LiMa16-deformation}]\label{lemma:isotopy to strong isotopy}
		If $ (X,\omega,D,\{p_j\},\{I_j\}) $ and $ (X',\omega',D',\{p_j'\},\{I_j'\}) $ are $ D $-symplectic isotopy, then they are strong $ D $-symplectic isotopy.
	\end{lemma}
	In order to fix a neighborhood of the intersections points during the isotopy, we also need the following result which extends a path of symplectic ball embeddings to a Hamiltonian isotopy.
	\begin{lemma}[Theorem 3.3.1 of \cite{McSa17-baby}]\label{lemma:ball isotopy}
		Let $ (X,\omega ) $ be a closed symplectic manifold. Let $ 0\le \lambda<r $ and let $ \psi_0,\psi_1:B(r)\to X $ be symplectic embeddings that are joined by a smooth family of symplectic embeddings $ \psi_t:B(\lambda )\to X $. Then there exists a Hamiltonian isotopy $ \{\phi_t \} $ of $ X $ and a constant $ \lambda<\rho<r $ such that $ \phi_0=id $, $ \phi_1\circ \psi_0|_{B(\rho)}=\psi_1|_{B(\rho)} $ and $ \phi_t\circ \psi_0|_B(\lambda)=\psi_t $.
	\end{lemma}
	Combining the above two lemmas, we can modify the isotopy to the one considered in \cite{RiGoIv16} and apply the argument of \cite{SibTian05}.
	\begin{lemma}\label{lemma:strong isotopy to Ham}
		Two strong $ D $-symplectic isotopic $ \omega $-orthogonal symplectic divisors are symplectomorphic. In particular, they are Hamiltonian diffeomorphic if $ H^1(X;\RR)=0 $.
	\end{lemma}
	\begin{proof}
		Let $ (X,\omega_t,D,\{p_j\},\{I_{j,t} \}) $ be a strong $ D $-symplectic isotopy. By Moser stability, there is an isotopy $ \psi_t:X\to X $ such that $ \psi_0 =id$ and $ \psi_t^*\omega_t=\omega_0 $. Let $ D_t=\psi_t^{-1}(D) $ be the image of $ D $ in $ (X,\omega_0) $. Denote by $ \tilde{I}_{j,t}=\psi_t^{-1}\circ I_{j,t}:B(\epsilon_j)\to (X,\omega_0) $ the family of symplectic ball embeddings such that $ (\tilde{I}_{j,t})^{-1}(D_t)=(\{ z=0 \}\cup \{w=0\})\cap B(\epsilon_j) $. Note that $ \psi_t  $ is actually Hamiltonian if $ H^1(X;\RR)=0 $. So it suffices to show $ (X,\omega,D_0) $ and $ (X,\omega,D_1) $ are Hamiltonian diffeomorphic.
		
		By Lemma \ref{lemma:ball isotopy}, there is a Hamiltonian isotopy $ \phi_t $ of $ (X,\omega_0) $ such that $ \phi_t\circ \tilde{I}_{j,0}=\tilde{I}_{j,t} $, after possibly shrinking to a smaller $ \epsilon_j $. Replacing $ D_t $ by $ \phi_t^{-1}(D_t) $, we may assume that $ D_t\cap \tilde{I}_{j,0}(B(\epsilon_j)) $ is fixed during the isotopy.
		At a smooth point $ p $ on $ D_{t_0} $, choose a complex Darboux coordinate $ (z=x+iy,w=u+iv) $. For $ t $ close to $ t_0 $, there are functions $ f_t,g_t $ such that $ D_t $ is the graph $ w=f_t(z)+ig_t(z) $. Then following the argument of Siebert-Tian, we define \[
		H_t=-(\partial_t g_t)\cdot (u-f_t)+(\partial_t f_t)\cdot (v-g_t),
		\]
		which generates the isotopy.
		At each intersection point $ p_j\in D_{t_0} $, take the Darboux ball $ \tilde{I}_{j,0}:B(\epsilon_j)\to (X,\omega_0) $ and define $ H_t=0 $ in this chart, which generates the constant isotopy.
		
		Since each $ H_t $ vanishes along $ D_{t_0} $, locally defined $ H_t $ can be patched via a partition of unity of $ D_{t_0} $ to a function defined near $ D_{t_0} $ and then extend to $ M $ arbitrarily. Finally extend the construction to all $ t\in [0,1] $ via a partition of unity in $ t $.
		%
	\end{proof}
	
	\begin{prop}\label{prop:def to symp}
		Two $ \omega $-orthogonal symplectic divisors $ D $ and $ D' $ in $ (X,\omega) $ are strictly symplectic deformation equivalent if and only if they are symplectomorphic. Moreover, two general symplectic divisors are strictly symplectic deformation equivalent if and only if they are symplectomorphic after a small symplectic isotopy locally supported near the intersection points.
	\end{prop}
	\begin{proof}
		Up to symplectomorphism, there is a symplectic isotopy $ D_t $ with $ D_0=D $ and $ D_1=D' $.
		Regard $ D $ and $ D' $ as marked symplectic divisors with markings centered at the intersection points. Extend the isotopy $ D_t $ smoothly to an ambient isotopy $ \psi_t:X\to X $ so that $ \psi_t^{-1}(D_t)=D $ and let $ \omega_t=\psi_t^*\omega  $. Then $ (X,\omega_t,D,\{ p_j \}) $ is a $ D $-symplectic isotopy. By Lemma \ref{lemma:isotopy to strong isotopy}, $ (X,\omega,D) $ and $ (X,\omega_1,D) $ are strong $ D $-symplectic isotopic and thus symplectomorphic by Lemma \ref{lemma:strong isotopy to Ham}. So $ D $ is symplectomorphic to $ D' $ by composing with $ \psi_1 $.
		
		Two general symplectic divisors can be perturbed to be symplectic orthogonal by a symplectic isotopy locally supported near the intersection points (Lemma 2.3 of \cite{Gom95-fibersum}). After the local perturbation, they are still strictly symplectic deformation equivalent and thus symplectomorphic.
	\end{proof}

	In the holomorphic cateogory, we have the a version of Torelli theorem which basically says two log Calabi-Yau pairs are isomorphic if there is an integral isometry of the second cohomology, which maps the homology of one divisor to the other (Theorem 8.5 of \cite{Fr}). Here an integral isometry means an automorphism preserving the intersection form. In this subsection, we give a symplectic analogue of this Torelli theorem.

	On a closed 4-manifold $ X $, each diffeomorphism induces an automorphism of the lattice of the second integral cohomology, which gives a natural map $ \Diff(X)\to \aut (H^2(X;\ZZ)) $. We denote the image of this map by $ D(X) $. This map is not always surjective, i.e. not every integral isometry of the second cohomology is generated by a diffeomorphism. For rational surfaces with Euler characteristic $ \chi\le 12 $, every integral isometry is induced from a diffeomorphism (\cite{Wall64-diffeo}). But this is not true when $ \chi>12 $ by \cite{FrMo88-diffeo}.
	
	Theorem \ref{thm: symplectic deformation class=homology classes} can be regarded as a symplectic analogue of the Torelli theorem, but only in the weak sense. Firstly, it only concerns symplectic deformation equivalence as compared to isomorphism in the holomorphic case. Secondly, as discussed above, not every integral isometry of the second cohomology is generated by a diffeomorphism. However, we observed that Theorem \ref{thm: symplectic deformation class=homology classes} can be improved to the following much stronger analogue of the Torelli theorem using Proposition \ref{prop:def to symp} and a characterization of $ D(X) $ in \cite{LiLi-genus}.
	
	\begin{prop}[Symplectic Torelli Theorem]\label{prop:torelli}
		Let $ D $ and $ D' $ be  two symplectic log Calabi-Yau divisors in $ (X,\omega) $ with component $C_i$ and $C_i'$ respectively, and an integral isometry \[
		\gamma:H^2(X;\ZZ)\to H^2(X;\ZZ)
		\]
		such that $ \gamma(PD[C_i'])=PD[C_i] $, and $ \gamma([\omega])=[\omega] $. Then $ D $ and $ D' $ are strictly symplectic deformation equivalent. They are symplectomorphic if they are both $ \omega $-orthogonal.
	\end{prop}
	\begin{proof}
		Notice that $ [\omega]\cdot c_1(X,\omega)>0 $ by the existence of symplectic log Calabi-Yau divisors.
		Since $ \gamma([\omega])=[\omega] $ and $ \gamma(c_1(X,\omega))=c_1(X,\omega) $, we have $ \gamma\in D(X) $ by Theorem 2.8 (1) of \cite{LiLi-genus}, i.e. there is a diffeomorphism $ \phi:X\to X $ such that $ \phi^*=\gamma  $. Since $ \phi^* $ preserves the symplectic class $ [\omega] $, it must be orientation-preserving and therefore a strict homological equivalence.
		Then the two divisors are strictly symplectic deformation equivalent by Theorem \ref{thm: symplectic deformation class=homology classes}. If they are both $ \omega $-orthogonal, then they are actually symplectomorphic by Proposition \ref{prop:def to symp}.
	\end{proof}
	
	An automorphism of $ H^2(X;\ZZ ) $ as in Proposition \ref{prop:torelli} is called a {\bf strict $ H^2 $-automorphism}.
	Denote by $ p\mathcal{LCY}^\perp(X,\omega)\subset p\mathcal{LCY}(X,\omega) $ the subspace of $ \omega $-orthogonal symplectic divisors in $ (X,\omega) $. Then the two main results of this subsection, Proposition \ref{prop:def to symp} and \ref{prop:torelli}, give two more equivalent definitions of the isotopy class space of symplectic log Calabi-Yau divisors \begin{align*}
	\mc{LCY}(X,\omega)&=p\mathcal{LCY}^\perp(X,\omega)/\Symp(X,\omega )\\
	&=p\mathcal{LCY}(X,\omega)/\text{strict $ H^2 $-automorphism}.
	\end{align*}

	\subsection{Operations and minimal models}\label{section:operations}
	
	In this subsection, we review various facts about the operations and minimal models of symplectic log Calabi-Yau pairs studied in \cite{LiMa16-deformation}, which will become useful later.

	For a symplectic divisor $ D\subset (X,\omega) $, a \textbf{toric blow-up} (resp. \textbf{non-toric blow-up}) of $ D $ is the total (resp. proper) transform of a symplectic blow-up centered at an intersection point (resp. a smooth point) of $ D $.
	A \textbf{toric blow-down} refers to blowing down an exceptional sphere contained in $ D $, intersecting exactly once with exactly two other irreducible components of $ D $. A \textbf{non-toric blow-down} refers to blowing down an exceptional sphere not contained in $ D $, intersecting exactly once with exactly one irreducible component of $ D $.
	These operations preserve the symplectic log Calabi-Yau condition and have analogues in the holomorphic category.
	
	\begin{remark}
		Theorem \ref{thm: symplectic deformation class=homology classes} implies the following. \begin{itemize}
			\item Let $ D_1,D_2 $ be two strictly symplectic deformation equivalent symplectic log Calabi-Yau divisors in $ (X,\omega) $ such that a nodal point $ p_1\in D_1 $ corresponds to $ p_2\in D_2 $ and a component $ C_1\subset D_1 $ corresponds to $ C_2\subset D_2 $. For $ i=1,2 $, let $ \tilde{D}_i $ be the non-toric blow-up of $ D_i $, of the same size, at some point on $ C_i $. Then $ \tilde{D}_1 $ and $ \tilde{D}_2 $ are strictly symplectic deformation equivalent. The same holds if $ \tilde{D}_i $ is toric blow-up of $ D_i $, of the same size, at $ p_i $.
			\item Similarly, let $ e_1 $ be a toric/non-toric exceptional class of $ D_1 $ and let $ \overline{D}_1 $ be the blow-down of $ D_1 $. Then there is a corresponding toric/non-toric exceptional class $ e_2 $ of $ D_2 $ under the strict homology equivalence, such that the corresponding blow-down $ \overline{D}_2 $ of $ D_2 $ is strictly symplectic deformation equivalent to $ \overline{D}_1 $.
		\end{itemize}
	\end{remark}
	
	\begin{definition}
		A  symplectic log Calabi-Yau pair  $(X,D,\omega)$ is called
		minimal
		if $(X, \omega)$ is minimal,  or
		$(X, D, \omega)$ is a symplectic Looijenga pair
		with $X=\mathbb{CP}^2 \# \overline{\mathbb{CP}^2}$.
	\end{definition}
	Through a maximal sequence of non-toric blow-downs and then a maximal sequence of toric blow-downs, we get a minimal pair from any symplectic log Calabi-Yau pair $ (X,D,\omega ) $ and is called a minimal model of $ (X,D,\omega ) $.
	
	
	We recall here the homology types of minimal symplectic log Calabi-Yau pairs (modulo cyclic symmetry and homological equivalence). The following result stated in \cite{LiMa16-deformation} (but with no detailed account) is obtained by considering the intersection numbers of the components and adjunction formula. We will give a detailed explanation of this in the proof of Lemma \ref{lemma:minimal surjectivity} in the next section.
	
	\begin{theorem}[\cite{LiMa16-deformation}]\label{thm:minimal model}
		Any minimal symplectic log Calabi-Yau pair $ (X,D,\omega) $ is homological equivalent to one of the following.
		
		$\bullet$  Case $(A)$: $X$ is a symplectic ruled surface with base genus $1$. $D$ is a torus.
		
		$\bullet$  Case $(B)$: $X=\mathbb{CP}^2$,
		$c_1=3h$.

		$(B1)$ $D$ is a torus,
		
		$(B2)$ $D$ consists of a $h-$sphere and a $2h-$sphere, or
		
		$(B3)$ $D$ consists of three $h-$spheres.
		
		$\bullet$ Case $(C)$: $X=S^2 \times S^2$, $c_1=2f_1+2f_2$, where $f_1$ and $f_2$ are the homology classes
		of the two factors.
		
		$(C1)$ $D$ is a torus.
		
		$(C2)$ $k(D)=2$ and  $[C_1]=bf_1+f_2, [C_2]=(2-b)f_1+f_2$.
		
		$(C3)$ $k(D)=3$ and  $[C_1]=bf_1+f_2, [C_2]=f_1, [C_3]=(1-b)f_1+f_2$.
		
		$(C4)$ $k(D)=4$ and $[C_1]=bf_1+f_2,  [C_2]=f_1, [C_3]=-bf_1+f_2, [C_4]=f_1$.
		
		The graphs in (C1), (C2), (C3) and (C4) are given respectively  by
		\[
		\begin{tikzpicture}
		\node (x) at (-0.5,0) [circle,fill,outer sep=5pt, scale=0.5] [label=above:$ 8 $]{};
		
		\node (x1) at (1,0) [circle,fill,outer sep=5pt, scale=0.5] [label=above:$ 2b $]{};
		\node (y1) at (2.5,0) [circle,fill,outer sep=5pt, scale=0.5] [label=above:$ 4-2b $]{};
		\draw (x1) to[bend right] (y1);\draw (y1) to[bend right] (x1);
		
		\node (x2) at (4,0) [circle,fill,outer sep=5pt, scale=0.5] [label=above:$ 2b $]{};
		\node (y2) at (5.5,0) [circle,fill,outer sep=5pt, scale=0.5] [label=above:$ 0 $]{};
		\node (z2) at (4,-1) [circle,fill,outer sep=5pt, scale=0.5] [label=below:$ 2-2b $]{};
		\draw (x2) -- (y2);\draw (y2) -- (z2);\draw (z2) -- (x2);
		
		\node (x3) at (7,0) [circle,fill,outer sep=5pt, scale=0.5] [label=above:$ 2b $]{};
		\node (y3) at (8.5,0) [circle,fill,outer sep=5pt, scale=0.5] [label=above:$ 0 $]{};
		\node (z3) at (8.5,-1) [circle,fill,outer sep=5pt, scale=0.5] [label=below:$ -2b $]{};
		\node (w3) at (7,-1) [circle,fill,outer sep=5pt, scale=0.5] [label=below:$ 0 $]{};
		\draw (x3) -- (y3);\draw (y3) -- (z3);\draw (z3) -- (w3);\draw (w3) to (x3);
		\end{tikzpicture}
		\]	

		$\bullet$ Case $(D)$: $X=\mathbb{C}P^2 \# \overline{\mathbb{C}P^2}$, $c_1=f+2s$, where $f$ and $s$ are the fiber class and section class
		with $f\cdot f=0$, $f\cdot s=1$ and $s\cdot s=1$.

		$(D1)$ $D$ cannot be a torus because it would not be minimal.
		
		$(D2)$ $k(D)=2$,  and either
		$([C_1],[C_2])=(af+s,(1-a)f+s)$ or $([C_1],[C_2])=(2s, f)$.
		
		$(D3)$  $k(D)=3$ and  $[C_1]=af+s, [C_2]=f, [C_3]=-af+s$.
		
		$(D4)$ $k(D)=4$ and  $[C_1]=af+s, [C_2]=f,  [C_3]=-(a+1)f+s, [C_4]=f$.
		
		The graphs in (D2), (D3) and (D4) are given respectively by
		\[
		\begin{tikzpicture}
		\node (x) at (-2,0) [circle,fill,outer sep=5pt, scale=0.5] [label=above:$ 2a+1 $]{};
		\node (y) at (-0.5,0) [circle,fill,outer sep=5pt, scale=0.5] [label=above:$ 3-2a $]{};
		\draw (x) to [bend right] (y);\draw (y) to [bend right] (x);
		
		\node (x1) at (1,0) [circle,fill,outer sep=5pt, scale=0.5] [label=above:$ 4 $]{};
		\node (y1) at (2.5,0) [circle,fill,outer sep=5pt, scale=0.5] [label=above:$ 0 $]{};
		\draw (x1) to[bend right] (y1);\draw (y1) to[bend right] (x1);
		
		\node (x2) at (4,0) [circle,fill,outer sep=5pt, scale=0.5] [label=above:$ 2a+1 $]{};
		\node (y2) at (5.5,0) [circle,fill,outer sep=5pt, scale=0.5] [label=above:$ 0 $]{};
		\node (z2) at (4,-1) [circle,fill,outer sep=5pt, scale=0.5] [label=below:$ 1-2a $]{};
		\draw (x2) -- (y2);\draw (y2) -- (z2);\draw (z2) -- (x2);
		
		\node (x3) at (7,0) [circle,fill,outer sep=5pt, scale=0.5] [label=above:$ 2a+1 $]{};
		\node (y3) at (8.5,0) [circle,fill,outer sep=5pt, scale=0.5] [label=above:$ 0 $]{};
		\node (z3) at (8.5,-1) [circle,fill,outer sep=5pt, scale=0.5] [label=below:$ -2a-1 $]{};
		\node (w3) at (7,-1) [circle,fill,outer sep=5pt, scale=0.5] [label=below:$ 0 $]{};
		\draw (x3) -- (y3);\draw (y3) -- (z3);\draw (z3) -- (w3);\draw (w3) to (x3);
		\end{tikzpicture}
		\]	
		%
		
	\end{theorem}

	We now introduce the \textbf{smoothing} operation, which corresponds to the nodal trade operation of the almost toric fibrations (see Section \ref{section:ATF intro}). Geometrically, given $D=C_1\cup\cdots C_{k}\cup C_{k+1}\subset (X,\omega)$ in the space of symplectic log Calabi-Yau divisors $p\mathcal{LCY}(X,\omega)$, let $p$ be the intersection point of $C_{k}$ and $C_{k+1}$ and $U\subset X$ an arbitrarily small neighborhood of $p$. Then one can always perform the local surgery to get another $D'=C_1\cup\cdots C_{k-1}\cup C_k'\subset (X,\omega)$ in the space $p\mathcal{LCY}(X,\omega)$ such that outside $U$, $C_k'$ coincides with $C_k\cup C_{k+1}$ by smoothing the singular point $p$.
	Homologically, this operation transfers the homological sequence $([C_1],\cdots,[C_k],[C_{k+1}])$ representing $D$ into $([C_1],\cdots,[C_{k-1}],[C_k']=[C_k]+[C_{k+1}])$ representing $D'$. Remember that $\mathcal{LCY}(X,\omega)$ is equal to $p\mathcal{LCY}(X,\omega)$ modulo strict symplectic deformation equivalence and $\widetilde{\mathcal{LCY}}(X)$ is equal to the union of $p\mathcal{LCY}(X,\omega)$ with $\omega$ being all the symplectic forms modulo symplectic deformation equivalence. By passing to the quotient space, we can define two elements in the space $\mathcal{LCY}(X,\omega)$ or $\mathcal{\widetilde{LCY}}(X)$ can be related by the smoothing operation if and only if there exists representatives in their equivalent classes which can be related by the smoothing operation.
	Thus we can view smoothing as the operation on the space $\mathcal{LCY}(X,\omega)$ or $\widetilde{\mathcal{LCY}}(X)$.
	
	For example, consider the symplectic log Calabi-Yau pair $(S^2\times S^2,C_1\cup C_2\cup C_3 \cup C_4,\omega)$ in case $(C4)$ of Theorem \ref{thm:minimal model}. We can get case $(C3)$(or $(C2)/(C1)$) if we perform the smoothing operation at $C_2\cap C_3$(or $C_2\cap C_3,C_3\cap C_4/C_1\cap C_2,C_2\cap C_3,C_3\cap C_4,C_4\cap C_1$).		
	
	We have the following observation:
	
	\begin{lemma}\label{lemma:smoothing}
		Assuming $X$ is rational, then any element in $\widetilde{\mathcal{LCY}}(X)$ can be obtained by performing smoothing operations on some element in $\widetilde{t \mathcal{LCY}}(X)$.
		\begin{proof}
			It is easy to see the conclusion holds for the minimal cases. From Theorem \ref{thm:minimal model} we see that cases $(B3),(C4),(D4)$ are toric, and all the other cases in $(B),(C),(D)$ can be obtained by smoothing the toric ones.
			Therefore, by the Lemma 3.4 of \cite{LiMa16-deformation}, we only need to show if $(\omega,D)\in\widetilde{ \mathcal{LCY}}(X)$ comes from smoothing an element in $\widetilde{t \mathcal{LCY}}(X)$, then its toric or non-toric blow-up will come from smoothing some element in $\widetilde{t \mathcal{LCY}}(X\#\overline{\mathbb{C}P^2})$.
			
			We may assume $D=C_1\cup\cdots C_{k-1}\cup C_k$ is obtained from smoothing the toric divisor $(A_{11}\cup\cdots\cup A_{1a_1})\cup\cdots\cup(A_{k1}\cup\cdots\cup A_{ka_k})$ at the intersection points $A_{il}\cap A_{i(l+1)}$ such that $[C_i]=\Sigma_{p=1}^{a_i}[A_{ip}]$ for all $1\leq i\leq k$. For the toric or non-toric blow up of $D$, without loss of generality, suppose their homology sequences are $([C_1]-E,E,[C_2]-E,\cdots,[C_k])$ and $([C_1]-E,[C_2],\cdots,[C_k])$ respectively, where $E$ is the exceptional class representing the generator of $\overline{\mathbb{C}P^2}$. Then we can perform small toric blow-up on the toric divisor $(A_{11}\cup\cdots\cup A_{1a_1})\cup\cdots\cup(A_{k1}\cup\cdots\cup A_{ka_k})$ at $ A_{1a_1}\cap A_{21} $ to get another toric divisor $(B_{11}\cup\cdots\cup B_{1a_1})\cup e\cup(B_{21}\cup\cdots\cup B_{2a_2})\cup\cdots$ on $X\#\overline{\mathbb{C}P^2}$ whose homological sequence is $$([A_{11}],\cdots, [A_{1a_1}]-E,E,[A_{21}]-E,\cdots,[A_{k1}],\cdots, [A_{ka_k}])$$
			So performing the smoothing operation at the intersection points $B_{il}\cap B_{i(l+1)}$ will lead to the toric blow-up of $D$ and if we continue to do one more smoothing operation at $e\cap B_{21}$ we will get the non-toric blow-up of $D$.
		\end{proof}
		
	\end{lemma}

	\subsection{$\mathcal{HLCY}(X,\omega)$, realizability and stability}
	In this section we introduce a more combinatorical description of the isotopy class space of symplectic log Calabi-Yau divisors, which is given by the underlying homological configurations. In the rest of this paper, whenever we talk about symplectic log Calabi-Yau divisors, these two descriptions will be used interchangeably.
	
	Fix a closed oriented 4-manifold $ X $. Let $ T=(V ,E, \{A_\alpha  \}_{\alpha \in V}) $ be a decorated graph with an underlying finite connected undirected graph $ (V,E) $. Then there is the map \[\phi:E\rightarrow \{\{x,y\}|x,y\in V,x\neq y\}\]
	Each vertex $ \alpha $ is labeled by a homology class $ A_\alpha \in H_2(X;\ZZ ) $, satisfying $ A_\alpha \cdot A_\beta = \# \phi^{-1}(\{\alpha,\beta\})$ for $ \alpha \neq \beta $.
	Such $ T $ is called a {\bf homological configuration} in $ X $ and we write $ [T]=\sum_{\alpha\in V} A_\alpha  $. Two homological configurations $ (V,E,\{ A_\alpha\}) $ and $ (V',E',\{ A'_\beta \}) $  are called equivalent if there is a graph isomorphism $ f: (V,E)\to (V',E') $ such that $ A_{\alpha}=A_{f(\alpha)} $ for all $ \alpha\in V $. We will always be talking about homological configurations up to equivalences.
	Note that for a fixed symplectic structure $ \omega $, we can associate to each vertex $ v_\alpha  $ an $ \omega  $-symplectic genus \[ g_\alpha =1+\dfrac{1}{2}(A_\alpha^2-c_1(X,\omega)\cdot A_\alpha ),\] which we also take as part of $ T $.
	When the decorated graph $T$ is a cycle, namely $V=\{1,\cdots,l\},E=\{e_1,\cdots,e_l\}$ with $ l\ge 2 $, $\phi(e_l)=\{1,l\}$, $\phi(e_i)=\{i,i+1\}$ for $1\leq i\leq l-1$ or $ V=\{ 1 \} $, $ E=\emptyset $, we call it a {\bf cyclic homological configuration}.
	We usually denote a cyclic homological configuration by a sequence of homology classes $ T=(A_1,\dots,A_l) $ and (anti)cyclic permutations of such sequences produce equivalent cyclic homological configurations.
	

	\begin{definition}\label{def:hlcy}
		For any symplectic rational manifold $(X,\omega)$, we define the space of \textbf{pre-homological log Calabi-Yau divisor} p$\mathcal{HLCY}(X,\omega)$ to be the set of equivalent classes of cyclic homological configurations $T=(V,E,\{ A_1,\cdots,A_l\} )$ satisfying that:
		\begin{itemize}
			\item Each $A_i\in H_2(X;\mathbb{Z})$ has an $\omega$-symplectic surface representative;
			\item $[T]=PD(c_1(X,\omega))$

			
		\end{itemize}
		There is a natural action of the the symplectomorphism group $\Symp(X,\omega)$ on p$\mathcal{HLCY}(X,\omega)$, factoring through $ D(X) $. And the space of \textbf{homological log Calabi-Yau divisor} $\mathcal{HLCY}(X,\omega)$ is defined to be p$\mathcal{HLCY}(X,\omega)/\Symp(X,\omega)$.
	\end{definition}
	
	\begin{rmk}
		If $\omega$ and $\omega'$ are two symplectic forms on $X$ with the same class $[\omega]=[\omega']$, the set of homology classes which can be represented by $\omega$-symplectic surfaces and  $\omega'$-symplectic surfaces are the same (\cite{ALLP-stability}). Therefore we can further define $p\mathcal{HLCY}(X,u)$ to be any $p\mathcal{HLCY}(X,\omega)$ with $u=[\omega]$. Moreover, from \cite{LiWu12-lagrangian} we know that the homological action $\Symp(X,\omega)/Symp_h(X,\omega)=D_{K_{\omega},[\omega]}$ by the symplectomorphism group also only depends on the symplectic class $[\omega]$. As a result, we could also define $\mathcal{HLCY}(X,u)=\mathcal{HLCY}(X,\omega)$ for any $u=[\omega]$.
	\end{rmk}
	
	Our goal in this section is to identify the geometric object, the isotopy class space $\mathcal{LCY}(X,\omega)$ with the homologcial object, the space $\mathcal{HLCY}(X,\omega)$. There is a natural map
	$$F:\mathcal{LCY}(X,\omega)\rightarrow \mathcal{HLCY}(X,\omega),$$
	which sends the geometric divisor $D=\cup_{i=1}^k C_i$ to its homological configuration $T=([C_1],\cdots,[C_k])$, and we say $T$ has a {\bf realization} by $D$.
	We would often write $ (D) $ instead of $ T $ when it is the underlying homological configuration of the divisor $ D $.
	Obviously the homological configurations obtained by the image of $F$ satisfy all the conditions in Definition \ref{def:hlcy}. We will prove $F$ is a bijection in the rest of this section.
	
	\begin{lemma}
		$F$ is well-defined and injective.
	\end{lemma}
	\begin{proof}
		If two geometric divisors are strictly deformation equivalent, then up to a symplectomorphism, they will be isotopic and thus have the same homological configurations. Now since we have taken the quotient of $\Symp(X,\omega)$ in the definition of $\mathcal{HLCY}(X,\omega)$, there is no ambiguity of the image of $F$. We see that $F$ is well-defined.
		
		If we have two geometric divisors $D=\cup_{i=1}^k C_i$ and $D'=\cup_{i=1}^k C'_i$ in $X$ whose image under $F$ are the same, then by definition, we have a symplectomorphism $h\in \Symp(X,\omega)$ such that $h_*([C_i])=[C_i']$ for all $1\leq i\leq k$. This implies $(X,\omega,D)$ and $(X,\omega,D')$ are strictly homologically equivalent and thus will be identified in the space $\mathcal{LCY}(X,\omega)$ according to our definition. This shows that $F$ is injective.
	\end{proof}
	
	The proof of surjectivity requires more work. Firstly we prove it for minimal models, that is, when $X=\CC\PP^2,\CC\PP^2\#\overline{\CC\PP^2}$ or $S^2\times S^2$. We use the standard basis $\{H\},\{H,E\},\{F,B\}$ of their second homology groups. Note that it's a classical result that any symplectic forms can be diffeomorphic to the ones satisfying $\omega(H)>0,\omega(H)>\omega(E)>0,\omega(F)\geq\omega(B)>0$ and the first Chern classes are $3H,3H-E,2F+2B$. Therefore we may make these assumptions in the following lemma.
	\begin{lemma}\label{lemma:minimal surjectivity}
		For $X=\CC\PP^2$ with symplectic form $\omega$ such that $\omega(H)>0$ and $c_1(X,\omega)=3H$, \[p\mathcal{HLCY}(X,\omega)=\mathcal{HLCY}(X,\omega)=\{(3H),(2H,H),(H,H,H)\}.\]
		For $X=\CC\PP^2\#\overline{\CC\PP^2}$ with symplectic form $\omega$ such that $\omega(H)>\omega(E)>0$ and $c_1(X,\omega)=3H-E$, elements of $p\mathcal{HLCY}(X,\omega)$ are of the following forms:
		
		\begin{enumerate}
			\item $(3H-E)$
			\item $(2H,H-E),((a+1)H-aE,(-a+2)H+(a-1)E)_{a\in\ZZ_+}$
			\item $(aH+(-a+1)E,H-E,(-a+2)H+(a-1)E)_{a\in\ZZ_+}$
			\item $(aH+(-a+1)E,H-E,(-a+1)H+aE,H-E)_{a\in\ZZ_+}$
		\end{enumerate}
		For $X=S^2\times S^2$ with symplectic form $\omega$ such that $\omega(F)\geq\omega(B)>0$ and $c_1(X,\omega)=2B+2F$, elements of $p\mathcal{HLCY}(X,\omega)$ are of the following forms:
		
		\begin{enumerate}
			\item $(2F+2B)$
			\item $(2F+B,B),((F+bB,F+(2-b)B)_{b\in\ZZ_+}$
			\item $(F+bB,B,F+(1-b)B)_{b\in\ZZ_+}$
			\item $(F+(b-1)B,B,F+(1-b)B,B)_{b\in\ZZ_+}$
		\end{enumerate}
		As a result, $F$ is surjective when $X$ is one of these three manifolds.
	\end{lemma}
	\begin{proof}
		The case when $l=1$ is always trivial, next we will assume $l\geq 2$. Note that when $l\geq 2$, by condition, $A_i\cdot(c_1-A_i)=A_i\cdot(A_1+\cdots+A_{i-1}+A_{i+1}+\cdots+A_l)=2$. So by adjunction we see that all the classes are actually represented by symplectic surfaces of genus $g_i=0$.
		
		For $\CC\PP^2$, the result is obvious since we know the classes of symplectic surfaces must be $kH$ with positive $k$.
		
		For $X=\CC\PP^2\#\overline{\CC\PP^2}$, suppose $A=xH+yE$ with $ A^2=t $ has a symplectic sphere representative, then by adjunction formula we have that
		\[\begin{cases}
		x^2-y^2=t\\
		t=3x+y-2
		\end{cases}.\]
		It follows that
		$\begin{cases}
		x=\frac{t+1}{2}\\
		y=\frac{-t+1}{2}
		\end{cases}$ or
		$\begin{cases}
		x=\frac{t}{4}+1\\
		y=\frac{t}{4}-1
		\end{cases}$.
		In the former case we get $A=aH+(1-a)E$ with odd self intersection $2a-1=t$, and in the later cases we get $A=(k+1)H+(k-1)E$ with even self intersection $4k=t$. Note that the intersection between two odd classes $a_1H+(1-a_1)H$ and $a_2H+(1-a_2)E$ is $a_1+a_2-1$, the intersection between two even classes $(k_1+1)H+(k_1-1)E$ and $(k_2+1)H+(k_2-1)E$ is $2(k_1+k_2)$, and the intersection between an odd class $aH+(1-a)E$ and an even class $(k+1)H+(k-1)E$ is $2ak-k+1$.
		Using the fact that the homological configuration is cyclic, we have some simple observations.
		\begin{itemize}
			\item When $l>3$, $A_i$, $A_{i+1}$ and $A_{i+2}$ can not all be odd classes.
			\item When $l>2$, if $A_i$ is an even class, $A_{i-1},A_{i+1}$ must be odd classes and $A_i=H-E$.
			\item If $A_i=H-E$, all the other $A_j$ with $j\neq i+1,i-1$ must also be $H-E$.
		\end{itemize}
		Now we could easily exclude the cases when the length $l>4$ by considering the arrangement of odd classes and even classes according to the first two bullets above. Then it's possible to list all the homological configurations for $l\leq 4$ which will coincide with our statement.
		
		The analysis for $S^2\times S^2$ is similar to $\CC\PP^2\#\overline{\CC\PP^2}$. By adjunction formula, if $xF+yB$ has a self intersection $t$-spherical symplectic representative, then we have
		$\begin{cases}
		2xy=t\\
		t=2x+2y-2
		\end{cases}$. Since we made the assumption $\omega(F)\geq\omega(B)\geq0$ it follows that all the possible classes are either $F+bB$ for arbitrary $b$ or $bF+B$ for non negative $b$. Note that the intersection between $F+b_1B$ and $F+b_2B$ or $b_1F+B$ and $b_2F+B$ is $b_1+b_2$, and the intersection between $F+b_1B$ and $b_2F+B$ is $1+b_1b_2$. We then see that:
		\begin{itemize}
			\item When $l>3$, if $A_i=bF+B$, then $b\leq 1$ and all the other $A_j$ with $j\neq i+1,i-1$ must be $B$ if $b=0$ or $F-B$ if $b=1$
			\item If $A_i=F+b_1B,A_{i+1}=F+b_2B,A_{i+2}=F+b_3B$, then $l>3$ and $b_1=b_3=0,b_2=1$.
			\item If $A_i=B$ or $F$, then all the other $A_j$ with $j\neq i+1,i-1$ must be $B$ or $F$.
		\end{itemize}
		Again, the above observations are enough for us to exclude $l>4$ and list all the possible cases for $l\leq 4$.\\
		Note that the homological configurations $(H,H,H),(aH+(-a+1)E,H-E,(-a+1)H+aE,H-E),(F+(b-1)B,B,F+(1-b)B,B)$ can be realized as the boundary divisors of the toric actions (moment map's preimage of the boundary of Delzant polygon). And all the other homological configurations come from the smoothings of the toric ones. Therefore the claim of $F$ is surjective has been verified for these three cases.
	\end{proof}
	
	We need some results before proving the surjectivity for all rational manifolds.
	A symplectic form $ \omega  $ on $ X $ is called {\bf $ T $-positive} if $ \omega (A_\alpha)>0 $ for all $ \alpha  $. The following result is crucial for our proof of realizability.
	\begin{lemma}[\cite{DoLiWu18-stability}]\label{lemma:stability}
		Let $ (X,\omega ) $ be a symplectic manifold with $ b^+=1 $ and $ T $ a homological configuration realized by a symplectic divisor $ D $. Then $ D $ is $ \omega $-stable, that is, for any $ T $-positive symplectic form $ \omega' $ deformation equivalent to $ \omega  $, there is a symplectic divisor $ D' $ realizing $ T $ with respect to $ \omega' $.
	\end{lemma}

	\begin{lemma}\label{lemma:induction}
		Suppose $(X,\omega)$ is the symplectic blow up of $(X',\omega')$ with the canonical identification $H_2(X,\ZZ)=H_2(X',\ZZ)\oplus \ZZ E$ and the homological configuration $(A_1,\cdots,A_l)\in \mathcal{HLCY}(X',\omega')$ has a realization. Then if $(A_1,\cdots,A_i-E,\cdots,A_l)$ or $(A_1,\cdots,A_{i}-E,E,A_{i+1}-E,\cdots,A_l)$ is in $\mathcal{HLCY}(X,\omega)$, they also have a realization.
	\end{lemma}
	
	\begin{proof}
		For the realization $(X',\omega',D')$ of $(A_1,\cdots,A_l)$ we could always perform small blow up to obtain $(X'',\omega'',D'')$ realizing $(A_1,\cdots,A_i-E',\cdots,A_l)$ or $(A_1,\cdots,A_{i}-E',E',A_{i+1}-E',\cdots,A_l)$. Here we identify $H_2(X'',\ZZ)=H_2(X',\ZZ)\oplus \ZZ E'$. Choose a diffeomorphism $f:X\rightarrow X''$ such that $f_*(E)=E'$ and $f_*|_{H_2(X',\ZZ)}$ is identity. We see that $(A_1,\cdots,A_i-E,\cdots,A_l)$ or $(A_1,\cdots,A_{i}-E,E,A_{i+1}-E,\cdots,A_l)$ has a realization under the symplectic form $(f^{-1})^*\omega''$.
		
		It is well known (for example, Theorem 3.5 and Theorem 3.18 of \cite{Li2008}) that the diffeomorphism group acts transitively on the space of symplectic forms in a fixed class and any two symplectic forms on a rational manifold are deformation equivalent up to a diffeomorphism fixing the symplectic class.
		We can pick $g\in \Diff^+(X)$ such that $g^*(f^{-1})^*\omega''$ is deformation equivalent to $\omega$. Since $g_*\in D_{K_{(f^{-1})^*\omega''},[(f^{-1})^*\omega'']}\cong \Symp(X,(f^{-1})^*\omega'')/Symp_h(X,(f^{-1})^*\omega'')$, we could pick $h\in \Symp(X,(f^{-1})^*\omega'')$ such that $h_*g_*$ is identity. It then follows that \begin{align*}(A_1,\cdots,A_i-E,\cdots,A_l)&=h_*g_*(A_1,\cdots,A_i-E,\cdots,A_l)\text{ or}\\
		(A_1,\cdots,A_{i}-E,E,A_{i+1}-E,\cdots,A_l)&=h_*g_*(A_1,\cdots,A_{i}-E,E,A_{i+1}-E,\cdots,A_l)
		\end{align*} has a realization under the form $g^*((f^{-1})^*\omega'')$. Of course the assumption that the blown-up configuration is in $\mathcal{HLCY}(X,\omega)$ guarantees the $\omega$-positivity. By lemma \ref{lemma:stability} we see that it then also has a realization under the form $\omega$.
	\end{proof}
	
	\begin{remark}\label{rmk:positive area}
		Note that one can run the above argument for a single class to get the following useful statement that will be used later: with the same assumption in Lemma \ref{lemma:induction}, if $A$ can be represented by $\omega'$-symplectic surface and $A-E$ has positive $\omega$-area, then it can be represented by $\omega$-symplectic surface.
	\end{remark}
	
	As we want to prove by induction, the following useful result will enable us to perform the blow down process which reduces the arbitrary case to the case of minimal models.
	
	\begin{lemma}[\cite{Pi-max}]\label{lemma:minimal area}
		Let $ (X,\omega ) $ be a symplectic manifold with $X\neq\CC\PP^2\#\overline{\CC\PP^2}$, then for any $\omega$-tame almost complex structure $J$, all symplectic exceptional classes of minimal area have $J$-holomorphic representatives.
	\end{lemma}
	
	\begin{prop}\label{prop:realization}
		$F$ is surjective for all $(X,\omega)$. Therefore $\mathcal{LCY}(X,\omega)=\mathcal{HLCY}(X,\omega)$.
	\end{prop}
	\begin{proof}
		Suppose $X=\CC\PP^2\#n\overline{\CC\PP^2}$ with $n\geq 2$ and let $E$ be the minimal area symplectic excetional class. For $(A_1,\cdots,A_l)\in \mathcal{HLCY}(X,\omega)$, the case when $l=1$ is trivial and we firstly assume $l\geq3$. By adjunction formula, $E\cdot(A_1+\cdots+A_l)=1$.
		Since each $ A_i $ has a symplectic representative and thus a $ J $-holomorphic representative for some tame $ J $, by Lemma \ref{lemma:minimal area}, we have $ E\cdot A_i\ge 0 $ or $ A_i= cE $ for some positive integer $ c $. Note that the symplectic genus of $ A_i $ being non-negative forces $ c=1 $.
		Thus either $E\cdot A_i=1$ for some $i$ and $E\cdot A_j=0$ for all $j\neq i$, or $E= A_i<0$ for some $i$.
		
		In the former case, we could pick a tame $J$ such that the symplectic surface representing $A_i$ is $J$-holomorphic. By lemma \ref{lemma:minimal area}, $E$ also has a $J$-holomorphic representative. By positivity of intersection, they intersect transversally at one point. Thus if we blow down the exceptional sphere $C_i$ representing $E$, we get another symplectic manifold $(X',\omega')$ with $A_i-E$ having $\omega'$ representative. For other $A_j$, one can also find exceptional sphere $C_j$ representing $E$ such that they don't intersect. Since all symplectic exceptional spheres are symplectic isotopic, and there exists an ambient symplectic isotopy inducing that, there must be a representative of $A_j$ not intersecting $C_i$. Thus $A_j$ also have $\omega'$ representative in $X'$. Now we see that $(A_1,\cdots,A_i-E,\cdots,A_l)\in \mathcal{HLCY}(X',\omega')$. By induction and lemma \ref{lemma:induction},  $(A_1,\cdots,A_l)$ has a realization.
		
		For the later case,
		we have $A_{i-1}\cdot E=A_{i+1}\cdot E=1$ and $A_j\cdot E=0$ for $j\neq i-1,i,i+1$. Now by a similar blow down argument as above, $(A_1,\cdots,A_{i-1}+E,A_{i+1}+E,\cdots,A_l)\in \mathcal{HLCY}(X',\omega')$ for some blow down manifold $(X',\omega')$. Again, the result follows from induction and lemma \ref{lemma:induction}.
		
		Now assume $l=2$, the homological configuration becomes $(K-E,E)$. Since $n\geq2$, we could always find another exceptional class $E'$ such that $E'\cdot E=0$ and thus $E'\cdot (K-E)=1$. By Theorem 1.2.7 of \cite{McOp13-nongeneric}, there exists some tame almost complex structure $J$ such that both $E'$ and $K-E$ has $J$-holomorphic representative. By positivity of intersection, these two curves meet transversally at one point. If we blow down the $E'$ curve to get the manifold $(X',\omega')$, it follows that $(K-E-E',E)\in p\mathcal{HLCY}(X',\omega')$. Now the result also follows from lemma \ref{lemma:induction}.
		
	\end{proof}

	From Proposition \ref{prop:realization} we also see that $ \mc{LCY}(X,\omega) $ should be controlled by the set of symplectic sphere classes in $ (X,\omega) $, which leads to the stability result (Theorem \ref{thm:stability}).
	
	\begin{proof}[Proof of Theorem \ref{thm:stability}]
		By Proposition \ref{prop:realization}, it suffices to show $ \mc{HLCY}(X,\omega)=\mc{HLCY}(X,\omega') $. Note that by Proposition 4.1 of \cite{LiLiu01-uniqueness} if $ \mc{S}^{-1}_\omega=\mc{S}^{-1}_{\omega'} $ we must have $ K_\omega=K_{\omega'}=:K $. For any homological configuration $ T=(A_1,\dots,A_k)\in p\mc{HLCY}(X,\omega) $, we would have $ \sum A_i=c_1(X,\omega)=c_1(X,\omega') $ and each $ A_i\in \mc{S}_\omega=\mc{S}_{\omega'} $. So $ T\in p\mc{HLCY}(X,\omega') $ and $ p\mc{HLCY}(X,\omega)=p\mc{HLCY}(X,\omega')\subset H_2(X;\Z) $. By Proposition 4.1 of \cite{LiWu12-lagrangian}, $ D_{K,\omega} $ is generated by reflections along classes in $ \mc{L}_{K,\omega} $. Note that \[ \mc{L}_{K,\omega}\bigsqcup \mc{S}^{-2}_\omega\bigsqcup (-\mc{S}^{-2}_\omega) =\mc{L}_{K,\omega'}\bigsqcup \mc{S}^{-2}_{\omega'}\bigsqcup (-\mc{S}^{-2}_{\omega'}), \]
		we must have $ \mc{L}_{K,\omega}=\mc{L}_{K,\omega'} $ and thus $ \mc{HLCY}(X,\omega)=\mc{HLCY}(X,\omega') $.
	\end{proof}

	\subsection{Finiteness}

	In the holomorphic category, we have the following finiteness result.
	\begin{thm}[\cite{Fr}]\label{thm:holomorphic finite deformation}
		There are only finitely many deformation types of holomorphic log Calabi-Yau pairs $ (Y,D) $ with the same self-intersection sequence.
	\end{thm}
	Since there is always a Kahler pair in a symplectic deformation class of symplectic log Calabi-Yau pairs (\cite{LiMa16-deformation}), we have the following analogous finiteness result in the symplectic category.
	\begin{thm}[\cite{LiMaMi-logCYcontact}] \label{thm: finite deformation}
		There are only finitely many symplectic deformation types of symplectic log Calabi-Yau  pairs with the same self-intersection sequence.
	\end{thm}
	
	Now we show that such finiteness also holds for strict symplectic deformation classes. We need to make a digression first.
	
	Having established the equivalence between $\mathcal{LCY}(X,\omega)$ and $\mathcal{HLCY}(X,\omega)$, we also want to describe it in a more convenient way. The finiteness of $\mathcal{LCY}(X,\omega)$ will be shown by that description in this section. In order to do that we need to introduce the concepts of reduced cones and blowup forms for rational manifolds. Denote by $ M_l $ the $ l $-fold blow-up $ \CC\PP^2\# l\overline{\CC\PP}^2 $. Fix a standard basis $ \{ H,E_1,\dots,E_l \} $ of $ H_2(M_l;\ZZ) $. A class $ \lambda H-\sum \delta_iE_i $ is called {\bf reduced} if \[
	\delta_1\ge \delta_2\ge \dots \ge \delta_k>0\quad and \quad \lambda \ge \delta_1+\delta_2+\delta_3.
	\]
	We also say a second cohomology class is reduced if its Poincare dual is reduced.
	\begin{definition}
		A symplectic form $ \omega  $ is {\bf reduced} if $ [\omega ] $ is reduced, in which case $ [\omega ] $ is called a reduced symplectic class and $ (\lambda,\delta_1,\dots,\delta_k) $ is called the reduced vector of $ [\omega] $. The space $ \tilde{P}_l=\tilde{P}(M_l) $ of reduced symplectic classes is called the {\bf reduced symplectic cone}.
		The {\bf normalized reduced symplectic cone} $ P_l=P(M_l) $ is the space of reduced symplectic classes on $ M_l $ with $ \lambda=1 $.
	\end{definition}
	
	For $M_{1}'=S^2\times S^2 $, let $B=[S^2\times pt],F=[pt\times S^2]$. The reduced symplectic classes are those of the forms satisfying $\omega(F)=\lambda,\omega(B)=\mu $ with $ \lambda \geq \mu> 0 $ and it is normalized if $ \lambda=1 $. The crucial fact we need about reduced symplectic cones is that, for symplectic rational surfaces, the reduced symplectic cone is the fundamental domain of the action of orientation-preserving diffeomorphism on the $K_0$-symplectic cone  (\cite{Taubes95-SW=Gr}, \cite{LiLiu95-ruled}). This means whenever we have a symplectic rational manifold $(X,\omega)$ in hand, then up to scaling, we may always assume it is $(M_l,\omega_{\delta_1,\cdots,\delta_l})$ with reduced or $(M_1',\omega'_\mu)$.

	A {\bf blowup form} on $M_l$ is a symplectic form for which there exist pairwise disjoint embedded symplectic spheres in the class $H,E_1,\cdots,E_l$. The result in \cite{KaKe17-blowup} shows that every normalized reduced class encoded by $(\delta_1,\cdots,\delta_l)$ with $1>\delta_1^2+\cdots+\delta_l^2$ contains a blowup form. Therefore, given $(M_l,\omega_{\delta_1,\cdots,\delta_l})$ with some reduced blowup form, we could blow down those disjoint exceptional spheres to obtain a sequence of symplectic manifolds $(M_{l-1},\omega_{\delta_1,\cdots,\delta_{l-1}}),\cdots,(M_{1},\omega_{\delta_1})$, all of which are equipped with reduced blowup forms. Note that in $(M_2,\omega_{\delta_1,\delta_2})$, $H-E_1-E_2$ also has exceptional spherical representative, whose blow down manifold is $(M_1',\omega'_{\frac{1-\delta_1}{1-\delta_2}})$ after scaling. Note that in this case the inclusion $H_2(M_1',\ZZ)\hookrightarrow H_2(M_2,\ZZ)$ sends $F$ to $H-E_2$ and $B$ to $H-E_1$ so that $H_2(M_2,\ZZ)=H_2(M_1',\ZZ)\oplus \ZZ \langle H-E_1-E_2\rangle$.
	
	\begin{remark}
		Since any symplectic form is diffeomorphic to some blowup form and the spaces of divisors (both $\mathcal{LCY}$ and $\mathcal{HLCY}$) for diffeomorphic symplectic forms can be identified naturally, our strategy for proving the finiteness in this section and counting results later is by only considering the spaces under reduced blowup forms.
	\end{remark}
	
	Now we need the following results:
	
	\begin{lemma}(\cite{KaKe17-blowup})\label{lemma:kk17}
		If $l\geq 3$, $E_l$ has the minimal area among all symplectic exceptional classes.
	\end{lemma}
	
	\begin{lemma}(\cite{Zhang-curve})\label{lemma:zhang}
		For any tame $J$ on $M_2$, if $E_2$ has no $J$-holomorphic representative, then $E_1$ and $H-E_1-E_2$ must have $J$-holomorphic representatives.
	\end{lemma}
	
	Using the above lemmas, by the same argument in Proposition \ref{prop:realization} it can be shown that:
	\begin{lemma}\label{lemma:reduced induction}
		Assume $\omega_{\delta_1,\cdots,\delta_l}$ is a reduced blowup form on $M_l$. Suppose $(A_1,\cdots,A_k)\in p\mathcal{HLCY}(M_l,\omega_{\delta_1,\cdots,\delta_l})$, when $k\geq3,l\geq2$, one of the followings occurs:
		\begin{itemize}
			\item  There is exactly one $1\leq i\leq k$ such that $A_i\notin H_2(M_{l-1};\mathbb{Z})$, in which case $A_i+E_l\in H_2(M_{l-1};\mathbb{Z})$ and $(A_1,\cdots,A_{i-1},A_{i}+E_l,A_{i+1},\cdots,A_k)\in p\mathcal{HLCY}(M_{l-1},\omega_{\delta_1,\cdots,\delta_{l-1}})$;
			\item  There is exactly one $1\leq i\leq k$ such that $A_i=E_l$, in which case $A_j+E_l\in H_2(M_{l-1};\mathbb{Z})$ if $|i-j|\equiv 1$ (mod $k$), $A_j\in H_2(M_{l-1};\mathbb{Z})$ for all the other $j\neq i$ and $(A_1,\cdots,A_{i-2},A_{i-1}+E_l,A_{i}+E_l,A_{i+1},\cdots,A_k)\in p\mathcal{HLCY}(M_{l-1},\omega_{\delta_1,\cdots,\delta_{l-1}})$.
		\end{itemize}
		When $k=2$, if none of the above cases happens, $(A_1,A_2)=(3H-E_1-\cdots-E_{l-1}-2E_l,E_l)$.
	\end{lemma}
	\begin{proof}
		By lemma \ref{lemma:kk17}, unless $l=2$ and $E_2$ is not of minimal area, $E_l$ can always be represented by $J$-holomorphic sphere for any tame almost complex structure $J$ due to lemma \ref{lemma:minimal area}, in which case the result will follow from the argument in Proposition \ref{prop:realization}. So we only need to consider the case when $l=2$ and $E_2$ is not of minimal area.
		
		If $E_2$ is in one of the component of $(A_1,\cdots,A_k)$, it will fall into the second bullet without any pseudoholomorphic argument. So we further assume there is no $A_i$ equal to $E_2$. We want to show any component $A_i=aH-bE_1-cE_2$ has non negative intersection with $E_2$ so that we can apply Theorem 1.2.7 of \cite{McOp13-nongeneric} to find tame $J_i$ such that $E_2,A_i$ both have $J_i$-holomorphic representatives. If $A_i^2<0$, according to the classification of negative symplectic sphere classes in $M_2$ in Proposition 3.4 of \cite{LiLi20-pacific} (See also Proposition \ref{prop:spherical classes}),it's immediate that $E_2\cdot A_i\geq0$. So next we assume $A_i^2\geq0$. Firstly we pick an arbitrary tame $J$ such that $A_i$ can be represented by $J$-holomorphic sphere. By lemma \ref{lemma:zhang}, if $E_2$ has no $J$-holomorphic representative, then $E_1,H-E_1-E_2$ must have. By positivity of intersection and the assumption that $A_i^2\geq0$, it follows that $E_1\cdot A_i\geq 0$ and $(H-E_1-E_2)\cdot A_i\geq 0$ for arbitrary $i$. Since $E_1\cdot(A_1+\cdots+ A_k)=(H-E_1-E_2)\cdot(A_1+\cdots+ A_k)=1$ by adjunction,
		we can conclude that $1\geq E_1\cdot A_i\geq 0$ and $1\geq (H-E_1-E_2)\cdot A_i\geq 0$. Now we have $1\geq b\geq0$, $1\geq a-b-c\geq0$ and $3a-b-c\geq 2$ by adjunction. Now it's easy to see these three inequalities imply $c\geq0$.
		
		Therefore, even if $E_2$ is not of minimal area, we are still able to find $J_i$ such that $E_2,A_i$ both have $J_i$-holomorphic representatives. And we can use the argument in Proposition \ref{prop:realization} to get the conclusion.
	\end{proof}
	
	With the help of this lemma and the list in Lemma \ref{lemma:minimal surjectivity}, we can write down all the homology classes ingredients for homological log Calabi-Yau divisors under the reduced blowup form on $M_l$ with pairwise disjoint symplectic spherical classes $H,E_1,\cdots,E_l$:
	
	The homology classes appearing in the homological log Calabi-Yau divisors must belong to
	\[
	\mathcal{H}_l=\left\{
	\begin{array}{l}
	kH-(k-1)E_1-\sum_{i=2}^l \epsilon_i E_i, k\in\ZZ,\epsilon_i\in\{0,1\}\\
	H-E_1-\sum_{i=2}^l \epsilon _i E_i, \epsilon_i\in\{0,1\}\\
	E_p-\sum_{i=p+1}^l\epsilon _i E_i, 2\leq p\leq l,\epsilon_i\in\{0,1\}\\
	2H-\sum_{i=2}^l \epsilon _i E_i, \epsilon_i\in\{0,1\}\\
	3H-E_1-\cdots-E_{l}\\
	3H-E_1-\cdots-E_{p-1}-2E_p-\sum_{i=p+1}^l\epsilon _i E_i, 2\leq p\leq l,\epsilon_i\in\{0,1\}
	\end{array}
	\right\}
	\]
	
	The homology classes appearing in the toric homological log Calabi-Yau divisors must belong to
	\[
	t\mathcal{H}_l=\left\{
	\begin{array}{l}
	kH-(k-1)E_1-\sum_{i=2}^l \epsilon_i E_i, k\in\ZZ,\epsilon_i\in\{0,1\}\\
	H-E_1-\sum_{i=2}^l \epsilon _i E_i, \epsilon_i\in\{0,1\}\\
	E_p-\sum_{i=p+1}^l\epsilon _i E_i, 2\leq p\leq l,\epsilon_i\in\{0,1\}
	\end{array}
	\right\}
	\]

	Moreover, by Remark \ref{rmk:positive area} and induction, we see that $A\in\mathcal{H}_l$ can be represented by a symplectic surface if and only if $\omega_{\delta_1,\cdots,\delta_l}(A)>0$. The immediate corollaries are the followings:
	\begin{corollary}\label{lemma:elliptic case}
		For each symplectic rational surface $ (X,\omega) $ with $ c_1(X,\omega)\cdot [\omega]>0 $, there is a unique elliptic symplectic log Calabi-Yau divisor in $ \mc{LCY}(X,\omega) $.
	\end{corollary}
	
	\begin{corollary}\label{cor:ingredients}
		Under the reduced blowup form $\omega$ on $M_l$ with pairwise disjoint symplectic spherical classes $H,E_1,\cdots,E_l$, the cyclic homological configuration $T$ given by $(A_1,\cdots,A_k)$ is in $\mathcal{HLCY}(M_l,\omega)$ if and only if:
		\begin{itemize}
			\item $A_i\in \mathcal{H}_l$, for all $1\leq i\leq k$;
			\item $[T]=PD(c_1(M_l,\omega))$;
			\item $\omega$ is $T$-positive.
		\end{itemize}
	\end{corollary}
	
	\begin{corollary}\label{lemma:finiteness}
		$\mathcal{LCY}(X,\omega)$ is finite.
	\end{corollary}
	\begin{proof}
		Since the length of the divisors in $\mathcal{LCY}(X,\omega)$ has an upper bound, it's enough to show for a fixed $\omega$, only fintely many classes $A\in\mathcal{H}_l$ are in the homological configurations. This is implied by the obvious inequality $\omega\cdot c_1(X,\omega)\geq \omega(A)>0$, thus a bound of $k$.
	\end{proof}

	\subsection{Tautness and toric symplectic log Calabi-Yau divisors}\label{section:toric divisors}
	In the holomorphic category, a labeled anticanonical pair is called taut if its isomorphism class is determined by the self-intersection sequence. The tautness of holomorphic anticanonical pairs is used to prove Torelli Theorem in \cite{Fr}. Now we introduce two kinds of symplectic analogy of tautness as follows. A label on a symplectic log Calabi-Yau divisor $D\in p\mathcal{LCY}(X,\omega)$ is a bijection from $\{1,\cdots,k\}$ to the irreducible components of $D$ such that each two consecutive numbers are mapped to two components which intersect. For simplicity, we will just use $(C_1,\cdots,C_k)$ to denote a labeled symplectic log Calabi-Yau divisor. Two labeled divisors $(D=(C_1,\cdots,C_k),\omega)$ and $(D'=(C_1',\cdots,C_k'),\omega')$ with the same number of components are symplectic deformation equivalent if they are connected by a family of symplectic divisors $ (D^t=(C_1^t,\cdots,C_k^t),\omega^t) $, up to an orientation-preserving diffeomorphism. Clearly, two unlabeled divisors are symplectic deformation equivalent in the sense of Definition \ref{def:deformation} if and only if there exist labels on them such that they are symplectic deformation equivalent as labeled divisors. Similarly, we can also define the strict deformation equivalence for labeled divisors.
	
	\begin{definition}
		A labeled symplectic log Calabi-Yau divisor $(D=(C_1,\cdots,C_k),\omega)$ on $X$ is called {\bf def-taut} if all the labeled symplectic log Calabi-Yau divisors $(D'=(C_1',\cdots,C_k'),\omega')$ on $X$ with $[C_i]^2=[C_i']^2$ for all $1\leq i\leq k$ are symplectic deformation equivalent to $D$. It's called {\bf iso-taut} if all the labeled symplectic log Calabi-Yau divisors $(D'=(C_1',\cdots,C_k'),\omega')$ on $X$ with $[C_i]^2=[C_i']^2,\omega(C_i)=\omega'(C_i')$ for all $1\leq i\leq k$ are strictly symplectic deformation equivalent to $D$.
	\end{definition}
	
	Next we will focus on the iso-tautness which will be used in the proof of Proposition \ref{prop:CY to toric}. We put the discussion of def-tautness into the appendix since there is no direct application in the remaining of this paper.

	Recall that a symplectic Looijenga pair $ (X,\omega,D) $ is called a \textbf{toric symplectic log Calabi-Yau pair} if $ q(D)=0 $, where $q(D)=12-k(D)-D^2=12-3k-\sum_{i=1}^k s_i  $. When we fix the ambient symplectic manifold $ (X,\omega) $, we also call $ D $ a {\bf toric symplectic log Calabi-Yau divisor}. From the definition we see that being toric only depends on the self-intersection sequence $ s $.
	In the following we discuss some basic properties of toric symplectic log Calabi-Yau divisors.
	
	The following lemma
	shows that this definition of being toric coincides with that in \cite{Fr}.
	\begin{lemma}\label{lemma:toric=toric blow up}
		A symplectic log Calabi-Yau divisor $ D\subset (X,\omega) $ is toric if and only if it is an iterated toric blow-up of toric symplectic log Calabi-Yau divisors in $ \CP^2 $, $ \CP^2\# \overline{\CP}^2 $ and $ S^2\times S^2 $.
	\end{lemma}
	\begin{proof}
		The 'if' part is obvious since $ q(D) $ is preserved under toric blow-ups and $ q(D)=0 $ for minimal models B3, C4 and D4. Now suppose $ D $ is toric. By minimal reduction (\cite{LiMa16-deformation}), we can do a maximal sequence of non-toric blow-downs and then a maximal sequence of toric blow-downs to get a minimal model $ (X',\omega',D') $. Note that non-toric blow-down decreases $ q(D) $, we have $ q(D')\le q(D)=0 $. By checking the minimal models in Theorem \ref{thm:minimal model}, we see that all of them have $ q(D')\ge 0 $. So we must have $ q(D') =0$ and $ D $ is a toric blow-up of $ D' $. The only possible $ D' $ are B3, C4 and D4, which are toric ones among minimal models.
	\end{proof}
	
	The above lemma implies that a Kahler log Calabi-Yau pair $ (X,\omega,D) $ is a toric symplectic log Calabi-Yau pair if and only if $ (X,D) $ is a toric anticanonical pair in the holomorphic category.
	
	
	A toric anticanonical divisor is taut while a general anticanonical divisor might not be taut in the holomorphic category (\cite{Fr}). Such rigidity of toric divisors persists to hold in the symplectic category.
	\begin{lemma}\label{lemma:taut}
		Any labeled toric symplectic log Calabi-Yau divisor is both def-taut and iso-taut.
	\end{lemma}
	\begin{proof}
		Let $ D,D'\subset (X,\omega) $ be two toric symplectic log Calabi-Yau divisors with the same self-intersection sequence $ s(D)=s(D') $. Then there are two Kahler log Calabi-Yau divisors $ (\overline{\omega},\overline{D}) $ and $ (\overline{\omega}',\overline{D}') $ in $ X $ symplectic deformation equivalent to $ (\omega,D) $ and $ (\omega,D') $ respectively by Theorem \ref{thm: symplectic deformation class=homology classes}. The toric Kahler log Calabi-Yau divisors are taut (in the holomorphic sense) by Lemma 2.15 of \cite{Fr} and there is an isomorphism $ \phi:(X,\overline{D}) \to (X,\overline{D}')$. The isomorphism $ \phi  $ is in particular an orientation-preserving diffeomorphism and thus a homological equivalence. So $ (X,\omega,D) $ is also def-taut as a symplectic Calabi-Yau pair.
		
		Suppose $ (X,\omega,D) $ and $ (X',\omega',D') $ satisfy $ (s,a)=(s',a') $, then by above discussion there is a diffeomorphism $ \Phi:X\to X' $ such that $ \Phi_*[C_i]=[C'_i] $. Then we have $ \Phi^*[\omega']\cdot [C_i]=\Phi^*[\omega']\cdot (\Phi_*)^{-1}[C'_i]=[\omega']\cdot [C'_i]=\sum a_i'=\sum a_i=[\omega]\cdot [C_i] $, i.e $ (\Phi^*[\omega']-[\omega])\cdot [C_i] $ for all $ i $. Then $ \Phi^*[\omega']=[\omega] $ and $ \Phi  $ is a strictly homological equivalence, which follows from the fact that $ \{ [C_i] \} $ contains a basis of $ H_2(X;\ZZ) $.
		
		We include a proof of this fact for completeness. For each of the minimal models B3, C4 and D4, we can see that the homology classes of the components generate the second homology. We proceed by induction and assume that $ \{ [C_i] \} $ generates $ H_2(X;\ZZ) $ for every toric symplectic log Calabi-Yau pair $ (X,\omega,D) $. Denote by $ (X',\omega',D') $ a toric blow-up of $ (X,\omega,D) $. Then $ H_2(X';\ZZ)=H_2(X;\ZZ)\oplus \ZZ E $ where $ E $ is the exceptional class, such that \[
		\{ [C'_1],\dots,[C'_{k+1}] \}=\{ [C_1],\dots,[C_{k-1}]-E,E,[C_k]-E \},
		\]
		which is linearly equivalent to $ \{ [C_1],\dots,[C_{k-1}],[C_k],E \} $, and thus generates $ H_2(X';\ZZ) $.
	\end{proof}
	%
	
	\begin{lemma}\label{lemma:q(D)=0 imply b^+=1}
		Let $ D $ be a toric symplectic log Calabi-Yau divisor, we have $ b^+(D)=1 $.
	\end{lemma}
	\begin{proof}
		If $ Q_D $ is negative definite, then $ q(D)\ge 3 $ (\cite{Fr}, Corollary 1.3) and cannot be toric. If $ Q_D $ is negative semi-definite, then $ D $ is a cycle of $ -2 $ spheres. By Lemma \ref{lemma:toric=toric blow up}, $ D $ is obtained from $ (1,1,1) $ or $ (n,0,-n,0) $ by toric blow-ups. In particular, there must be at least one $ -1 $ sphere in $ D $. So a symplectic circular spherical divisor $ D $ cannot be negative semi-definite. So we must have $ b^+(D)=1 $.
	\end{proof}
	\begin{rmk}
		A symplectic circular spherical divisor $ D\subset (X,\omega) $ is an embedded normal crossing symplectic divisor which is topologically a cycle of spheres (see \cite{LiMaMi-logCYcontact} for definition). It follows from Theorem 1.3 of \cite{LiMaMi-logCYcontact} that such $ D $ must be log Calabi-Yau if $ b^+(D)=1 $. So a symplectic circular spherical divisor $ D $ is a toric symplectic log Calabi-Yau divisor if and only if $ b^+(D)=1 $ and $ q(D)=0 $.
	\end{rmk}
	Another way to characterize toric symplectic log Calabi-Yau divisors is through their boundaries. To see this, we need to recall some notions on divisor neighborhoods in \cite{LiMa14-divisorcap}.
	Let $ (D=\cup_{i=1}^k C_i,\omega) $ be a symplectic divisor. A closed regular neighborhood $ P(D) $ of $ D $ is  called a \textbf{concave/convex plumbing} if it is a strong symplectic cap/filling of its boundary. A concave plumbing is also called a \textbf{divisor cap} of its boundary. Let $ Q_D $ be the intersection matrix of $ D $ and $ a=(C_i\cdot [\omega])\in (\RR_+)^k $ be the area sequence of $ D $. A symplectic divisor $ D $ is said to satisfy positive (resp. negative) \textbf{GS criterion} if there exists $ z\in (\RR_+)^k $ (resp. $ (\RR_{\le 0})^k $) such that $ Q_D z=a $.
	\begin{thm}[\cite{LiMa14-divisorcap},\cite{GaSt09}]\label{divisorcap}
		Let $ (D,\omega ) $ be an $ \omega $-orthogonal symplectic divisor. Then $ D $ has a concave (resp. convex) plumbing if $ (D,\omega ) $ satisfies the positive (resp. negative) GS criterion. Moreover, the induced contact boundary is unique up to contactomorphism.
	\end{thm}
	We will not recall the detailed construction here, but refer the readers to \cite{GaSt09}, \cite{LiMa14-divisorcap} and \cite{LiMi-ICCM}.
	\begin{lemma}\label{lemma:T^3=toric}
		A symplectic log Calabi-Yau divisor $ D\subset (X,\omega) $ has boundary diffeomorphic to $ T^3 $ if and only if it is toric.
	\end{lemma}
	\begin{proof}
		By \cite{GoLi14}, we have $ Coker(Q_D)=Coker(A(D)-I) $, where $ A(D) $ is the monodromy of boundary torus bundle of $ D $. If $ Q_D $ is negative definite, then $ A(D)\neq I $ and $ Y_D $ cannot be $ T^3 $.
		If $ Q_D $ is negative semi-definite, then all components of $ D $ has self-intersection $ -2 $. One can compute the monodromy to be $\begin{pmatrix}
		k+1 & k\\ -k & -k+1
		\end{pmatrix}$, which is conjugate to $ \begin{pmatrix}
		1 & 0 \\ -k & 1
		\end{pmatrix} $, where $ k $ is the length of $ D $. So the boundary cannot be $ T^3 $.
		
		If $ b^+(Q_D)=1 $, then $ D $ has self-intersection sequence $ s(D) $ to be one listed in Theorem 1.3 of \cite{LiMaMi-logCYcontact}, up to toric blow-ups and blow-downs. Such equivalence is called a toric equivalence and preserves $ q(D) $. We can compute their $ q(D) $ by hand as follows.
		\begin{itemize}
			\item For $ s(D)=(1,p) $ or $ (-1,-p) $ with $ p=1,2,3 $, we have $ q(D)=6\pm (1+p)\ge 2 $.
			\item For $ s(D)=(1,1,p) $ with $ p\le 1 $, we have $ q(D)=3-(2+p)\ge 0 $ and $ q(D)=0 $ if and only if $ p=1 $.
			\item For $ s(D)=(0,p) $ with $ p\le 4 $, we have $ q(D)=6-p\ge 2 $.
			\item For $ s(D)=(1,p) $ with $ p\le -1 $, we have $ q(D)=6-1-p\ge 6 $.
			\item The sequence $ s(D)= (1,1-p_1,p_2,\dots,p_{l-1},1-p_l)$ with $ p_i\ge 2, l\ge 2$ is called blown-up if it is non-toric blow-up of a divisor $ D' $, where $ D' $ is toric equivalent to $ D'' $ with $ s(D'')=(1,1,1) $. Since non-toric blow-up increases $ q(D) $, so we have $ q(D)>q(D')=0 $.
		\end{itemize}
		So $ q(D)=0 $ if and only if $ D $ is toric equivalent to $ (1,1,1) $. And by computing the monodromy of $ Y_D $ listed above, we see that $ Y_D $ is $ T^3 $ if and only if $ D $ is toric equivalent to $ (1,1,1) $.
	\end{proof}

	\section{$ c_1 $-nef cone and counting symplectic log Calabi-Yau divisors}\label{section:counting}
	Now we investigate the enumerative aspect of symplectic log Calabi-Yau divisors. From now on, we make the convention that the homology class $H-\sum_{i\in I} E_i$ can be denoted by $H_I$. For example, $H-E_1-E_2-E_3$ can be written as $H_{123}$.
	\subsection{$ c_1 $-nef subcone and counting}\label{section:c1-nef}


	To ensure the existence of log Calabi-Yau divisors, the symplectic forms must satisfy $ c_1(M_l,\omega )\cdot [\omega ]>0 $. In the case of symplectic rational surfaces, $ c_1(M_l,\omega ) $ is unique up to orientation-preserving diffeomrophism (\cite{LiLiu01-uniqueness}). Since the reduced symplectic cone $ \tilde{P}_l $ is the fundamental domain inside the symplectic cone under the action of orientation-preserving diffeomorphisms, we have a unique first Chern class for symplectic classes in $ \tilde{P}_l $. We denote this class by $ c_1(M_l) $ and define the \textbf{$ c_1 $-nef} subcone to be \[
	\tilde{N}_l=\tilde{N}(M_l)=\{ A\in H^2(M_l;\RR)| PD(A)\text{ is reduced, }A\cdot c_1(M_l)>0 \},
	\]
	and the normalized $ c_1 $-nef subcone $ N_l $ to contain those normalized reduced classes in $ \tilde{N}_l $.
	We can see below that $ \tilde{N}_l $ is indeed a subcone of the reduced symplectic cone $ \tilde{P}_l $ and thus $ N_l $ a subcone of $ P_l $.
	\begin{lemma}\label{lemma:nef cone}
		For $ A\in \tilde{N}_l $, we have $ A\cdot A>0 $ and $ A $ is a symplectic class.
	\end{lemma}
	\begin{proof}
		When $ l\le 9 $, it is previously known that any reduced class has positive square. We give a brief argument here. Let \[
		A=H-\delta_1E_1-\dots -\delta_9E_9
		\]
		be a normalized reduced class in $ M_l $, where $ \delta_i= 0 $ for $ i>l $. Then we have $ A\cdot (H-E_i-E_j-E_k)\ge 0 $ for any distinct $ i,j,k $ since $ 1\ge \delta_i+\delta_j+\delta_k $. Let \[
		B=\delta_1(H-E_1-E_2-E_3)+\delta_4(H-E_4-E_5-E_6)+\delta_7(H-E_7-E_8-E_9),
		\]
		and we have $ A\cdot B\ge 0 $. Consider \begin{align*}
		A-B=&(1-\delta_1-\delta_4-\delta_7)H+(\delta_1-\delta_2)E_2+(\delta_1-\delta_3)E_3\\&+(\delta_4-\delta_5)E_5+(\delta_4-\delta_6)E_6+(\delta_7-\delta_8)E_8+(\delta_7-\delta_8)E_9.
		\end{align*}
		Then we have $ A\cdot (A-B)\ge 0 $ and thus $ A\cdot A\ge 0 $. Since $ A\cdot c_1>0 $, we must have $ A\cdot A>0 $ by light cone lemma (\cite{McDuff97-gromov}).
		
		When $ l\ge 10 $, we use Lagrangian multiplier to find minimum of $ A\cdot A $. Let $ f(\delta_1,\dots,\delta_l)=A\cdot A=1-\sum \delta_i^2 $ and $ g(\delta_1,\dots,\delta_l)=A\cdot c_1(X)=3-\sum \delta_i $. By Lagrangian multiplier, we know $ f $ attains minimum when $ \delta_1=\dots=\delta_l=\delta $. Since $ A\cdot c_1(X)>0 $, we have $ \delta<\dfrac{3}{l} $. Then $ A\cdot A\ge min(f)=1-l\delta^2>1-\dfrac{9}{l}>0 $ as $ l\ge 10 $.
		
		Since $ A $ is reduced, we have $ A $ is a symplectic class by Proposition 4.9 of \cite{LiLiu01-uniqueness}.
	\end{proof}
	
	Since $ c_1(M_l) $ is fixed by orientation-preserving diffeomorphisms, the $ c_1 $-nef cone $ N_l $ is also preserved. This provides a natural and simplified domain for the existence of log Calabi-Yau divisors. One advantage of this $ c_1 $-nef cone over the larger reduced symplectic cone is that $ c_1 $-nef cone is a  convex linear cone with finitely many faces while the reduced symplectic cone can have boundaries cut out by quadratic equations (\cite{LiLiWu19-symp}).
	
	%

	We are interested in counting the number of symplectic log Calabi-Yau divisors in a symplectic rational surface $ (M_k,\omega ) $ with $ [\omega]\in N_k $, the normalized $ c_1 $-nef subcone. Recall Corollary \ref{lemma:elliptic case} shows that there is always a unique elliptic log Calabi-Yau divisor in the $ c_1 $-nef cone. So we will only consider cycles of spheres in the following subsections and automatically plus $ 1 $ to each counting result.
	
	Proposition \ref{prop:realization} implies that counting symplectic log Calabi-Yau divisors is equivalent to counting log Calabi-Yau homological configurations up to symplectomorphisms. But for a symplectic class in the interior of the reduced symplectic cone, the only orientation-preserving diffeomorphism preserving the symplectic class is the identity. So we have the following convenient lemma.
	\begin{lemma}\label{lemma:divisor = homology type}
		For a symplectic class $ [\omega] $ in the interior of the reduced symplectic cone, we have that $ D_{K_0,[\omega]}=\{id\} $.
		In particular, two symplectic log Calabi-Yau divisors $ D=\bigcup C_i $ and $ D'=\bigcup C_i' $ are strictly symplectic deformation equivalent if and only if $ [C_i]=[C_i'] $ for all $ i$, up to a cyclic or anti-cyclic relabelling of the components. In other words, $p\mathcal{HLCY}(X,\omega)=\mathcal{HLCY}(X,\omega)$.
	\end{lemma}

	The linear boundaries of the reduced cone are given by equations of the form $ \delta_i=\delta_j  $ or $ \delta_i+\delta_j+\delta_k=1 $. Then the class $ E_i-E_j $ or $ H-E_i-E_j-E_k $ is represented by a Lagrangian sphere (\cite{LiWu12-lagrangian}) and the Dehn twist along this Lagrangian sphere is a symplectomorphism generating a strict homological equivalence between some symplectic log Calabi-Yau divisors, which causes repetitions of log Calabi-Yau homological configurations.
	
	So in order to count symplectic log Calabi-Yau divisors, it suffices to count Calabi-Yau homological configurations for each symplectic structure and subtract repetitions along the boundary of the normalized reduced symplectic cone. This is the strategy we will employ in the rest of this section.
	
	\subsection{Count of minimal models}
	We start by counting symplectic log Calabi-Yau divisors in $\CC\PP^2, \CC\PP^2\#\overline{\CC\PP}^2$ and $S^2\times S^2$.
	They are all minimal and their count can be read off from Theorem \ref{thm:minimal model} or Lemma \ref{lemma:minimal surjectivity} as follows.
	
	On $ \CC\PP^2 $, the only normalized symplectic class is $ [\omega_{FS}] $ such that $ \omega_{FS}(H)=1 $.
	On $\CC\PP^2\#\overline{\CC\PP}^2$, we consider the normalized symplectic form $ \omega_\delta $ with $\omega_{\delta}(H)=1, \omega_{\delta}(E)=\delta<1 $. Similarly on $S^2\times S^2$, we consider $ \omega_\mu $ with $  \omega_{\mu}(B)=1, \omega_{\mu}(F)=\mu\geq 1$. Then we have the following counts.
	
	\begin{lemma}
		\begin{align*}
		| \mc{LCY}(\CC\PP^2,\omega_{FS}) |&=3,\\
		| \mc{LCY}(S^2\times S^2, \omega_\mu)|&=\begin{cases}
		1+(1+\lceil \mu +1 \rceil)+\lceil \mu \rceil+\lceil \mu \rceil=3\lceil \mu \rceil +3 , \text{when}\,\,\, \mu>1\\
		1+\lceil \mu +1 \rceil+\lceil \mu \rceil+\lceil \mu \rceil=3\lceil \mu \rceil +2 , \text{when}\,\,\, \mu=1
		\end{cases}\\ \
		| \mc{LCY}(\CC\PP^2\#\overline{\CC\PP}^2, \omega_{\delta})|&=1+(1+\lceil \dfrac{1}{1-\delta} \rceil)+\lceil \dfrac{1}{1-\delta} \rceil+\lceil \dfrac{\delta}{1-\delta} \rceil=3\lceil \dfrac{\delta}{1-\delta} \rceil+4.
		\end{align*}
	\end{lemma}
	
	\begin{proof}
		Their $p\mathcal{HLCY}$ spaces has been listed in Lemma \ref{lemma:minimal surjectivity}. Note that those homological configurations can be realized if and only if all the homology classes have positive symplectic areas.
		
		For $\CC\PP^2$, there is nothing to check since the symplectomorphism must act trivially on the homology group.
		
		For $ S^2\times S^2 $, if $\mu>1$, then there is no Lagrangian $(-2)$-sphere so that any symplectomorphism must act trivially on the homology group (\cite{LiWu12-lagrangian}). This implies the space $p\mathcal{HLCY}$ is the same as the space $\mathcal{HLCY}$. So the count for the type $1,2,3,4$ in Lemma \ref{lemma:minimal surjectivity} should be $1,1+\lceil \mu +1 \rceil,\lceil \mu \rceil,\lceil \mu \rceil$ respectively. And the total count is $3\lceil \mu \rceil +3$. When $\mu=1$, the homological action by symplectomorphism group is generated by switching $F$ and $B$. Thus $(2F+B,B),(F+2B,F)$ will be identified after passing from $p\mathcal{HLCY}$ to its quotient $\mathcal{HLCY}$. In this case, the count should be $3\lceil \mu \rceil +2$.

		For $ \CC\PP^2\#\overline{\CC\PP}^2 $, the symplectomorphism must act trivially on the homology group. Thus we only need to count the space $p\mathcal{HLCY}$. The count for the type $1,2,3,4$ in Lemma \ref{lemma:minimal surjectivity} should be $1,1+\lceil \dfrac{1}{1-\delta} \rceil,\lceil \dfrac{1}{1-\delta} \rceil,\lceil \dfrac{\delta}{1-\delta} \rceil$ respectively. So the total count is given by $3\lceil \dfrac{\delta}{1-\delta} \rceil+4$.
		
	\end{proof}

	\subsection{Counting symplectic log Calabi-Yau divisors in $ M_2 $ and $M_3$}
	The remaining rational symplectic surfaces are $ M_l=\CC\PP^2\# l\overline{\CC\PP}^2,l\geq2 $ with normalized reduced symplectic form $ \omega $ given by the reduced basis $ (\delta_1,\dots,\delta_l) $, where $ [\omega] $ lies in the normalized $ c_1 $-nef cone $ N_l $.
	In the subsection, we give a detailed count of symplectic log Calabi-Yau divisors in $ M_2 $ and toric symplectic log Calabi-Yau divisors in $M_3$ for all symplectic forms in the $ c_1 $-nef cone. Our counting in $ M_2 $ and also $ M_3 $ depends crucially on the following homological classification of negative symplectic spheres of log Calabi-Yau divisors by Corollary \ref{cor:ingredients}. It also follows from the Proposition 3.4 in \cite{LiLi20-pacific}.

	\begin{prop} \label{prop:spherical classes}
		All possible negative symplectic spherical classes in $ M_2 $ and $ M_3 $ can be enumerated as follows.
		
		For $M_2$, we have that	for integer $ k\ge 1 $, \begin{itemize}
			\item $ (-1) $ classes are $E_1, E_2, H-E_1-E_2$,
			\item $ (-2k-1) $ class is  $-kH+(k+1)E_1$,
			\item $ (-2k) $ class is $-(k-1)H+kE_1-E_2$.
		\end{itemize}
		For $M_3$, we have that for integer $ k\ge 1 $, \begin{itemize}
			\item $ (-1) $ classes are $E_1, E_2, E_3, H-E_1-E_2, H-E_1-E_3, H-E_2-E_3;$
			\item $ (-2) $ classes are $E_1-E_2, E_1-E_3, E_2-E_3, H-E_1-E_2-E_3;$
			\item $ (-2k-1) $ classes are $-kH+(k+1)E_1, -(k-1)H+kE_1-E_2-E_3;$
			\item $ (-2k-2) $ classes are $-kH+(k+1)E_1-E_2, -kH+(k+1)E_1-E_3.$
		\end{itemize}
		Moreover, these classes can be realized by symplectic spheres if and only if their pairings with the symplectic class are positive.
	\end{prop}
	
	The normalized $ c_1 $-nef cone $ N_2 $ coincides with the reduced symplectic cone $ P_2 $ and is the convex hull in $ \RR^2 $ of the points $ O=(0,0),M=(\dfrac{1}{2},\dfrac{1}{2}) $ and $ A=(1,0) $ with line segments $ OA $ and $MA$ removed. The cone $ N_2 $ is cut by hyperplanes corresponding to negative sphere classes Proposition \ref{prop:spherical classes} into infinitely many regions, which we describe as follows.
	
	Consider points $P_i=(\dfrac{i}{i+1},\dfrac{1}{i+1})$ on edge $MA$ and $Q_i=(\dfrac{i}{i+1},0)$ on edge $OA$, where $i\geq 1$ is an integer. And we also let $ Q_0=O $. The region $\triangle P_{i}P_{i+1}Q_{i}$ contains the interior of this triangle and the interior of the edge $Q_{i}P_{i+1}$. The region $\triangle Q_{i-1}Q_{i}P_{i}$ contains the interior of this triangle and the interior of the edge $Q_{i}P_{i}$.
	Then we have \[
	N_k=(\bigsqcup_{i=1}^\infty \triangle P_{i}P_{i+1}Q_{i})\sqcup(\bigsqcup_{i=1}^\infty  \triangle Q_{i-1}Q_{i}P_{i})\sqcup OM.
	\]
	
	\begin{prop}\label{prop:general count M2}
		The count of log Calabi-Yau divisors  $ \mc{LCY}(2;1, \delta_1, \delta_2) $ in $ (M_2;1,\delta_1,\delta_2 ) $ is given by \begin{enumerate}[label=$ (\arabic*) $]
			\item $ 14i+19 $ in region $\triangle P_iP_{i+1}Q_i ,i\ge 1$,
			\item $ 14i+26 $ in region $ \triangle Q_iQ_{i+1}P_{i+1}, i\ge 0 $,
			\item $ 13 $ in the interior of the line segment $ OM $.
		\end{enumerate}
		In other words, in the interior of reduced symplectic cone, we have \[
		|\mc{LCY}(2;1,\delta_1,\delta_2)|=7(\lceil\dfrac{\delta_1}{1-\delta_1}\rceil +\lceil \dfrac{\delta_1-\delta_2}{1-\delta_1}\rceil)+12,
		\]
		On the boundary of reduced symplectic cone, that is, when $ 0<\delta_1=\delta_2<\dfrac{1}{2} $, we have $ |\mc{LCY}(2;1,\delta_1,\delta_2)|=13 $ . The results are shown in Figure \ref{fig:M2}.
	\end{prop}
	\begin{proof}
		For $ n\ge 1 $, we can enumerate all possible self-intersection sequences for log Calabi-Yau divisors in $ M_2 $ as follows:
		\begin{enumerate}[label=$ (\arabic*) $]
			\item $ (n+3,-n), $
			\item $ (n+2,-1,-n),(n+1,0,-n) $,
			\item $ (0,n,-1,-n),(0,n-1,0,-n),(n+1,-1,-1,-n) $,
			\item $ (0,n-1,-1,-1,-n)$,
			\item $ (3,0),(2,1),(1,0,0),(1,-1,1), $
			\item $(7)$.
		\end{enumerate}
		
		For each normalized reduced vector $ (1,\delta_1,\delta_2) $, we denote by $ F_n(2;1,\delta_1,\delta_2),n\ge 1 $, the total count of log Calabi-Yau divisors with self-intersection sequence to be one of $ (1)-(4) $ and denote by $ F_0(2;1,\delta_1,\delta_2) $ the count corresponding to $ (5) $. We will also denote by $ \# s $ the count of log Calabi-Yau divisors with self-intersection sequence $ s $. By Proposition \ref{prop:spherical classes}, we consider cases $ n=2k $ and $ n=2k+1 $ with $ k\ge 1 $ for each $ s $ listed above.
		
		For $ n=2k $, we have that all possible homology types are \begin{enumerate}[label=$ (\arabic*) $]
			\item $ ((k+2)H-(k+1)E_1,-(k-1)H+kE_1-E_2 )$
			\item $ ((k+2)H-(k+1)E_1-E_2,E_2,-(k-1)H+kE_1-E_2),\\((k+1)H-kE_1,H-E_1,-(k-1)H+kE_1-E_2) $
			\item $ (H-E_1,(k+1)H-kE_1-E_2,E_2,-(k-1)H+kE_1-E_2),\\(H-E_1,kH-(k-1)E_1,H-E_1,-(k-1)H+kE_1-E_2),\\((k+1)H-kE_1,H-E_1-E_2,E_2,-(k-1)H+kE_1-E_2) $
			\item $ (H-E_1,kH-(k-1)E_1,H-E_1-E_2,E_2,-(k-1)H+kE_1-E_2) $
		\end{enumerate}
		These homology types can be realized by log Calabi-Yau divisors if and only if the symplectic area of each component is positive.
		So we have \[
		F_{2k}(2;1,\delta_1,\delta_2)=\begin{cases}
		7\quad \text{if } -(k-1) + k\delta_1-\delta_2>0\\
		0 \quad \text{otherwise}
		\end{cases}\]
		
		
		For $ n=2k+1 $, all possible homology types are \begin{enumerate}[label=$ (\arabic*) $]
			\item $ ((k+3)H-(k+2)E_1-E_2,-kH+(k+1)E_1 )$
			\item $ ((k+2)H-(k+1)E_1,H-E_1-E_2,-kH+(k+1)E_1),\\((k+2)H-(k+1)E_1-E_2,H-E_1,-kH+(k+1)E_1) $
			\item $ (H-E_1,(k+1)H-kE_1,H-E_1-E_2,-kH+(k+1)E_1),\\(H-E_1,(k+1)H-kE_1-E_2,H-E_1,-kH+(k+1)E_1),\\((k+2)H-(k+1)E_1-E_2,E_2,H-E_1-E_2,-kH+(k+1)E_1) $
			\item $ (H-E_1,(k+1)H-kE_1-E_2,E_2,H-E_1-E_2,-kH+(k+1)E_1) $
		\end{enumerate}
		So we have \[
		F_{2k+1}(2;1,\delta_1,\delta_2)=\begin{cases}
		7 \quad \text{if }-k + (k+1)\delta_1>0\\
		0\quad \text{otherwise}
		\end{cases}\]
		
		For $ n=1 $, we have the following homology types:\begin{enumerate}[label=$ (\arabic*) $]
			\item $ (4,-1)=(2H , H- E_1- E_2),$ $ (3H-2E_p-E_q,E_p) $,
			\item $ (3,-1,-1)=(2H-E_i,E_i,H-E_1-E_2),\\ (2,0,-1)=(2H-E_1-E_2,H-E_i,E_i) $,
			\item $ (0,1,-1,-1)=(H-E_i,H,H-E_1-E_2,E_i), \\  (0,0,0,-1)=(H-E_p,H-E_q,H-E_p,E_p), \\  (2,-1,-1,-1)=(2H-E_1-E_2,E_1,H-E_1-E_2,E_2) $,
			\item $ (0,0,-1,-1,-1)=(H-E_2,H-E_1,E_1,H-E_1-E_2,E_2) $,
		\end{enumerate}
		where $ i\in \{1,2\} $ and $ \{p,q\}=\{1,2\} $. So
		\[
		F_1(2;1,\delta_1,\delta_2)=\begin{cases}
		13, \quad \text{if }\delta_1>\delta_2;\\
		8,\quad \text{if } \delta_1=\delta_2.
		\end{cases}
		\]
		For those $ s $ in $ (5) $, we could also list their homology types as follows:
		\begin{itemize}
			\item $ (3,0)=(2H-E_p,H-E_q) $
			\item $ (2,1)=(2H-E_1-E_2,H) $
			\item $ (1,0,0)=(H,H-E_1,H-E_2) $
			\item $ (1,-1,1)=(H,H-E_1-E_2,H) $
		\end{itemize}
		where $ \{p,q\}=\{1,2\} $. So \[
		F_0(2;1,\delta_1,\delta_2)=\begin{cases}
		5, \quad \text{if }\delta_1>\delta_2;\\
		4,\quad \text{if } \delta_1=\delta_2.
		\end{cases}
		\]
		
		Also by Corollary \ref{lemma:elliptic case}, there is always a unique divisor in $ (7) $. Then we have \[
		|\mc{LCY}(2;1,\delta_1,\delta_2)|=1+ F_0(2;1,\delta_1,\delta_2)+F_1(2;1,\delta_1,\delta_2)+\sum\limits_{n=2}^\infty F_n(2;1,\delta_1,\delta_2)
		\]
		
		The result follows by recognizing \begin{align*}
		\triangle P_iP_{i+1}Q_i=& \{ -i+(i+1)\delta_1>0, -i+(i+1)\delta_1-\delta_2\le 0 \},\\
		\triangle Q_iQ_{i+1}P_{i+1}=& \{ -i+(i+1)\delta_1-\delta_2>0, -(i+1)+(i+2)\delta_1\le 0 \},\\
		OM = & \{ \delta_1=\delta_2 \}.
		\end{align*}
		Also we get the formula because when $ \delta_1>\delta_2 $ we have  \[
		\lceil\dfrac{\delta_1}{1-\delta_1}\rceil +\lceil \dfrac{\delta_1-\delta_2}{1-\delta_1}\rceil= \begin{cases}
		2i+1 \quad \text{if }-i+(i+1)\delta_1>0, -i+(i+1)\delta_1-\delta_2\le 0,\\
		2i+2 \quad \text{if } -i+(i+1)\delta_1-\delta_2>0, -(i+1)+(i+2)\delta_1\le 0.
		\end{cases}
		\]
	\end{proof}
	
	\begin{figure}[h]
		\centering
		\begin{minipage}{.5\textwidth}
			\centering
			\includegraphics[width=\linewidth]{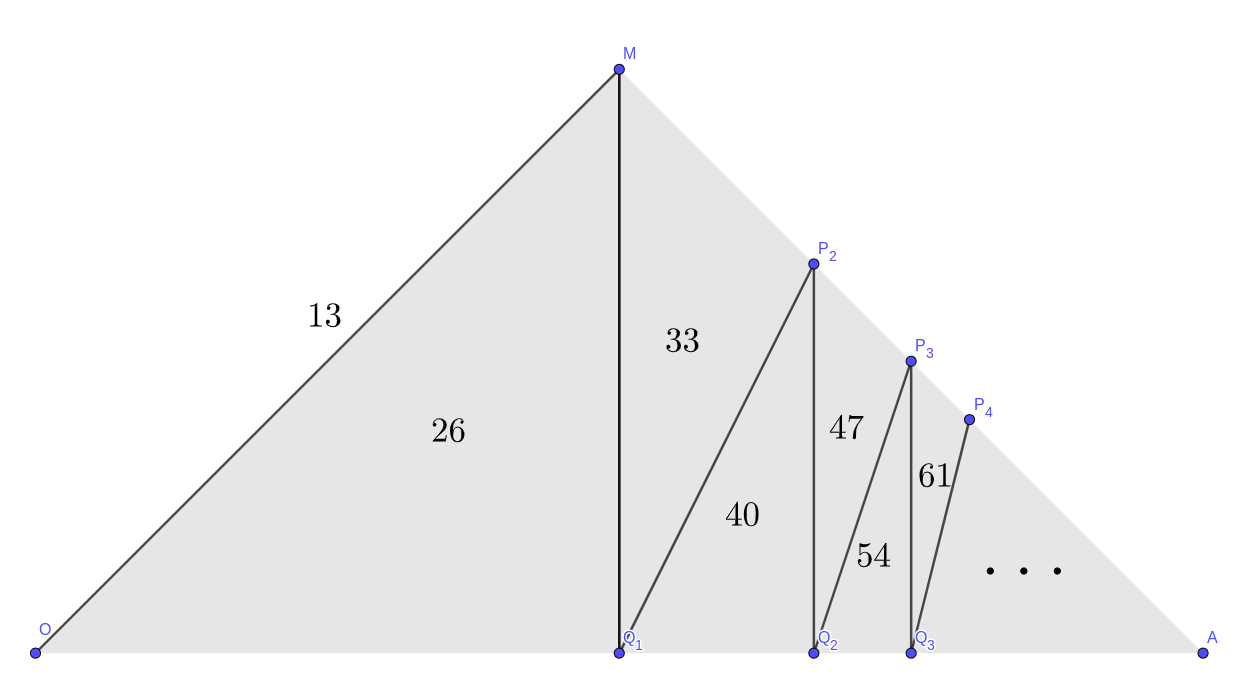}
			\caption{Counting of $\mathcal{LCY}$ for $M_2$.\label{fig:M2}}
		\end{minipage}%
		\begin{minipage}{.5\textwidth}
			\centering
			\includegraphics[width=\linewidth]{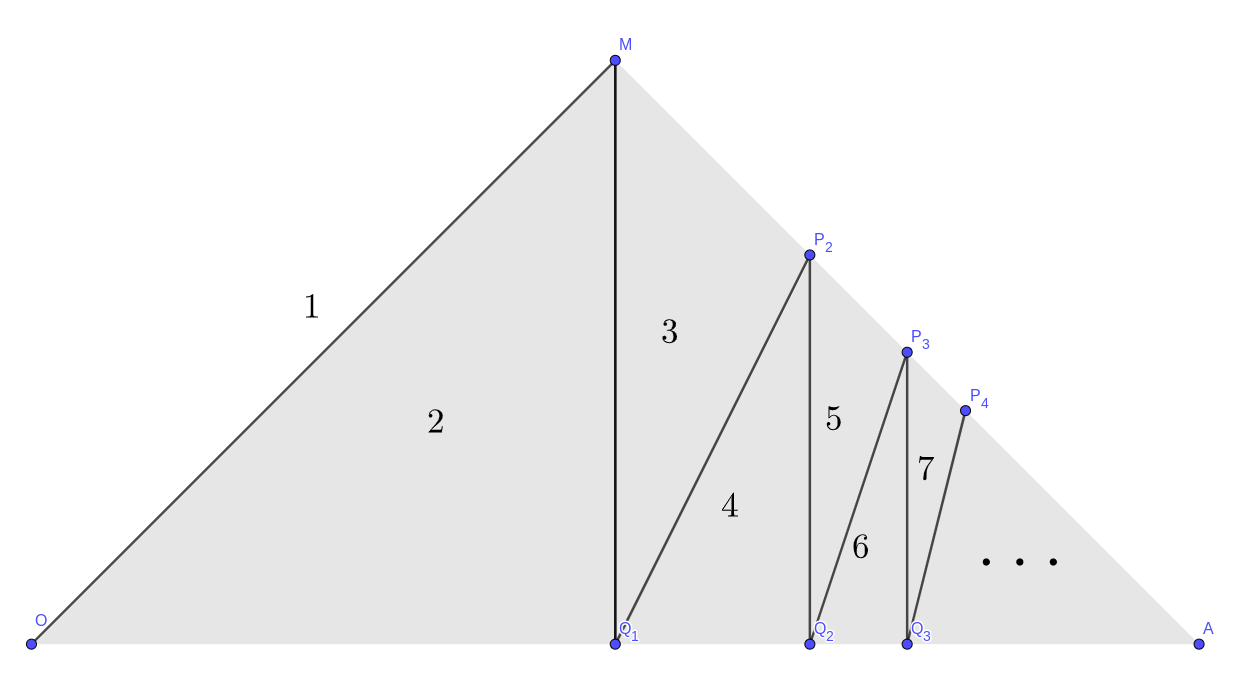}
			\caption{Counting of $t\mathcal{LCY}$ for $M_2$.\label{fig:M2toric}}
		\end{minipage}
	\end{figure}
	
	By only counting the toric log Calabi-Yau divisors in the proof of the above proposition, we get the following count of toric divisors.
	\begin{corollary}\label{cor:toric count M2}
		The count of toric log Calabi-Yau divisors on $ (M_2;1,\delta_1,\delta_2) $ is given by \begin{enumerate}[label=$ (\arabic*) $]
			\item $2i+1$ in the region $\triangle Q_{i}P_{i}P_{i+1}$;
			\item  $2i+2$  in the region $\triangle Q_{i}Q_{i+1}P_{i+1}$;
			\item $1$ in the interior of edge $OM=Q_{0}P_{1}$.
		\end{enumerate}
		In other words, we have \[
		|t\mc{LCY}(M_2;1,\delta_1,\delta_2)|=\lceil\dfrac{\delta_1}{1-\delta_1}\rceil +\lceil \dfrac{\delta_1-\delta_2}{1-\delta_1}\rceil
		\]The results are shown in Figure \ref{fig:M2toric}.
	\end{corollary}
	\begin{proof}
		The toric symplectic log Calabi-Yau divisors has self-intersection sequence $ (0,n-1,-1,-1,-n) $. For $ n=2k $, the only homology type is $ (H-E_1,kH-(k-1)E_1,H-E_1-E_2,E_2,-(k-1)H+kE_1-E_2) $. So the number of toric symplectic log Calabi-Yau divisors is \[
		tF_{2k}(2;1,\delta_1,\delta_2)=1 \quad \text{if } -(k-1) + k\delta_1-\delta_2>0,\delta_1+\delta_2<1.
		\]
		For $ n=2k+1 $, the only homology type is $ (H-E_1,(k+1)H-kE_1-E_2,E_2,H-E_1-E_2,-kH+(k+1)E_1) $. Then \[
		tF_{2k+1}(2;1,\delta_1,\delta_2)=1 \quad \text{if }-k + (k+1)\delta_1>0,\delta_1+\delta_2<1.
		\]
		For $ n=1 $, the only homology type is $ (H-E_2,H-E_1,E_1,H-E_1-E_2,E_2) $. Then \[
		tF_1(2;1,\delta_1,\delta_2)=1 \quad \text{if }\delta_1>\delta_2,\delta_1+\delta_2<1.
		\]
		Above counts are $ 0 $ everywhere else. We get the results by summing over the counts.
	\end{proof}

	Next we give a detailed count of toric symplectic log Calabi-Yau divisors in $ M_3 $ for all symplectic forms in the normalized $ c_1 $-nef cone $ N_3 $.
	
	The normalized $ c_1 $-nef cone $ N_3 $ is the same as the reduced symplectic cone $ P_3 $ and is the convex hull in $ \RR^3 $ of the points $ O=(0,0,0),M=(\dfrac{1}{3},\dfrac{1}{3},\dfrac{1}{3}) , A=(1,0,0) $ and $ B=(\dfrac{1}{2},\dfrac{1}{2},0) $ with the triangle $ OAB $ removed. By Proposition \ref{prop:spherical classes}, there is a pattern of four walls in $ N_3 $ given by $-(k-1)H+kE_1, -(k-1)H+kE_1-E_2-E_3, -kH+(k+1)E_1-E_2, -kH+(k+1)E_1-E_3$. They never intersect with each other except on the edges. The cone $ N_3 $ is again cut by these walls into infinitely many regions, which we describe as follows.
	
	Consider points $P_i=(\dfrac{i+1}{i+3},\dfrac{1}{i+3},\dfrac{1}{i+3})$ on edge $MA$, $Q_i=(\dfrac{i}{i+1},0,0)$ on edge $OA$ and $R_{i}=(\dfrac{i+1}{i+2},\dfrac{1}{i+2},0)$ on edge $AB$, where $i\geq 0$ is an integer. The regions  $$Q_{i-1}R_{i-1}P_{2i-1}P_{2i-2}, \quad Q_{i-1}Q_{i}R_{i-1}P_{2i-1}, \quad Q_{i}R_{i-1}P_{2i-1}P_{2i}, \quad Q_{i}R_{i-1}R_{i}P_{2i}$$ contain the interior subcones and the interior of their faces $$Q_{i-1}R_{i-1}P_{2i-2},\quad Q_{i}R_{i-1}P_{2i-1}, \quad Q_{i}R_{i-1}P_{2i}, \quad Q_{i}R_{i}P_{2i}$$ respectively; the regions $$Q_{i-1}P_{2i-2}P_{2i-1}, \quad Q_{i-1}Q_{i}P_{2i-1}, \quad Q_{i}P_{2i-1}P_{2i}, R_{i-1}P_{2i-2}P_{2i-1}, \quad R_{i-1}P_{2i-1}P_{2i}, \quad R_{i-1}R_{i}P_{2i}$$ contain the interior triangles and the interior of their edges $$Q_{i-1}P_{2i-1}, \quad Q_{i}P_{2i-1}, \quad Q_{i}P_{2i}, \quad R_{i-1}P_{2i-1}, \quad R_{i-1}P_{2i}, \quad R_{i}P_{2i}$$ respectively; the region $P_{i-1}P_{i}$ contains the interior of this edge and the point $P_{i}$.
	
	These regions form the entire normalized reduced symplectic cone $ P_3=N_3 $, so we have listed all the possible symplectic classes on $M_3$. We claim the following counting results, which follows from the same method for $M_2$ cases. The proof is quite long, thus we put it in the appendix \ref{section:detail1}.
	
	\begin{prop}\label{prop:M3}
		Let $i\geq1$ be an integer. The number of toric symplectic log CY divisors on the rational surface $M_3$ is
		
		(1) $10i-2$ in the region $Q_{i-1}R_{i-1}P_{2i-1}P_{2i-2}$
		(2) $10i$  in the region $Q_{i-1}Q_{i}R_{i-1}P_{2i-1}$
		
		(3) $10i+3$ in the region $Q_{i}R_{i-1}P_{2i-1}P_{2i}$
		(4) $10i+5$ in the region $Q_{i}R_{i-1}R_{i}P_{2i}$
		
		(5) $4i-1$ in the region $Q_{i-1}P_{2i-2}P_{2i-1}$
		(6) $4i$ in the region $Q_{i-1}Q_{i}P_{2i-1}$
		
		(7) $4i+1$ in the region $Q_{i}P_{2i-1}P_{2i}$
		(8) $4i$ in the region $R_{i-1}P_{2i-2}P_{2i-1}$
		
		(9) $4i+2$ in the region $R_{i-1}P_{2i-1}P_{2i}$
		(10) $4i+3$ in the region $R_{i-1}R_{i}P_{2i}$
		
		(11) $i+1$ in the region $P_{i-1}P_{i}$
		(12) $3$ in the region $OBM=Q_{0}R_{0}P_{0}$
		
		(13) $2$ in the region $BM=R_{0}P_{0}$
		(14) $1$ in the region $OM=Q_{0}P_{0}$ and the monotone point $M=P_{0}=(\dfrac{1}{3},\dfrac{1}{3},\dfrac{1}{3})$.
		The results are shown in Figure \ref{fig:M3}.

	\end{prop}
	
	\begin{figure}
		\centering
		\begin{minipage}{.5\textwidth}
			\centering
			\includegraphics[width=\linewidth]{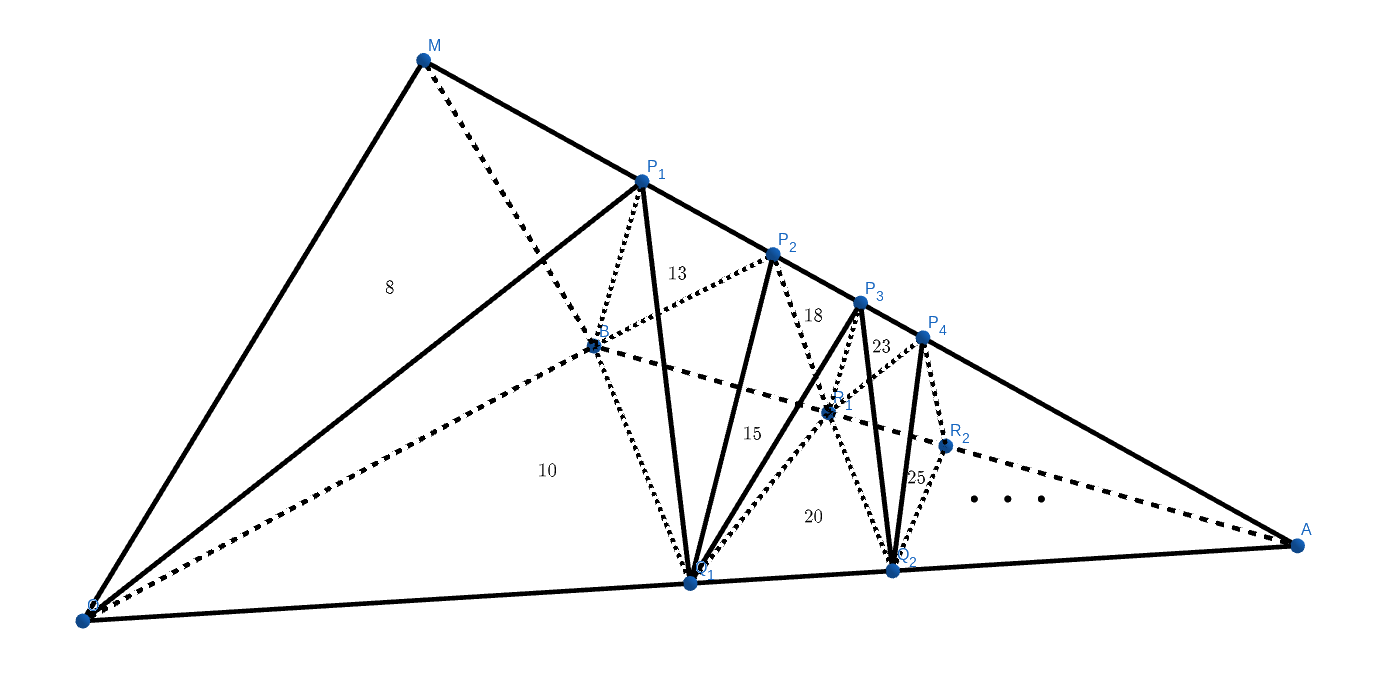}
		\end{minipage}%
		\begin{minipage}{.5\textwidth}
			\centering
			\includegraphics[width=\linewidth]{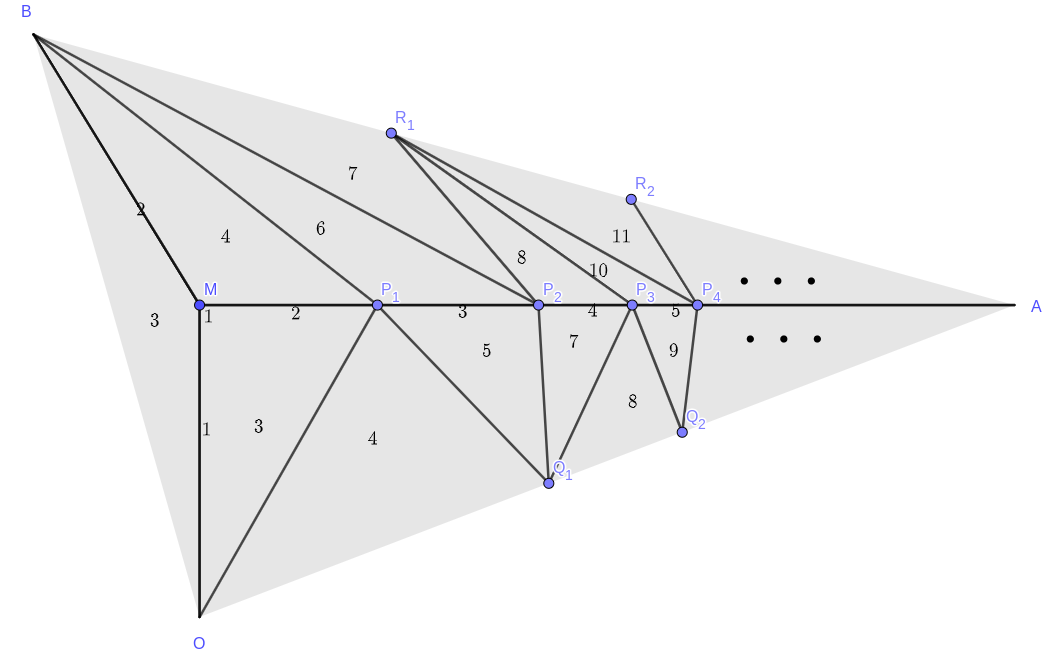}
		\end{minipage}
		\caption{Counting of $t\mathcal{LCY}$ for $M_3$ in the interior and boundary of the $c_1$-nef cone. \label{fig:M3}}
	\end{figure}


	

	\subsection{Toric regions}\label{section:toric cone}
	
	Instead of giving a counting formula of toric symplectic log Calabi-Yau divisors, we could also investigate the symplectic classes for which they could exist. This is called the {\bf toric region} $T_l$ of the manifold $ M_l $. By Theorem \ref{thm:toric=tLCY} that will be proved in the next section, this will also be the part of the normalized $c_1$-nef cone where there exists a toric action. From the counting results for $M_3,M_2$ and minimal models, we already know the following:
	
	\begin{corollary}
		When $l\leq 3$, the toric region $T_l$ will be the whole normalized reduced symplectic cone $P_l$.
	\end{corollary}
	
	Now we discuss the cases for rational surfaces $M_4,M_5,M_6$. We start with the following simple observation:
	
	\begin{fact}\label{fact}
		For $4\leq l\leq 6$, the toric regions $T_l$ contains the interior of normalized reduced symplectic cones $P_l$.
		\begin{proof}
			Consider the following homological configurations for $M_4,M_5,M_6$ respectively:
			$$(H_{134},E_4,E_1-E_4,H_{12},E_2,H_{23},E_3)$$
			$$(H_{134},E_4,E_1-E_4,H_{12},E_2-E_5,E_5,H_{235},E_3)$$
			$$(H_{134},E_4,E_1-E_4,H_{125},E_5,E_2-E_5,H_{236},E_6,E_3-E_6)$$
			Note that they can be realized by Corollary \ref{cor:ingredients} since it is easy to see that each homology class in the above sequences has positive pairing with any $[\omega]$ in the interior of $P_l$. Also the realization of them are toric log Calabi-Yau since $q(D)=12-k(D)-D^2=12-(l+3)-(9-l)=0$.
			
		\end{proof}	
		
	\end{fact}
	
	Recall that the normalized reduced symplectic cone $ P_4 $ for $M_4$ is the convex hull of $O=(0,0,0,0)$, $M=(\dfrac{1}{3},\dfrac{1}{3},\dfrac{1}{3},\dfrac{1}{3})$, $A=(1,0,0,0)$, $B=(\dfrac{1}{2},\dfrac{1}{2},0,0)$ and $C=(\dfrac{1}{3},\dfrac{1}{3},\dfrac{1}{3},0)$ not containing the face $OABC$.
	
	\begin{prop}
		The toric region $T_4$ for $M_4$ is equal to $P_4$ deleting two edges $MO,MA$.
	\end{prop}
	\begin{proof}
		If the homological sequence for $M_4$ in Fact \ref{fact} can not be realized by some $[\omega]\in P_l$, then either $\omega(H_{134})=0$ or $\omega(E_1-E_4)=0$. They lead to the restrictions $\delta_1+\delta_2+\delta_3=1,\delta_2=\delta_3=\delta_4$ or $\delta_1=\delta_2=\delta_3=\delta_4$, which exactly correspond to the edge $MA$ or $MO$. Thus we only need to show the nonexistence of toric symplectic log Calabi-Yau divisors on $MA$ and $MO$.
		
		Firstly by considering the toric blow up operation on the self-intersection sequcences of $M_3$, we can classify all the possible self-intersection sequences for $M_4$ which are:
		\begin{align*}
		&\{A_n=(-2,-1,-2,-n,-1,-1,n-2)\}_{n=1}^{\infty}\\
		&\{B_n=(-1,-2,-1,-n-3,-1,-1,n)\}_{n=1}^{\infty}\\
		&\{C_n=(-1,-2,-2,-1,-n-3,0,n)\}_{n=1}^{\infty}\\
		&\{D_n=(-1,-3,-1,-2,n,0,-n-2)\}_{n=-\infty}^{\infty}\\
		&\{E_n=(-3,-1,-2,-2,n,0,-n-1)\}_{n=-\infty}^{\infty}
		\end{align*}
		
		Next by Proposition 3.4 of \cite{LiLi20-pacific}, for $k\geq1$, $\mc{S}_{\omega}^{-2k-2}$ and $\mc{S}_{\omega}^{-2k-3}$ are subsets of the following sets respectively, with the only restrictions on positive pairing with $\omega$:
		\begin{align*}
		\mc{S}_{\omega}^{-2k-2}\subset &\{ -(k-1)H+kE_1-E_2-E_3-E_4, -kH+(k+1)E_1-E_2,\\
		&-kH+(k+1)E_1-E_3, -kH+(k+1)E_1-E_4\}\\
		\mc{S}_{\omega}^{-2k-3} \subset &\{-kH+(k+1)E_1-E_2-E_3, -kH+(k+1)E_1-E_2-E_4,\\
		&-kH+(k+1)E_1-E_3-E_4, -(k+1)H+(k+2)E_1\}
		\end{align*}
		
		For $\mc{S}_{\omega}^{-1}$, $\mc{S}_{\omega}^{-2}$ and $\mc{S}_{\omega}^{-3}$, their possible choices are as follows, again with the postive pairing condition as the only restriction:
		\begin{align*}
		&\mc{S}_{\omega}^{-1}\subset \{E_k, H-E_i-E_j|1\leq k\leq  4,1\leq i< j\leq  4 \}\\
		&\mc{S}_{\omega}^{-2}\subset\{E_p-E_m, H-E_i-E_j-E_k|1\leq p<m \leq  4,1\leq i< j< k \leq  4 \}\\
		&\mc{S}_{\omega}^{-3}\subset \{-H+2E_1, H-E_1-E_2-E_3-E_4, E_i-E_j-E_k|1\leq i< j< k\leq  4 \}
		\end{align*}

		
		Now we can explain the nonexistence of toric symplectic log Calabi-Yau divisor when the symplectic class is on the edge $OM$ or $MA$, which is the consequence of the lack of some $ (-2) $- and $ (-3) $-spheres.
		
		Along the edge $MO$, we have \begin{align*}
		\mc{S}_{\omega}^{-2}&=\{H-E_i-E_j-E_k|1\leq i< j< k \leq  4 \}\\
		\mc{S}_{\omega}^{-3}&=\{H-E_1-E_2-E_3-E_4\}.
		\end{align*} From this we can see that any two elements in $\mc{S}_{\omega}^{-2}$ or $\mc{S}_{\omega}^{-3}$ have negative intersection number, which excludes the existence of toric symplectic log Calabi-Yau divisors of type $A_n, C_n, D_n, E_n$ since the any two $(-2)$- or $(-3)$-components in the divisor must have intersection number $1$ if they are adjacent in the sequence or $0$ otherwise. Moreover, $B_n$ type is also impossible since $\mc{S}_{\omega}^{\leq-4}=\emptyset$ for $ [\omega]\in MO $.
		
		Along the edge $MA$, we have \begin{align*}
		\mc{S}_{\omega}^{-2}&=\{E_1-E_2,E_1-E_3,E_1-E_4,H-E_2-E_3-E_4\}\\
		\mc{S}_{\omega}^{-3}&=\{-H+2E_1,E_1-E_2-E_3,E_1-E_2-E_4,E_1-E_3-E_4\}.
		\end{align*} Still any two elements in $\mc{S}_{\omega}^{-2}$ or $\mc{S}_{\omega}^{-3}$ have negative intersection number, which excludes divisors of type $A_n, C_n, D_n, E_n$. And $B_n$ type is also excluded because although $\mc{S}_{\omega}^{\leq-4}$ may not be empty, the intersection numbers between the elements in  $\mc{S}_{\omega}^{-2}$ and $\mc{S}_{\omega}^{\leq-4}$ are negative. However the $(-2)$-component and $(-n-3)$-component in $B_n$ must have intersection number $0$ since they are nonadjacent.
	\end{proof}
	
	A similar but more complicated analysis can be done for $ M_5 $. The normalized reduced symplectic cone $P_5$ for $M_5$ is the convex hull of
	\begin{align*}
	O=(0,0,0,0,0), M=(\dfrac{1}{3},\dfrac{1}{3},\dfrac{1}{3},\dfrac{1}{3},\dfrac{1}{3}),
	A=(1,0,0,0,0),\\
	B=(\dfrac{1}{2},\dfrac{1}{2},0,0,0), C=(\dfrac{1}{3},\dfrac{1}{3},\dfrac{1}{3},0,0),
	D=(\dfrac{1}{3},\dfrac{1}{3},\dfrac{1}{3},\dfrac{1}{3},0),
	\end{align*}
	not containing the face $OABCD$. Let's introduce a special point $X=(\dfrac{1}{2},\dfrac{1}{4},\dfrac{1}{4},\dfrac{1}{4},\dfrac{1}{4})$ on the edge $MA$. We have the following result, the proof of which is quite similar to the case for $M_4$ and contained in the appendix \ref{section:M5region}.
	\begin{prop}\label{prop:toricconeM5}
		The toric region $T_5$ for $M_5$ is equal to $P_5$ deleting the closed faces $MOD$, $MAD$ and $MOX$ which is a quarter of the face $MOA$.
		
	\end{prop}
	

	\begin{rmk}
		After we establish the equivalence between toric divisors and toric actions later, we can also check the above results from the perspective of performing equivariant blow up, by analysing the possibilities of chopping the corners of Delzant polygons.
	\end{rmk}
	
	Next we give some qualitative characterizations of the toric regions for general $M_l$.
	\begin{lemma}\label{lemma:star shaped}
		The toric region $ T_l $ is star-shaped in the directions $ \delta_2,\dots,\delta_l $, i.e. for any $ (1;\delta_1,\delta_2,\dots,\delta_l)\in T_l $, $ (1;\delta_1,\lambda_2\delta_2,\dots,\lambda_l\delta_l) $ is also in $ T_l $ for all $ 0<\lambda_l\le \dots \le \lambda_2 \le 1 $.
	\end{lemma}
	\begin{proof}
		Let $ a:=[\omega_{(1;\delta_1,\delta_2,\dots,\delta_k)}] $ and $ a_\lambda:=[\omega_{(1;\delta_1,\lambda_2\delta_2,\dots,\lambda_l\delta_l)}] $. We could prove this lemma directly by using Corollary \ref{cor:ingredients} but we can actually show a stronger result: $ \mc{S}_{a}\subset \mc{S}_{a_\lambda} $.
		By Theorem \ref{thm:stability}, This implies $ t\mc{HLCY}(M_l,a)\subset t\mc{HLCY}(M_l,a_\lambda) $.
		
		For $ A=aH-\sum b_iE_i\in \mc{S}_a $, we see that  $ A\in \mc{S}_{a_\lambda} $ if and only if $ a_\lambda\cdot A>0 $ by Lemma \ref{lemma:stability}.\begin{itemize}
			\item If $ a> 0 $, then $ b_i\ge 0 $ for all $ i $ by Lemma 3.3 of \cite{Chen-SCY}, which implies \[
			a_\lambda\cdot A=a-b_1\delta_1- \sum \lambda_i b_i\delta_i>a-b_1\delta_1-\sum b_i\delta_i>0 .\]
			\item If $ a<0 $, then $ A$ must be of the form $-|a|H+(|a|+1)E_1-E_{j_2}-\dots -E_{j_s} $ by Lemma 3.4 of \cite{Chen-SCY}. Then we have \[
			a_\lambda\cdot A=-|a|+(|a|+1)\delta_1-\lambda(\delta_{j_2}+\dots +\delta_{j_s})\ge -|a|+(|a|+1)\delta_1-(\delta_{j_2}+\dots +\delta_{j_s})>0 .\]
			\item If $ a=0 $, then by adjunction formula we have $ \sum b_i(b_i-1)=2 $. So there is one $ k $ such that $ b_k=-1$ or $2 $ and $ b_j=0$ or $1 $ for $ j\neq k $. Since $ a\cdot A=\sum b_i\delta_i>0 $, we must have $ b_k=-1 $ and $ j<k $ if $ b_j=1 $. Then \[
			a_\lambda\cdot A=\lambda_k\delta_k-\sum_{j>k} \lambda_jb_j\delta_j\ge \lambda_k(\delta_k-\sum_{j>k}b_j\delta_j)=\lambda_k(a\cdot A)>0.  \]
		\end{itemize}
	\end{proof}
	\begin{proposition}\label{prop:toric region connected}
		Toric region  $ T_l $  is connected for each $ l $.
	\end{proposition}
	\begin{proof}
		For $ l\le 2 $, $ T_l=P_l $ is the entire normalized reduced cone and is thus connected. So we only consider the case $ l\ge 3 $.
		The region \[ T_l^\varepsilon:=\{ (1;\delta_1,\dots,\delta_l)\in P_l \, |\, \delta_{i-1}>\delta_i, \delta_i<\varepsilon, i\ge 2 \} \] is contained in the toric region $ T_l $ for sufficiently small $ \varepsilon $ since
		\[ (H-E_1, E_1-E_2, E_2-E_3, ..., E_{k-1}-E_{k}, E_k, H-E_1-E_2-... - E_k, H) \]
		is a toric log Calabi-Yau divisor in this region.
		Note that every $ (1;\delta_1,\delta_2,\dots,\delta_l) $ in $ T_l $ is connected to $ (1;\delta_1,\lambda_2\delta_2,\dots,\lambda_l\delta_l)\in T^\varepsilon_l $ for sufficiently small $ 0<\lambda_l<\dots<\lambda_2\le 1 $. So we conclude that $ T_l $ is connected.
	\end{proof}

	\subsection{General counting formulas for restrictive reduced symplectic classes}\label{section:general counting}
	When the blow up times are getting higher, both the chamber structure on the normalized $c_1$-nef cone and the homological action of symplectomorphism group become quite involved. It's hopeless to give a comprehensive description of the counting according to the division of the symplectic cone as Proposition \ref{prop:M3} for all higher blow ups. Nevertheless, when we only focus on a restrictive part of reduced symplectic classes, it's possible to write down a general counting formula.
	
	In this section, by a \textbf{restrictive} reduced symplectic class, we mean a reduced symplectic class encoded by $\vec\delta=(\delta_1,\cdots,\delta_l)$ which staconetisfies the extra assumptions:
	\[1>\delta_1+\cdots+\delta_l \]
	\[\delta_1>\delta_2\]
	\[\delta_k>\delta_{k+1}+\cdots+\delta_l\]
	for all $2\leq k\leq l-1$. There are two advantages of counting in this region. Firstly we don't need to consider the homological action of the symplectomorphism group by Lemma \ref{lemma:divisor = homology type} since this region is contained in the interior of the reduced symplectic cone. Secondly, the classes $H-\sum_{i=1}^l \epsilon_iE_i$ and $E_i-\sum_{j>i} \epsilon_jE_j$ can be used freely to create the homological configurations.
	
	Now we introduce the strategy for this general counting formula, which is different from the counting of $M_2,M_3$ in the previous sections. There, we firstly determine all the possible self-intersection sequence, according to which we then list all the possible homological configurations by Corollary \ref{cor:ingredients} and consider the quotient by homological action of the symplectomorphism group. We also make use of Propostion \ref{prop:spherical classes} which makes our counting more efficient. However, for the counting issue of general $M_l$, we are going to apply Lemma \ref{lemma:reduced induction} directly. The homological configurations in $M_l$ come from the homological toric and non-toric blow ups of either homological configurations in $M_1$ or $(3H-E_1-\cdots-E_{j-1}-2E_j,E_j)$ in $M_j$. So we will actually count the number of the possible blow up patterns modulo some symmtries resulting the same homological configurations in $M_l$.
	
	Note that by the restrictive condition, the only possible  classes having negative symplectic area in $\mathcal{H}_l$ (Corollary \ref{cor:ingredients}) are \[kH-(k-1)E_1-\sum_{i=2}^l \epsilon_i E_i, k\in\ZZ,\epsilon_i\in\{0,1\}\] with non positive $k$. We need some notations to record the number of such $k$ that can appear in the homological configurations:
	
	Let
	\[
	\mathcal{G}_l=\bigg\{
	g:\{2,\cdots\,l\}\rightarrow \{0,1\}
	\bigg\}
	\]
	\[
	\psi(\vec\delta,g)=\lceil\dfrac{\delta_1-\sum_{i=2}^l g(i)\delta_i}{1-\delta_1}\rceil
	\]
	
	We also need the following set of functions to record the length of the divisor at each step:
	\[
	\mathcal{F}_l^a=\bigg\{
	f:\{1,2,\cdots\,l\}\rightarrow\ZZ_+ \,\vert \,f(1)=a, f(i)-f(i-1)=0\,\,\text{or}\,\,1,\,\text{for all}\,\, 2\leq i\leq l
	\bigg\}
	\]
	In particular, for the purpose of merely counting toric divisors, we denote by $F_l^a$ the function in $\mathcal{F}_l^a$ such that $F_l^a(i)=a+i-1$ for all $i$.
	
	For any $g\in{\mathcal{G}_l}, f\in {\mathcal{F}_l^a}$ and reduced vector $\vec\delta=(\delta_1,\cdots,\delta_l)$, define
	\[
	\phi(f,g)=\prod_{i=2}^l ((1-g(i))f(i-1)-(-1)^{g(i)}(f(i)-f(i-1)+1))
	\]
	This number will represent the number of blow up patterns such that the length at each step is given by $f$ and there exists the class $kH-(k-1)E_1-\sum_{i=2}^l g(i) E_i$ in the ultimate configuration. (See its alternative expression in the proof of Lemma \ref{lemma:relation}).

	For the convenience of the statement, in the following we always use $\vec t$ to denote an ordered finite set $(t_1,\cdots,t_k)$ with all $t_i$ being positive integers. Let
	\[|\vec t|=t_1+\cdots+t_k
	\]
	\[a^{\vec t}=a^{t_1}(a+1)^{t_2}\cdots (a+k-1)^{t_k}
	\]
	where $a$ is an integer. It follows that we have the relation:
	\begin{lemma}\label{lemma:relation}
		\[
		\sum_{f\in{\mathcal{F}_l^a,g\in{\mathcal{G}_l}}}\phi(f,g)=2\sum_{|\vec t|=l-1}a^{\vec t}
		\]
	\end{lemma}
	\begin{proof}
		Note that
		\[((1-g(i))f(i-1)-(-1)^{g(i)}(f(i)-f(i-1)+1))=\begin{cases}
		2\,\,\,\,\,\,\text{if $g(i)=1,f(i)-f(i-1)=1$}\\
		1\,\,\,\,\,\,\text{if $g(i)=1,f(i)-f(i-1)=0$}\\
		f(i-1)-2\,\,\,\,\,\,\text{if $g(i)=0,f(i)-f(i-1)=1$}\\
		f(i-1)-1\,\,\,\,\,\,\text{if $g(i)=0,f(i)-f(i-1)=0$}
		\end{cases}
		\]
		If we fix some $f\in\mathcal{F}_l^a$ and sum over $g\in{\mathcal{G}_l}$, we will get
		\[\sum_{g\in{\mathcal{G}_l}}\phi(f,g)=f(1)f(2)\cdots f(l-1)
		\]
		Therefore
		\[\sum_{f\in{\mathcal{F}_l^a,g\in{\mathcal{G}_l}}}\phi(f,g)=\sum_{f\in{\mathcal{F}_l^a}}\sum_{g\in{\mathcal{G}_l}}\phi(f,g)=2\sum_{|\vec t|=l-1}a^{\vec t}\]
	\end{proof}
	
	To circumvent too many sigma notations in the general counting formulas, for any integer $a$ and $g\in\mathcal{G}_l$, we also introduce $a^{\vec g}$ to denote:
	\[\sum_{f\in\mathcal{F}_{l}^a}\phi(f,g)
	\]
	\begin{prop}\label{prop:general count}
		With the extra restrictive assumption on the reduced symplectic class encoded by $\vec\delta=(\delta_1,\cdots,\delta_l)$, we have
		\begin{equation*}
		\begin{split}
		|\mathcal{LCY}(M_l,\omega_{\delta_{1},\cdots,\delta_{l}})|&= 1+2^l+\sum_{|\vec t|=l-1}(2^{\vec t}+3^{\vec t})+\sum_{1\leq|\vec t|\leq j-1}2^{\vec t}\\
		&+\sum_{g\in{\mathcal{G}_l}}\bigg[1+\dfrac{1}{2}(2^{\vec g}+4^{\vec g})+3^{\vec g}\bigg]\psi(\vec\delta,g)\\
		\end{split}
		\end{equation*}
		
		\begin{equation*}
		\begin{split}
		|t\mathcal{LCY}(M_l,\omega_{\delta_{1},\cdots,\delta_{l}})|= \sum_{g\in{\mathcal{G}_l}}\dfrac{1}{2}\phi(F_l^4,g)\psi(\vec\delta,g)\\
		\end{split}
		\end{equation*}
		Moreover, when $\vec\delta$ doesn't satisfy the extra restrictive assumption, the above formulas will give a strictly upper bound of the counting.
	\end{prop}
	
	\begin{proof}
		Note that the assumption guarantees $\vec\delta$ is in the interior of the normalized reduced cone. Thus there is no difference among $\mathcal{LCY}$, $\mathcal{HLCY}$ and $p\mathcal{HLCY}$. Given any $(A_1,\cdots,A_k)\in\mathcal{HLCY}(M_l,\omega_{\delta_{1},\cdots,\delta_{l}})$ with $l\geq 2$, by Lemma \ref{lemma:reduced induction} we observe that it must be either the iterated toric or non-toric homological blow up of some element in $p\mathcal{HCLY}(M_1)=\cup_{\delta<1}p\mathcal{HCLY}(M_1,\omega_\delta)$ or $(3H-E_1-\cdots-E_{j-1}-2E_j,E_j)$ for some $2\leq j\leq l$. We call them the \textbf{germ} of $(A_1,\cdots,A_k)$. Conversely, blowing up different germs must give different divisors $M_l$. Only different toric or non-toric blow up patterns on the same germ may lead to the same divisor in $M_l$ due to the cyclic symmetry of the circular divisors.
		
		Therefore the counting is divided into the following parts according to the germ:
		\begin{itemize}
			\item The germ is some divisor $((k+1)H-kE_1,H_1,-kH+(k+1)E_1,H_1)$ with $k\geq 0$ in $p\mathcal{HCLY}(M_1)$. The assumption ensures that after $l-1$ times homological blow up, only $-kH+(k+1)E_1-\sum_{i=2}^l g(i) E_i$ may not in $\mathcal{S}_{
				\omega_{\delta}}$ for some $g\in\mathcal{G}_l$. Now we use the function $f$ to denote the length at each stage. So $\phi(f,g)$ actually denotes the number of blow up patterns which follows the length function $f$ and gives $-kH+(k+1)E_1-\sum_{i=2}^l g(i) E_i$ in the final homologcial configuration. And the sum over all possible $f,g$ and integers $k$, which is given by $\sum_{f\in{\mathcal{F}_l^{4}}, g\in{\mathcal{G}_l}}\phi(f,g)\psi(\vec\delta,g)$, leads to the count of blow up patterns starting from length $4$ divisor in $M_1$.
			
			Observe that whenever the blow up pattern is not performing $l-1$ non-toric blow up on $(k+1)H-kE_1$ and $-kH+(k+1)E_1$, there exist exactly two patterns resulting in the same divisor in $M_l$ due to the cyclic symmetry of $((k+1)H-kE_1,H_1,-kH+(k+1)E_1,H_1)$. The ultimate count should then be \[
			\sum_{g\in{\mathcal{G}_l}}\sum_{f\in{\mathcal{F}_l^{4}}}(\dfrac{1}{2}(\phi(f,g)-1)+1)\psi(\vec\delta,g)\]
			
			When counting only toric divisors, eliminating the symmetry is simply by dividing by $2$. So we could obtain the toric counting formula
			\[
			|t\mathcal{LCY}(M_l,\omega_{\delta_{1},\cdots,\delta_{l}})|= \sum_{g\in{\mathcal{G}_l}}\dfrac{1}{2}\phi(F_l^4,g)\psi(\vec\delta,g)
			\]
			\item The germ is some divisor $((k+2)H-(k+1)E_1,-kH+(k+1)E_1,H_1)$ with $k\geq -1$ in $p\mathcal{HCLY}(M_1)$. Note that if $k\geq 0$, there is no cyclic symmetry so that the count of patterns exactly gives the count of divisors, which is
			\[
			\sum_{g\in{\mathcal{G}_l}}\sum_{f\in{\mathcal{F}_l^{3}}}\phi(f,g)\psi(\vec\delta,g)
			\]
			When $k=-1$, the germ is $(H_1,H,H)$. Observe that exactly two patterns generate the same divisor unless the pattern is performing $l-1$ non-toric blow ups on $H_1$ or performing toric blow up once between $H$ and $H$ to get $H-E_j,E_j,H-E_j$ and non-toric blow ups on $H_1$ and $E_j$. The number of those patterns should be $1+\sum_{j=2}^l2^{l-j}=2^{l-1}$. The count of divisors thus should be
			\[
			\dfrac{1}{2}\bigg[-2^{l-1}+\sum_{f\in{\mathcal{F}_l^3},g\in{\mathcal{G}_l}}\phi(f,g)\bigg]+2^{l-1}
			\]
			\item The germ is some divisor $((k+3)H-(k+2)E_1,-kH+(k+1)E_1)$ with $k\geq -1$ or $(2H,H-E_1)$ in $p\mathcal{HCLY}(M_1)$. Observe that exactly two patterns generate the same divisor unless the pattern is performing $l-1$ non-toric blow ups. The count for $((k+3)H-(k+2)E_1,-kH+(k+1)E_1)$ with $k\geq 0$ is
			\[
			\sum_{g\in{\mathcal{G}_l}}\sum_{f\in{\mathcal{F}_l^{2}}}(\dfrac{1}{2}(\phi(f,g)-1)+1)\psi(\vec\delta,g)\]
			The count for $(2H,H-E)$ and $(2H-E,H)$ should both be
			\[	\sum_{g\in{\mathcal{G}_l}}\sum_{f\in{\mathcal{F}_l^{2}}}(\dfrac{1}{2}(\phi(f,g)-1)+1)\]
			\item The germ is $(3H-E_1-\cdots-E_{j-1}-2E_j,E_j)$ for some $2\leq j\leq l$. Similar to the cases of $(2H,H-E)$ and $(2H-E,H)$ above, the count should be
			\[
			\sum_{j=2}^l\bigg[\sum_{f\in{\mathcal{F}_{l-j+1}^2},g\in{\mathcal{G}_{l-j+1}}}\dfrac{1}{2}\phi(f,g)+2^{l-j-1}\bigg]
			\]
		\end{itemize}
		Adding the above four parts and one elliptic divisor together and applying Lemma  \ref{lemma:relation}, we get our desired counting formula. Moreover, we have enumerated all the possible homological types in the above analysis. If $[\omega]$ is not in the restrictive region, the number of divisors will be strictly less than the number given by the formula due to the lack of some classes $H-\sum_{i=1}^l \epsilon_iE_i$ or $E_i-\sum_{j>i} \epsilon_jE_j$.
	\end{proof}

	\begin{rmk}
		When $l=2,3$, this extra assumption only implies $\vec\delta$ is in the interior of the normalized reduced cone. We can compare the counting results and the above formulas:\\
		For the toric case, when $l=2$, $\mathcal{G}_2$ only contains two elements and $\dfrac{1}{2}\phi(F_2^4,g)$ is always equal to $1$. This corresponds to Corollary \ref{cor:toric count M2} :$\lceil\dfrac{\delta_1}{1-\delta_1}\rceil+\lceil\dfrac{\delta_1-\delta_2}{1-\delta_1}\rceil$. When $l=3$,
		\[
		\dfrac{1}{2}\phi(F_3^4,g)=\begin{cases}
		2\,\,\,\,g(2)=0,g(3)=1\\
		2\,\,\,\,g(2)=1,g(3)=1\\
		3\,\,\,\,g(2)=0,g(3)=0\\
		3\,\,\,\,g(2)=1,g(3)=0
		\end{cases}
		\]
		This gives the counting $3(\lceil\dfrac{\delta_1}{1-\delta_1}\rceil+\lceil\dfrac{\delta_1-\delta_2}{1-\delta_1}\rceil)+2(\lceil\dfrac{\delta_1-\delta_3}{1-\delta_1}\rceil+\lceil\dfrac{\delta_1-\delta_2-\delta_3}{1-\delta_1}\rceil)$ in the interior of the normalized reduced cone, which coincides with Proposition \ref{prop:M3} cases (1)-(4).\\
		For the general LCYs, when $l=2$, we see that \[\sum_{|\vec t|=1}2^{\vec t}=2\]
		\[
		\sum_{|\vec t|=1}3^{\vec t}=3
		\]
		\[
		1+\dfrac{1}{2}(2^{\vec g}+4^{\vec g})+3^{\vec g}=1+3+3=7 \,\,\,\,\,\,\,\text{for both $g\in\mathcal{G}_2$}
		\]
		This recovers $7(\lceil\dfrac{\delta_1}{1-\delta_1}\rceil+\lceil\dfrac{\delta_1-\delta_2}{1-\delta_1}\rceil)+12$ in Propostion \ref{prop:general count M2}.
	\end{rmk}

	\section{Symplectic log Calabi-Yau divisors and almost toric fibrations}
	
	In this section we will relate the almost toric fibrations with the symplectic log Calabi-Yau divisors studied in the previous sections. Let's start with the toric actions, which of course could be viewed as a special type of almost toric fibrations.
	\subsection{Toric manifolds and Delzant polytope}
	An action of the 2-torus $ \mb{T}\cong (S^1)^2 $ on a symplectic 4-manifold $ (X,\omega ) $ is a homomorphism \[
	\rho:\mb{T}\to \Symp(X,\omega)
	\]
	from the torus to the symplectomorphism group of $ (X,\omega) $, such that the map $ \rho^\sharp:\mb{T}\times X\to X $, defined by $ \rho^\sharp(t,x)=\rho(t)(x) $, is smooth. An effective $ \mb{T} $-action with generating vector fields $ \xi_1,\xi_2 $ is Hamiltonian if there exists a moment map $ \mu:X\to \RR^2 $ such that each component satisfies $ d\mu_j=-\iota_{\xi_j}\omega $, for $ j=1,2 $. Such a $ \mb{T} $-action is called a toric action and could be seen as a smooth injective map $ \rho:\mb{T}\to Ham(X,\omega ) $. The 4-tuple $ (X,\omega,\mb{T},\mu) $ is called a symplectic toric manifold.
	%
	\begin{definition}\label{def:toric equivalence}
		Two toric actions $ \rho_1,\rho_2 $ on a symplectic manifold $(X,\omega)$ are equivalent if there exists a symplectomorphism $\phi: X \rightarrow X$ and an automorphism $ h:\mb{T}\to \mb{T} $ such that the diagram \begin{center}
			\begin{tikzcd}
			\mb{T} \times X
			\ar[r,"{\rho_1^{\sharp}}"] \ar[d,"{(h,\phi)}"']
			& X \ar[d,"\phi"]\\
			\mb{T} \times X \ar[r,"{\rho_2^{\sharp}}"'] & X
			\end{tikzcd}
		\end{center}
		is commutative.
	\end{definition}
	Symplectic toric manifolds are interesting objects to study because much of their geometry and topology are determined by the combinatorial information of their moment image.
	
	\begin{definition}\label{def:delzant}
		A Delzant polytope $\Delta$ in $\mathbb{R}^2$ is a polytope satisfying:
		\begin{enumerate}[label=$ (\arabic*) $]
			\item simplicity, i.e. there are 2 edges meeting at each vertex;
			\item rationality, i.e. the edges meeting at each vertex p are of the form $p+tu_i,t\ge0, u_i\in \mathbb{Z}^2$;
			\item smoothness, i.e. for each vertex p, the corresponding $u_1,u_2$ can be choosen to be a $\mathbb{Z}$-basis of $\mathbb{Z}^2$.
		\end{enumerate}
	\end{definition}
	
	Recall that the moment image of a symplectic toric manifold is a Delzant polytope, which we call the moment polytope. Delzant's theorem (\cite{Delzant}) classifies equivalence classes of symplectic toric manifolds in terms of combinatorial data given by Delzant polytopes. In particular we have:
	\begin{prop}[\cite{KaKePi07-finite}]
		Two toric action on $ (X,\omega ) $ are equivalent if and only if their moment map images are $ AGL(2,\ZZ) $-congruent.
	\end{prop}
	
	Let $ \Delta  $ be the moment polytope of a symplectic toric manifold $ (X,\omega,\mathbb{T},\mu) $ and $ e $ be an edge in the boundary $ \partial \Delta  $. Then the preimage $ \mu^{-1}(e) $ is a symplectic sphere $ C_e $. So the preimage $ D=\mu^{-1}(\partial \Delta) $ is a cycle of symplectic spheres intersecting $ \omega $-orthogonally, which we call it the boundary divisor of $ \Delta  $. Actually the Poincare dual of $ D $ is $ c_1(X,\omega ) $ (\cite{Sym02}, Proposition 8.2), so it is a symplectic log Calabi-Yau divisor.
	
	We can read off the symplectic area and self-intersection of each component of $ D $ from the moment polytope as follows. Let $ \bar{e}=re_0 $ be a vector based at origin representing $ e $, with $ r>0 $ and $ e_0 $ a primitive vector. Then the symplectic area of $ C_e $ is $ \omega([C_e])=r $, which is also called the affine length of $ e $. Suppose $ e_1,e_2 $ are the two edges intersecting $ e $ and denote by $ n_1,n_2,n $ the inward unit normal vector along $ e_1,e_2,e $ respectively. Then the self-intersection of $ C_e $ is equal to the unique integer $ s $ such that $ n_1+n_2-sn=0 $.

	\subsection{Moment polytope and toric symplectic log Calabi-Yau divisors}\label{section:toric action}
	
	The main result of the subsection is the following construction of a moment polytope from a toric symplectic log Calabi-Yau pair, which leads to the proof of Theorem \ref{thm:toric=tLCY} at the end of this subsection.
	\begin{prop}\label{prop:CY to toric}
		Given an $ \omega $-orthogonal toric symplectic Calabi-Yau divisor $ D\subset (X,\omega) $, there is a toric action on $ (X,\omega ) $ such that $ D $ is the boundary divisor. In particular, the map \[
		\mc{T}(X,\omega)\to t\mc{LCY}(X,\omega)
		\]
		taking a toric action to its boundary divisor is a surjection.
	\end{prop}
	
	Define a set of primitive vectors in $ \RR^2 $ as follows. Let $ d_1=\begin{pmatrix} 0\\-1 \end{pmatrix} $, $ d_2=\begin{pmatrix} 1\\0 \end{pmatrix} $ and define $ d_i=-s_{i-1}d_{i-1}-d_{i-2} $ for $ i=3,\dots,k $. Such set of vectors is called a \textbf{generating set} associated to $ s=s(D) $. Recall for real numbers $ b_1,\dots,b_k $, we define the continued fraction as \[ [b_1,\dots,b_k]=b_1-\cfrac{1}{b_2-\cfrac{1}{b_3-\cfrac{1}{\cfrac{\substack{\vdots}}{b_{k-1}-\cfrac{1}{b_k}}}}}. \]
	If we write $ d_i=\begin{pmatrix}x_i\\y_i\end{pmatrix} $ with $ gcd(x_i,y_i)=1 $, then it's easy to check $ -\dfrac{x_i}{y_i}=[s_2,\dots,s_{i-1}] $ for $ i=3,\dots,k $. So the vector $ d_i $ is determined up to sign by the continued fraction $ [s_2,\dots,s_{i-1}] $.
	\begin{rmk}
		Special care is needed when either $ x_i $ or $ y_i $ is $ 0 $. We take the computation of continued fraction to be entirely formal so that the appearance of $ 0 $ doesn't affect the outcome. If $ x_i=0 $, then $ y_i $ must be $ \pm 1 $ and vice versa. This is because $	x\pm \dfrac{1}{y}=\dfrac{0}{a} $ implies $ |x|=|y|=|a|=1 $ and $ x\pm \dfrac{1}{y}=\dfrac{a}{0} $ implies that $ y=0 $ and $ |a|=1 $. So each continued fraction corresponds uniquely to a primitive vector.
	\end{rmk}
	All indices below are taken to be modulo $  k $.
	\begin{lemma}\label{lemma:continued fraction}
		When $ D $ toric, we have $ d_{i-1}+s_id_i+d_{i+1}=0 $ for all $ i=1,\dots,k $.
	\end{lemma}
	\begin{proof}
		It suffices to prove the case for $ i=1,k $. Note that \begin{align*}
		d_{k-1}+s_kd_k+d_1=0 &\quad\Leftrightarrow\quad  [s_2,\dots,s_k]=0\\
		d_k+s_1d_1+d_2=0 & \quad\Leftrightarrow\quad  [s_1,\dots,s_{k-1}]=0
		\end{align*}
		We claim that when $ D $ is toric, $ [s_{i+1},s_{i+2},\dots,s_k,s_1,\dots,s_{i-1}]=0 $ for all $ i $. It is easily checked to be true when $ s=(1,1,1) $ or $ s=(0,n,0,-n) $. So it suffices to prove that the continued fraction is invariance under toric blow-up, i.e. \[
		[b_1,\dots,b_i-1,-1,b_{i+1}-1,\dots,b_k]=[b_1,\dots,b_k].
		\]
		Since continued fractions splits like $ [b_1,\dots,b_k]=[[b_1,\dots,b_j],[b_{j+1},\dots,b_k]] $, it suffices to prove $ [x,y]=[x-1,-1,y-1] $, which is an easy computation.
	\end{proof}
	\begin{lemma}\label{lemma:delzant}
		Let $ D $ be toric symplectic log Calabi-Yau divisor, $ a $ its area vector and $ \{ d_i\} $ its generating set. Then we have \begin{enumerate}[label=(\arabic*)]
			\item $ \sum a_id_i=0 $,
			\item $ \det (d_i, d_{i+1}):=\det \begin{pmatrix}
			x_i & x_{i+1}\\y_i & y_{i+1}
			\end{pmatrix} =1$.
			\item The angle between two consecutive vectors $ d_i, d_{i+1} $ is less than $ \pi $. In particular, $ d_{i+1} $ is determined by either one of $ d_i,d_{i+2} $ and the continued fraction $ [s_2,\dots,s_{i}] $.
		\end{enumerate}
	\end{lemma}
	\begin{proof}
		Since $ b^+(D)=1 $, there exists $ z\in \RR^k $ such that $ Q_Dz=a $ (\cite{LiMaMi-logCYcontact}, Proposition 5.13), i.e. $ a_i=z_{i-1}+s_iz_i+z_{i+1} $ for all $ i $. Then we have \begin{align*}
		\sum a_id_i=&\sum (z_{i-1}d_i+z_{i-1}d_i+z_is_id_i)\\
		=&\sum z_i(d_{i-1}+s_id_i+d_{i+1})=0.
		\end{align*}
		
		Note when $ i=1 $, we have $ \det \begin{pmatrix}	0 & 1\\-1 & 0	\end{pmatrix}=1 $. Then (2) follows from \begin{align*}
		\det (d_i, d_{i+1})=\det (d_i, -s_id_i-d_{i-1}) =\det (d_i , -d_{i-1})=\det (d_{i-1}, d_i).
		\end{align*}
		For (3) suppose $ \theta $ is the angle between $ d_i,d_{i+1} $, then we have $ \sin \theta = \dfrac{\det (d_i, d_{i+1})}{||d_i||\cdot ||d_{i+1}||}>0 $ and thus $ \theta <\pi  $. Recall that $ d_i $ is determined up to sign by $ [s_2,\dots,s_i] $. The sign ambiguity disappears when $ d_i $ is determined because of the angle restriction.
	\end{proof}
	\begin{lemma}\label{lemma:winding = 1}
		Let $ \tilde{D} $ with self-intersection sequence $ \tilde{s}=(s_1,\dots,s_i-1,-1,s_{i+1}-1,\dots,s_k) $ be the toric blow-up of $ D $ with self-intersection sequence $ s=(s_1,\dots,s_k) $. Then the corresponding generating set $ \{ \tilde{d}_i \} $ has the property that $ \tilde{d}_j=d_j $ for $ j\le i $, $ \tilde{d}_{i+1}=d_i+d_{i+1} $ and $ \tilde{d}_j=d_{j-1} $ for $ j\ge i+1 $. In particular, if $ D $ is toric, then its associated generating set of vectors winds around the origin exactly once.
	\end{lemma}
	\begin{proof}
		The fact that $ \tilde{d}_j=d_{j-1} $ for $ j\ge i+3 $ follows easily from the blow-up invariance of continued fraction proved in Lemma \ref{lemma:continued fraction} and $ \tilde{d}_j=d_j $ for $ j\le i $ is trivial. Since $ [s_2,\dots,s_{i-1},s_i-1,-1]=[s_2,\dots,s_i] $ and $ \tilde{d}_{i+3}=d_{i+2} $, we have $ \tilde{d}_{i+2}=d_{i+1} $ and $ \tilde{d}_{i+1}=\tilde{d}_i+\tilde{d}_{i+2}=d_i+d_{i+1} $. For the second statement, let $ s $ be the self-intersection sequence of $ D $. As seen above, toric blow-up inserts a new vector to the generating set in between two vectors and doesn't add to the winding number. So it suffices to see when $ s=(1,1,1) $ or $ (0,n,0,-n) $, the associated vectors winds around the origin exactly once. But this is trivial because $ k\le 4 $ and angle between consecutive vectors is less than $ \pi  $.
	\end{proof}

	\begin{proof}[Proof of Proposition \ref{prop:CY to toric}]
		One can build a moment polytope for the induced toric action from divisor $ D $ as follows. Let $ Q_1=(0,0) $ and $ Q_{i+1}=Q_i+a_id_i $ for $ i=1,\dots,k-1 $. Lemma \ref{lemma:delzant} (1)(3) and Lemma \ref{lemma:winding = 1} imply that the polygonal chain with vertices $ Q_1,\dots,Q_k $ encloses a simple convex polygon $ P $. This polygon $ P $ is actually a Delzant polygon by Lemma \ref{lemma:delzant} (2). Then there is a symplectic toric manifold $ (X',\omega') $ with moment polytope $ P $ and boundary divisor $ D' $. From the construction we have $ (s(D'), a(D'))=(s(D),a(D)) $. So $ (X,\omega,D) $ is strictly homological equivalent to $ (X',\omega',D') $ by Lemma \ref{lemma:taut}. Since $ D $ is $ \omega $-orthogonal and $ D' $ is $ \omega' $-orthogonal, by Proposition \ref{prop:def to symp}, there is a symplectomorphism between $ (X,\omega,D) $ and $ (X',\omega',D') $. Composing this symplectomorphism with the moment map of $ (X',\omega') $, we get a toric action on $ (X,\omega) $ such that $ D $ is the boundary divisor by construction.
		

		%
	\end{proof}
	
	Recall that  $$ \mc{T}(X,\omega)=\{ \rho:T^2\to Ham(X,\omega ) \} / \sim^{t} $$ is the set of equivalence classes of toric actions on $ (X,\omega ) $ and  $$ t\mc{LCY}(X,\omega)=\{ D\in p\mathcal{LCY}(X,\omega)| D \text{ is a toric divisor} \}/ \sim^{s}$$
	is the set of strictly symplectic deformation classes of toric symplectic log Calabi-Yau divisors in $ (X,\omega) $. Now we are ready to establish the correspondence between $ \mc{T}(X,\omega) $ and $ t\mc{LCY}(X,\omega) $.
	
	
	\begin{proof}[Proof of Theorem \ref{thm:toric=tLCY}]

		Consider the map $ f:\mc{T}(X,\omega )\to t\mc{LCY}(X,\omega ) $, where for each toric action $ \rho:T^2\to \Symp(X,\omega)  $ with moment map $ \mu:X\to \RR^2 $ define $ f(\rho) $ to be the boundary divisor $ D $ of its moment polygon $ \mu(X) $. $ D $ is a cycle of symplectic spheres of length at least $ 3 $ and the homology class $ [D] $ is Poincare dual to the first Chern class $ c_1(X,\omega ) $. So $ (X,\omega,D) $ is a symplectic Looijenga pair. Take a small collar neighborhood $ R $ of the boundary in $ \mu(X) $, which lifts to a neighborhood $ P(D) $ of the divisor $ D $. We might assume the origin $ (0,0) $ is in the interior of $ \mu(X) $ by an affine translation. The outward radial vector field on $ \RR^2 $ lifts to a Liouville vector field near the boundary $ \partial P(D) $, which points into $ P(D) $. So $ (P(D),\omega) $ is a concave neighborhood of $ D $, which means $ D $ satisfies the positive GS criterion (\cite{LiMa14-divisorcap}) and must have $ b^+(Q_D)=1 $ (\cite{LiMaMi-logCYcontact}). The preimage of $ \mu(X)-Int(R) $ is diffeomorphic to $ D^2\times T^2 $. The boundary $ \partial P(D) $ is $ T^3 $ and thus $ q(D)=0 $ by Lemma \ref{lemma:T^3=toric}. So $ D $ is a toric symplectic log Calabi-Yau divisor.
		
		Let $ \rho,\rho' $ be two toric actions on $ (X,\omega) $ with boundary divisors $ D=f(\rho ) $ and $ D'=f(\rho') $.
		choose a cyclic labeling of both boundary divisors so that $ D=f(\rho)=\cup C_i $ and $ D'=f(\rho')=\cup C_i' $.
		Denote by $ (s,a) $ and $ (s',a') $ the self-intersection and area vectors of $ D $ and $ D' $ with respect some labeling. If $ \rho,\rho' $ are equivalent then their moment map images are $ AGL(2,\ZZ ) $-congruent (\cite{KaKePi07-finite}), which implies that $ (s,a) $ and $ (s',a') $ differ by cyclic and anti-cyclic permutations \cite{KaKePi14-count}. So $ D $ and $ D' $ are strictly symplectic deformation equivalent by Lemma \ref{lemma:taut} and $ f $ is well-defined.
		
		Suppose $ D , D'\in t\mc{LCY}(X,\omega) $ are strictly homological equivalent. Note $ (s,a) $ and $ (s',a') $ depend only on the homology classes of components of $ D $ and $ D' $. So they are the same up to cyclicly or anti-cyclicly relabeling the components. So the constructed moment polytopes in Proposition \ref{prop:CY to toric} are the same and the corresponding symplectic toric manifolds are equivariantly symplectomorphic. So $ f $ is injective.
		
		Again by Proposition \ref{prop:CY to toric}, $ f $ is also surjective and this finishes the proof.
	\end{proof}
	
	With Theorem \ref{thm:toric=tLCY}, properties of symplectic log Calabi-Yau divisors translate to properties of toric actions. Note that the proof of Theorem \ref{thm:stability} also implies the stability of toric symplectic log Calabi-Yau divisors. Combined with Theorem \ref{thm:toric=tLCY} we get the following stability result for toric actions. Recall that $ \mc{S}_\omega $ denotes the set of $ \omega $-symplectic sphere classes.
	\begin{corollary}
		Let $ X $ be a rational surface. Suppose $ \omega  $ and $ \omega' $ are two symplectic forms with $ \mc{S}_{\omega}=\mc{S}_{\omega'} $. Then \[ \mc{T}(X,\omega)=\mc{T}(X,\omega'). \]
	\end{corollary}
	
	


	
	

	\subsection{Comparison with the counting results of Karshon-Kessler-Pinsonnault}
	In \cite{KaKePi07-finite}, \cite{KaKePi14-count} and \cite{KKP}, Karshon, Kessler and Pinsonnault first considered the question of counting inequivalent toric actions on symplectic four-manifolds. Their strategy relies on analyzing the combinatorics of Delzant polygons. Now since we have established the equivalence between toric divisors and toric actions, we could compare our counting results of divisors with their results.   
	
	Firstly, combining Theorem \ref{thm:toric=tLCY} with Corollary \ref{lemma:finiteness}, we recover the following finiteness result of toric actions by Karshon, Kessler and Pinsonnault.
	\begin{corollary}[\cite{KaKePi07-finite}]
		A fixed symplectic rational surface only admits finitely many inequivalent toric actions.
	\end{corollary}
	
	By the general toric counting formula we could confirm the upper bound of toric actions given by Karshon-Kessler-Pinsonnault. This is  because according to Proposition \ref{prop:general count}:
	\begin{equation*}
	\begin{split}
	|t\mathcal{LCY}(M_l,\omega_{\delta_{1},\cdots,\delta_{l}})|&= \sum_{g\in{\mathcal{G}_l}}\dfrac{1}{2}\phi(F_l^4,g)\psi(\vec\delta,g)\\
	&=\sum_{g\in{\mathcal{G}_l},g(2)=0}\dfrac{1}{2}\phi(F_l^4,g)\psi(\vec\delta,g)+\sum_{g\in{\mathcal{G}_l},g(2)=1}\dfrac{1}{2}\phi(F_l^4,g)\psi(\vec\delta,g)\\
	&\leq\sum_{g\in{\mathcal{G}_l},g(2)=0}\dfrac{1}{2}\phi(F_l^4,g)\lceil\dfrac{\delta_1}{1-\delta_1}\rceil+\sum_{g\in{\mathcal{G}_l},g(2)=1}\dfrac{1}{2}\phi(F_l^4,g)\lceil\dfrac{\delta_1-\delta_2}{1-\delta_1}\rceil\\
	&=\dfrac{1}{2}(F_l^4(1)-2)F_l^4(2)\cdots F_l^4(l-1)\lceil\dfrac{\delta_1}{1-\delta_1}\rceil+\dfrac{1}{2}2F_l^4(2)\cdots F_l^4(l-1)\lceil\dfrac{\delta_1-\delta_2}{1-\delta_1}\rceil\\
	&=\dfrac{(l+2)!}{4!}\lceil\dfrac{\delta_1}{1-\delta_1}\rceil+\dfrac{(l+2)!}{4!}\lceil\dfrac{\delta_1-\delta_2}{1-\delta_1}\rceil
	\end{split}
	\end{equation*}

	\begin{corollary}[\cite{KaKePi14-count}]\label{cor:kkp upper bound}
		The number of toric actions on $ (M_l,\omega ) $ is at most \[
		(\lceil\dfrac{\delta_1}{1-\delta_1}\rceil +\lceil \dfrac{\delta_1-\delta_2}{1-\delta_1}\rceil)\cdot \dfrac{(l+2)!}{4!},
		\]
		where $ [\omega]= H-\delta_1E_1-\dots - \delta_lE_l  $ is a normalized reduced symplectic class.
	\end{corollary}
	\begin{rmk}
		By Proposition \ref{prop:M3},  for $ \CC\PP^2\# 3\overline{\CC\PP}^2 $ with $ \delta_1=\delta_2>\delta_3  $, the count is $ 3 $.  Explicitly, up to equivalence, these are
		\begin{align*}
		&(H-E_2-E_3,E_3,E_2-E_3,H-E_1-E_2,E_1,H-E_1)\\
		&(E_3,H-E_2-E_3,E_2,H-E_1-E_2,E_1,H-E_1-E_3)\\
		&(H-E_2,E_2-E_3,E_3,H-E_1-E_2-E_3,E_1,H-E_1)
		\end{align*}
		However, the count in \cite{KKP}  in this case is $5$ by Corollary 8.7 (2) and Example 8.9 (3). The error comes from not taking into account the symmetry of switching
		$E_1$ and $E_2$. Note there is a symplectomorphism switching $E_1$ and $E_2$ since $\delta_1=\delta_2$. In other words, $5$ is the result of counting the toric ones in $p\mathcal{LCY}$ but $3$ is the result of counting the toric ones in $\mathcal{LCY}$.
	\end{rmk}
	
	\subsection{Almost toric fibrations}\label{section:ATF intro}
	In this subsection, we introduce the basics of almost toric fibrations and their associated almost toric base diagrams. We refer to \cite{Sym02}, \cite{LeSy-ATF} and \cite{Evansnotes} for details.
	
	\begin{definition}[\cite{LeSy-ATF}, Definition 2.2]
		An almost toric fibration of a symplectic 4-manifold $ (X,\omega) $ is a Lagrangian fibration $ \pi:X \to B $ with only nodal and elliptic singularities. A toric fibration is a Lagrangian fibration induced by an effective Hamiltonian torus action.
		We denote by $ \mc{ATF}(X,\omega) $ the set of almost toric fibrations on $ (X,\omega) $ and write $ \mc{ATF}(X):=\bigcup \mc{ATF}(X,\omega) $, where the union ranges over all symplectic forms $ \omega $ of $ X $.
	\end{definition}
	The set $ B_0 $ of regular values of $ \pi $ carries an integral affine structure and there is a monodromy circling around a nodal singularity. With a suitable choice of cuts, we get an integral affine immersion $ B-\bigcup \text{cuts} \to \R^2 $.
	The image $ P $ of this immersion, together with nodal rays and nodes, is called an {\bf almost toric base diagram} representing the almost toric fibration, as an analogue to the moment polygon in the case of toric fibrations. The nodal rays are dotted rays representing the image of the cuts while the nodes are marks on the rays and represent focus-focus singularities.
	
	There are three important surgery operations on the base diagram that fix the symplectomorphism type of the manifold.
	The first is a \textbf{nodal trade}, which replaces a neighborhood of a corank $ 2 $ elliptic singularity with a local model of a focus-focus singularity. The effect on the base diagram is that we insert a nodal ray from the toric vertex, with a node representing the focus-focus singularity.
	The second is a \textbf{nodal slide}, which moves the node along the ray on the base diagram (see Figure \ref{fig:nodaltrade}). Nodal trades and slides change the almost toric fibration but not the underlying manifold.
	The third operation is called a \textbf{mutation} with respect to a nodal ray. It changes the base diagram in the following way. The base diagram is sliced into two parts by the nodal ray. One part is unchanged while the other part is acted on by an affine transformation in $AGL(2,\Z)$. Mutations change only the base diagram but not the almost toric fibration (see Figure \ref{fig:mutation}).
	
	\begin{figure}[h]
		\centering
		\begin{minipage}{.5\textwidth}
			\centering
			\includegraphics[width=\linewidth]{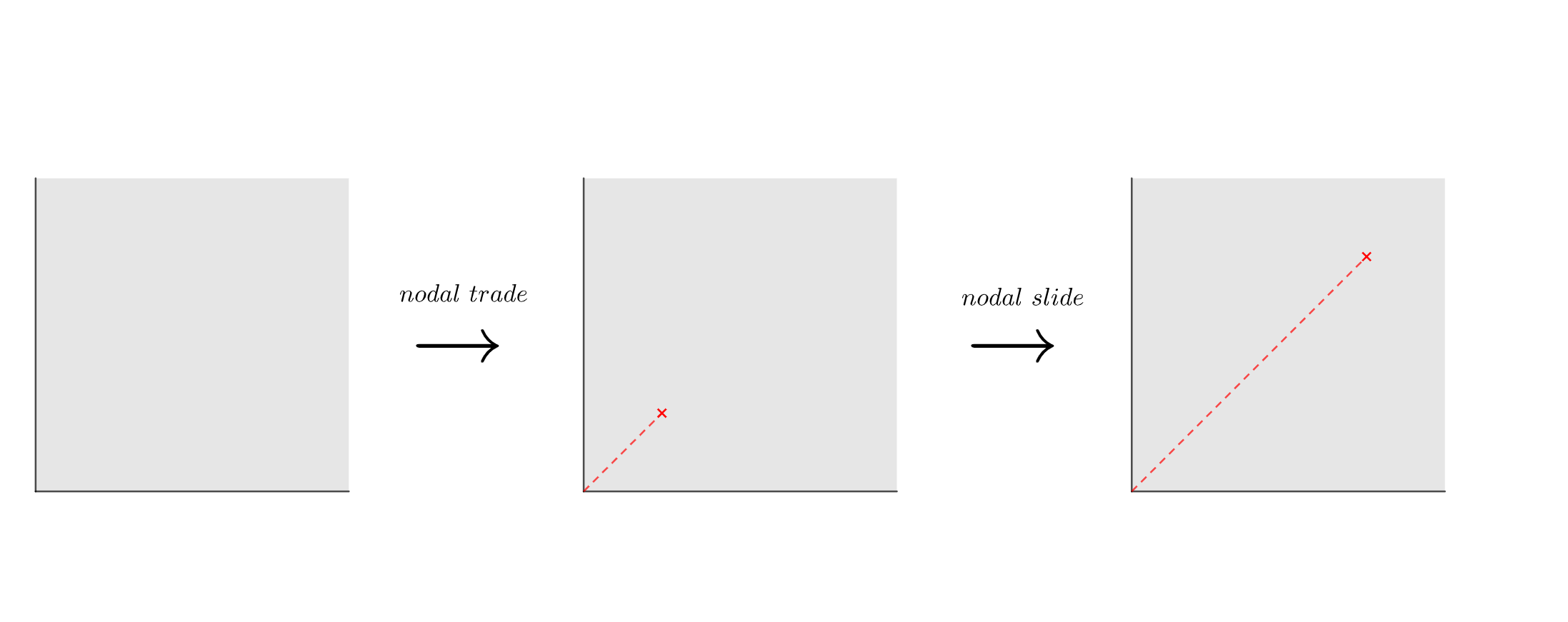}
			\caption{Nodal trade and nodal slide.\label{fig:nodaltrade}}
		\end{minipage}%
		\begin{minipage}{.5\textwidth}
			\centering
			\includegraphics[width=\linewidth]{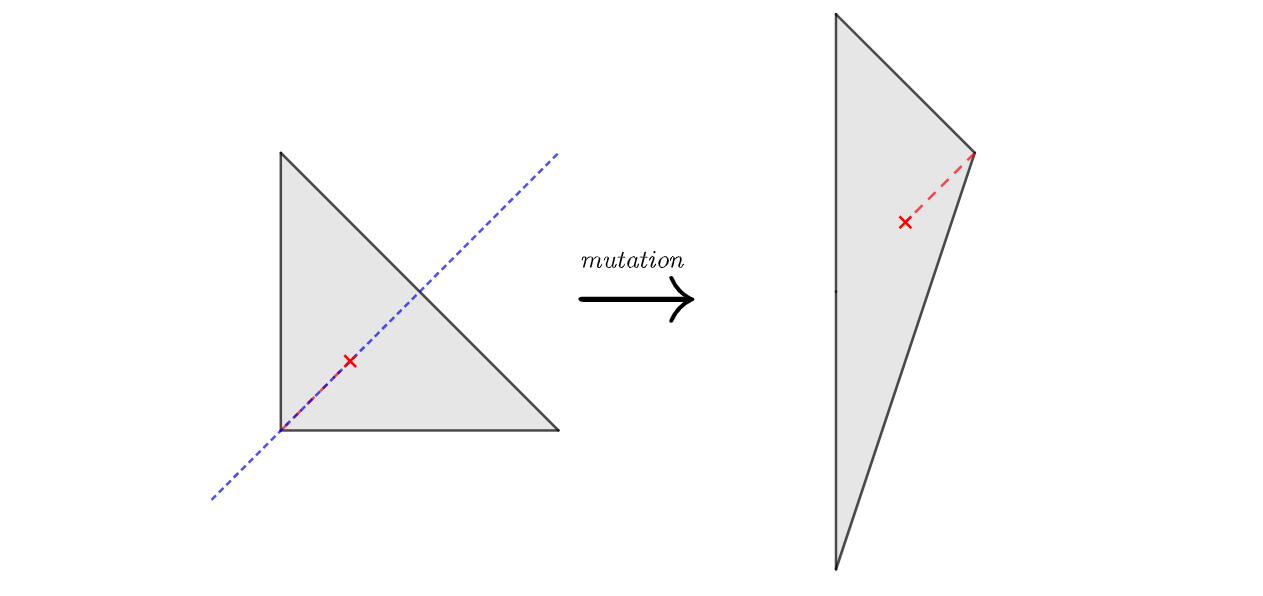}
			\caption{Mutation.\label{fig:mutation}}
		\end{minipage}
	\end{figure}

	Similar to the case of toric fibrations, the preimage $\pi^{-1}(\partial B)$ of the base is a symplectic divisor representing the Poincare dual of $ c_1(X,\omega) $ (Proposition 8.2 of \cite{Sym02}), i.e. a symplectic log Calabi-Yau divisor. We call it the boundary divisor of $ \pi $.
	
	As in toric fibrations, one can do toric blow-ups at elliptic corank $ 2 $ singularities of an almost toric fibration, which amounts to chopping corners centered at toric vertices on the base diagram (see Figure \ref{fig:toricblowup}).
	There is another fibration compatible way to blow up, called an \textbf{almost toric blow-up}. Consider an edge on the base diagram, which we may assume is in the $ (1,0) $-direction by an $ SL(2;\Z) $ transformation. An almost toric blow-up removes a right triangle with edge length $ \varepsilon $ and then adds a node at the top of the triangle with two dashed lines representing the cut (see Figure \ref{fig:nontoricblowup}). Since each edge represents a symplectic sphere, this has the effect of a non-toric blow-up at the corresponding symplectic sphere with size $ \varepsilon $.
	
	The smoothing operation introduced in Section \ref{section:operations} can also be realized as a nodal trade on the base diagram. Near the toric vertex, the boundary divisor comprises two transversely intersecting symplectic discs. After a nodal trade, the boundary divisor becomes a symplectic annulus which is a smoothing of the pair of discs.
	
	\begin{figure}[h]
		\centering
		\begin{minipage}{.5\textwidth}
			\centering
			\includegraphics[width=\linewidth]{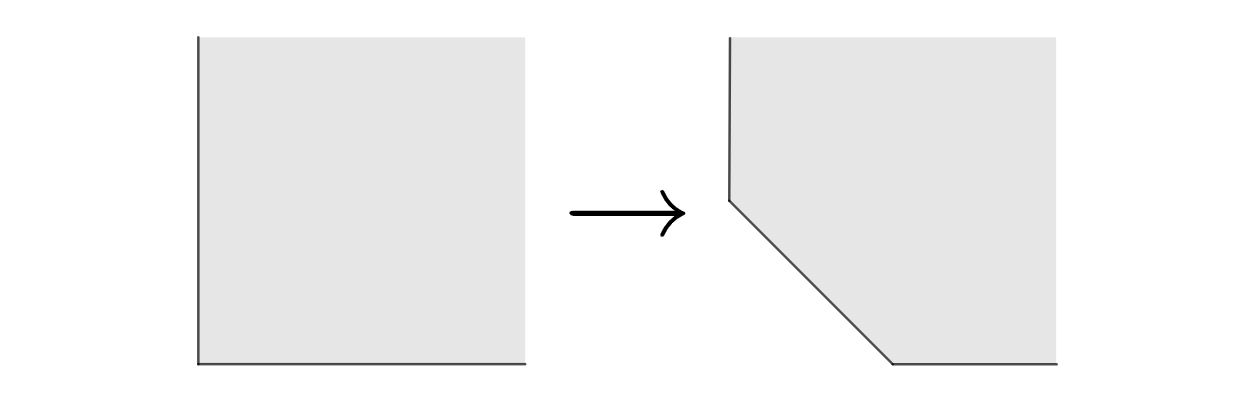}
			\caption{Toirc blow up.\label{fig:toricblowup}}
		\end{minipage}%
		\begin{minipage}{.5\textwidth}
			\centering
			\includegraphics[width=\linewidth]{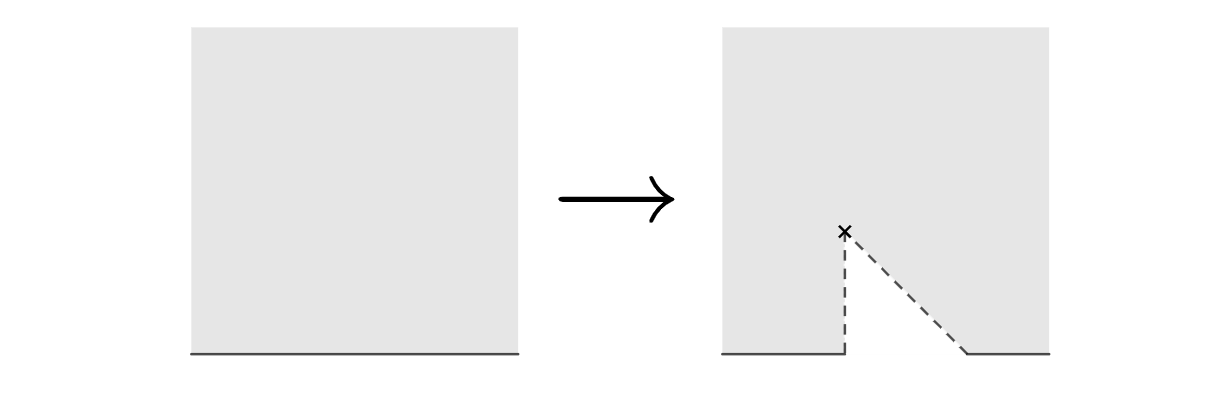}
			\caption{Almost toric blow up.\label{fig:nontoricblowup}}
		\end{minipage}
	\end{figure}
	
	\subsection{Realizing log Calabi-Yau divisors by ATF}
	For each almost toric fibration on a symplectic rational surface $ (X,\omega) $, its boundary divisor is a sympletic log Calabi-Yau divisor $ D $, which has a homology type $ (D) $. So we have a map \[
	\Phi:\mc{ATF}(X,\omega)\to \mc{LCY}(X,\omega),\]
	from the set of almost toric fibrations on $ (X,\omega) $ to the set of symplectic log Calabi-Yau divisors.
	
	In Engel's thesis (\cite{Engel-thesis}, Section 5.1 Part 1), he showed that every holomorphic log Calabi-Yau divisor is realized as the boundary divisor of some almost toric fibration. Since every deformation class of log Calabi-Yau divisor is realized by some holomorphic divisor, we have the following result.
	\begin{prop}[\cite{Engel-thesis}]\label{prop:engel}
		The map $ \tilde{\Phi}: \mathcal{ATF}(X)\to \widetilde{\mathcal{LCY}}(X) $ is surjective.
	\end{prop}
	\begin{remark}
		The above Proposition can also be derived from Lemma \ref{lemma:smoothing}. This is simply because every element in $t\widetilde{\mathcal{LCY}}(X)$ has a toric realization by Theorem \ref{thm:toric=tLCY} and the smoothing operations in the category of $\mathcal{LCY}$ corresponds to the nodal trade operations in the category of $\mathcal{ATF}$.
	\end{remark}
	Motivated by Engel's result, we raise the following conjecture on the realization of symplectic log Calabi-Yau divisors by almost toric fibrations.
	\begin{conjecture}\label{conj:ATF}
		The map $ \Phi:\mc{ATF}(X,\omega)\to \mc{LCY}(X,\omega) $ is surjective.
	\end{conjecture}
	
	Based on the observation that all log Calabi-Yau divisors in small rational manifolds are smoothing of toric log Calabi-Yau divisors and the fact that smoothing of boundary divisors can be realized by nodal trade in almost toric fibrations, we prove the following cases of Conjecture \ref{conj:ATF}.
	
	\begin{prop}\label{prop:surjectivity conjecture minimal}
		The map $ \mc{ATF} (X,\omega)\to \mc{LCY}(X,\omega) $ is surjective for $ X=\CP^2, S^2\times S^2 $, $ \CP^2\# \overline{\CP}^2 $ and $ \CP^2 \# 2\overline{\CP}^2 $.
		
	\end{prop}
	\begin{proof}
		It suffices to show every homological configuration $ (D)\in \mc{HLCY}(X,\omega) $ can be obtained from smoothing a toric homological configuration $ (\overline{D})\in t\mc{HLCY}(X,\omega) $. Then the corresponding almost toric fibration is obtained from the toric fibration of $ \overline{D} $ by nodal trades. Again it suffices to show this for normalized reduced symplectic classes. Also note that for $ \CP^2,S^2\times S^2 $ and $ \CP^2\# \overline{\CP}^2 $, all homological log Calabi-Yau divisors are listed in Theorem \ref{thm:minimal model}.
		
		(1) Consider $ X=\CP^2 $ with symplectic class $ [\omega_{FS}]=H $. Then \[ \mc{HLCY}(X,\omega_{FS})=\{ (3H), (2H,H),(H,H,H) \}\] and each of them is a smoothing of $ (H,H,H) $.
		
		(2) Consider $ X=S^2\times S^2 $ with a standard basis $ \{ B,F \} $ and a normalized reduced symplectic class $ [\omega_\mu]=B+\mu F $ with $ \mu\ge 1 $. Then $ t\mc{HLCY}(X,\omega_\mu)=\{ (B+kF,F,B-kF,F)\,|\, \mu>k\ge 0 \} $ and the non-toric homological log Calabi-Yau divisors are \[
		\{(B+(k+1)F,F,B-kF), (B+(k+2)F,B-kF)\,|\, \mu>k\ge 0\} \cup \{ (B+F,B+F),(2B+2F)\}.\]
		So it is clear they can be obtained from toric ones via smoothing.
		
		(3) The case for $ \CP^2\#\overline{\CP}^2 $ is similar to the case of $ S^2\times S^2 $.
		
		(4) Consider $ X=\CP^2\# 2\overline{\CP}^2 $ with a standard basis $ \{ H,E_1,E_2\} $ and a normalized reduced symplectic class $ [\omega] $.
		Suppose $ (D)\in \mc{HLCY}(X,\omega) $ a homological configuration which contains a $ (-n) $-component for some $ n\ge 2 $ . By Proposition \ref{prop:general count M2}, there is exactly one toric homological configuration $ (\overline{D}) $ with a $ (-n) $-component. Suppose $ n=2k $ for $ k\ge 1 $. Then the homological configuration $ (\overline{D}) $ is given by \[
		(H-E_1,kH-(k-1)E_1,H-E_1-E_2,E_2,-(k-1)H+kE_1-E_2).	
		\]
		Note that smoothing of $ (\overline{D}) $ gives all the possible homological log Calabi-Yau divisors with a $ (-n) $-component, as listed in the proof of Proposition \ref{prop:general count M2}:
		\begin{itemize}
			\item $ ((k+1)H-kE_1,H-E_1-E_2,E_2,-(k-1)H+kE_1-E_2) $,\\
			$ (H-E_1,kH-(k-1)E_1,H-E_1,-(k-1)H+kE_1-E_2) $,\\
			$ (H-E_1,(k+1)H-kE_1-E_2,E_2,-(k-1)H+kE_1-E_2)$
			\item $ ((k+2)H-(k+1)E_1-E_2,E_2,-(k-1)H+kE_1-E_2),\\((k+1)H-kE_1,H-E_1,-(k-1)H+kE_1-E_2) $
			\item $ ((k+2)H-(k+1)E_1,-(k-1)H+kE_1-E_2 )$.
		\end{itemize}
		In particular, $ (D) $ is a smoothing of $ (\overline{D}) $. The case of $ n=2k-1 $ for $ k\ge 1 $ is similar. They are all smoothings of the toric configuration \[
		(H-E_1, (k+1)H-kE_1-E-2, E_2, H-E_1-E_2, -kH+(k+1)E_1).\]
		
		The remaining homological log Calabi-Yau divisors are \begin{itemize}
			\item $ (2H-E_2,H-E_1) $,
			\item $ (3H-E_1-2E_2,E_2) $,
			\item $ (H,H-E_1-E_2,H) $,
		\end{itemize}
		which are $ [\omega] $-positive for any normalized reduced symplectic class $ [\omega] $. They are smoothings of the toric homological configuration \[
		(H-E_2,H-E_1,E_1,H-E_1-E_2,E_2),
		\]
		which is also $ [\omega] $-positive for all normalized reduced symplectic class.
	\end{proof}
	
	\begin{rmk}
		Not all symplectic log Calabi-Yau divisors can be obtained from smoothing toric ones. There could be some symplectic classes outside the toric cone so that they don't admit toric divisors. For example, $M_4= \CP^2\# 4\overline{\CP}^2 $ with reduced symplectic class $\delta_1=\delta_2=\delta_3=\delta_4<\frac{1}{4}$. There is no toric divisor in this symplectic class, but there are LCYs such as $ (H-E_1-\dots-E_4,H,H) $. Moreover, even if $[\omega]$ is in the toric cone, there could exsit some LCYs which don't come from the smoothing. For instance,  $M_3= \CP^2\# 3\overline{\CP}^2 $ with reduced symplectic class $(\delta_1,\delta_2,\delta_3)=(\frac{6}{15},\frac{5}{15},\frac{4}{15})$ which is in the region $(8)$ of Proposition \ref{prop:M3} for $i=1$. So the count of toric LCY is 4:
		\[(E_1-E_2,E_2,H_{12},H_3,E_3,H_{13})\]
		\[(E_1-E_3,E_3,H_{13},H_2,E_2,H_{12})\]
		\[(E_2-E_3,E_3,H_{23},H_1,E_1,H_{12})\]
		\[(H_{12},E_1,H_{13},E_3,H_{23},E_2)\]
		However $(H_{12},E_2-E_3,E_1-E_2,H_1,H)$ is a LCY under this symplectic class, which can not be obtained by smoothing the above toric LCYs.
	\end{rmk}
	
	As remarked above, the smoothing operation, which corresponds to nodal trade in ATF, is not enough to generate all the realization of LCY by ATF, even for $M_3$. Nevertheless, when we take almost toric blow up into consideration, we are able to construct ATF realization for divisors in $M_3$.
	\begin{prop}\label{prop:surjective M3}
		The map $ \mc{ATF} (X,\omega)\to \mc{LCY}(X,\omega) $ is surjective for $X=M_3$.
	\end{prop}
	
	\begin{proof}
		Suppose $D\in\mathcal{LCY}(M_3,\omega)$, $[\omega]$ is normalized and reduced with $\omega(E_i)=\delta_i$. Then $D$ is either the toric blow up or non-toric blow up of some $D'\in\mathcal{LCY}(M_2,\omega')$. By the previous proposition, $D'$ has ATF realization as the nodal trade of the toric one $$(-kH+(k+1)E_1,H_1,(k+1)H-kE_1-E_2,E_2,H_{12})$$ for some $k\in \ZZ$. Thus we can assume $D'$ is realized by the almost toric base diagram which is a Delzant polygon with branches at some vertices which could be sufficiently short by nodal slides. See Figure \ref{fig:ATF1}.
		
		\begin{figure}[h]
			\includegraphics*[width=\linewidth]{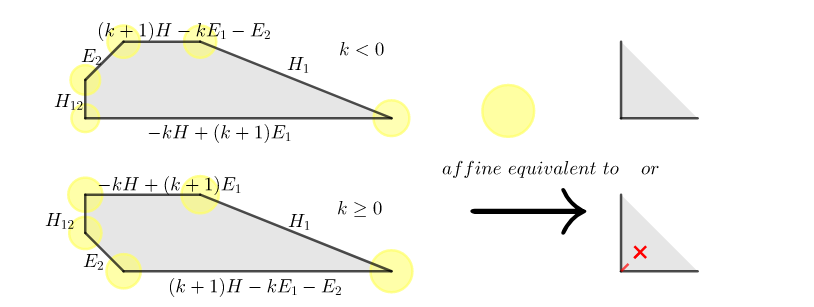}
			\caption{Possible ATF realizations for $M_2$.\label{fig:ATF1}}
		\end{figure}
		
		In the following, we will assume for $k\geq 0$. The case when $k<0$ is totally similar and will be omitted. Let's firstly consider the case when $D$ is the non-toric blow up on the component $C'$ of $D'$. Note that exactly one of the following two cases will happen:
		\begin{itemize}
			\item If $C'$ doesn't come from the smoothing of two (or more) toric components, or comes from the smoothing involving the component $H_1$ or $(k+1)H-kE_1-E_2$ in the toric divisors, we don't need to do mutations. By the reducedness conditions that $\delta_3\leq\delta_2,\delta_3\leq 1-\delta_1-\delta_2$,there always exists enough space for performing almost toric blow up, which realizes the non-toric blow up operation for LCYs. See Figure \ref{fig:ATF2}.
			
			\begin{figure}[h]
				\includegraphics*[width=\textwidth]{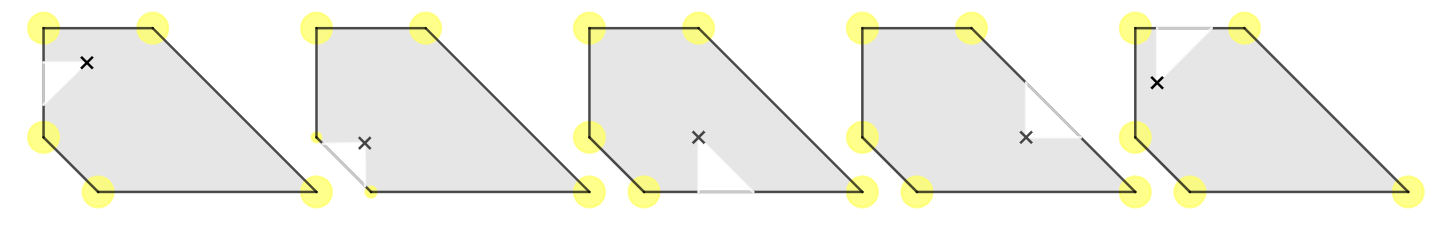}
				\caption{The case when the blow up size is smaller than the length of one edge.\label{fig:ATF2}}
			\end{figure}
			
			\item If $C'$ comes from the smoothing involving the component $H_{12}$ in the toric divisors, when $1=\delta_1+\delta_2+\delta_3$ we can use two types of mutations to create enough space for almost toric blow up. To be more precise, when $C'$ comes from the smoothing involving both $H_{12}$ and $E_2$, we can apply the mutation trick shown in the left graph of Figure \ref{fig:ATF3}; when $C'$ comes from the smoothing involving both $H_{12}$ and $-kH+(k+1)E_1$, we can apply the mutation trick shown in the right graph of Figure \ref{fig:ATF3}.
			
			\begin{figure}[h]
				\includegraphics*[width=\textwidth]{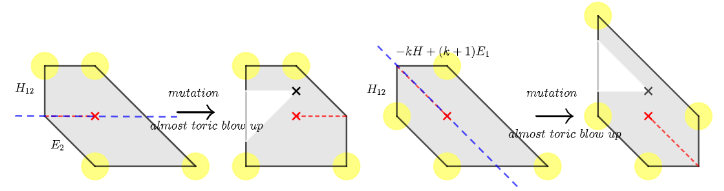}
				\caption{Two kinds of mutations creating enough space for almost toric blow up.\label{fig:ATF3}}
			\end{figure}
		\end{itemize}
		
		Next we assume $D$ is the toric blow up of $D'$, and we want to perform corner chopping of size $\delta_3$ at some vertex $P$ in the almost toric base diagram.
		
		Since $\omega$ is reduced, we have the following easy observation: the size $\delta_3$ can not be larger than the sum of affine lengths of two adjacent edges in the Delzant polygon. This would imply it suffices to consider the following two cases:
		\begin{itemize}
			\item If two edges having $P$ as endpoint both have lengths larger than $\delta_3$, we could just perform nodal slides such that the branches don't intersect the corner of size $\delta_3$ we want to chop.
			\item If one of two edges having $P$ as endpoint is the component $H_{12}$ with $1=\delta_1+\delta_2+\delta_3$ or $E_2$ with $\delta_2=\delta_3$, then we can apply the reflection along $H_{123}$ or $E_2-E_3$ (note that the conditions on $\vec\delta$ guarantee the reflection can be realized by symplectomorphisms) to the divisor $D$ to obtain an equivalent one. In these two cases, $D$ contains no component of $H_{12}$ or $E_2$ which implies there should be no $E_3$ component in the equivalent divisor given by the reflection. It follows that we can also view $D$ as the non-toric blow up of some divisor in $M_2$. Therefore the above discussion for non-toric blow up cases gives the realization.
			\item If one of two edges having $P$ as endpoint is the component $-kH+(k+1)E_1$ with $-k+(k+1)\delta_1\leq \delta_3$, then the other edge having $P$ as endpoint is either $H_{12}$ or $H_1$. In the former case, there might a very special situation when $k=0,\delta_1=\delta_2=\delta_3$. It's shown in Figure \ref{fig:ATF4} that the mutation is forbidden since the eigenray will intersect with another vertex. Nevertheless, we could apply the reflection along $E_1-E_3$ to reduce it to the non-toric blow up cases as what we did in the last paragraph. Other than this special situation, we could always do the mutations to create enough space for corner chopping. The same is true for the latter case. These are shown in Figure \ref{fig:ATF5}.

			\begin{figure}
				\centering
				\begin{minipage}{.3\textwidth}
					\centering
					\includegraphics[width=\linewidth]{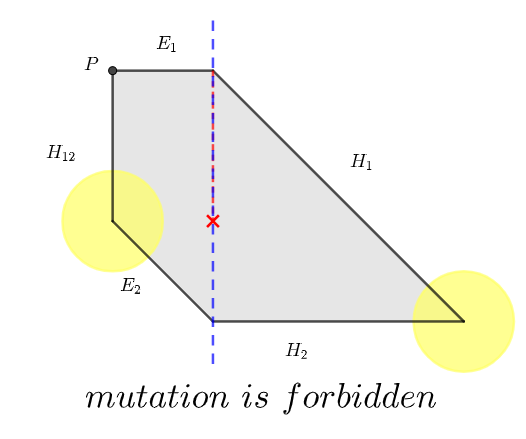}
					\caption{A special situation that mutation is not allowed.\label{fig:ATF4}}
				\end{minipage}%
				\begin{minipage}{.7\textwidth}
					\centering
					\includegraphics[width=\linewidth]{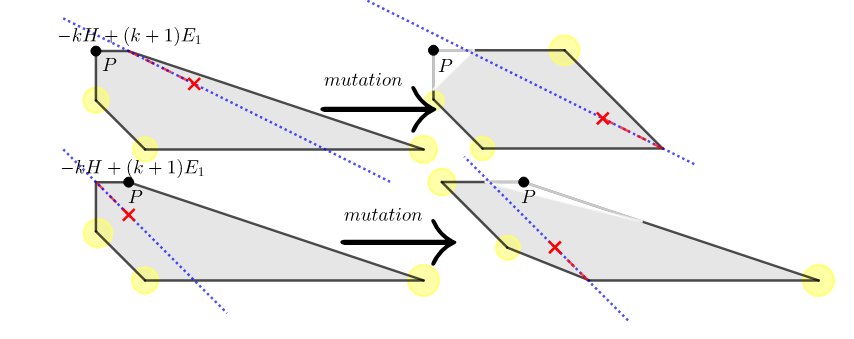}
					\caption{Two kinds of mutations creating enough space for toric blow up.\label{fig:ATF5}}
				\end{minipage}
			\end{figure}
		\end{itemize}

	\end{proof}
	
	\subsection{Connectedness of toric fibrations}
	In \cite{Sym02}, Symington observed that different toric fibrations on $ (S^2\times S^2,\omega) $ can be changed to each other via nodal trades, nodal slides and mutations, which gives a path of almost toric fibrations connecting toric fibrations. This observation led to the following conjecture.
	\begin{conjecture}[\cite{Sym02}]\label{conj:Symington}
		Any two toric fibrations of $ (M,\omega ) $ can be connected by a path of almost toric fibrations of $ (M,\omega ) $.
	\end{conjecture}
	
	Based on the counting results for toric divisors, we prove Symington's conjecture for $ \CP^2,S^2\times S^2,\CP^2\#\overline{\CP}^2,\CP^2\# 2\overline{\CP}^2 $ and $ \CP^2\# 3\overline{\CP}^2 $. We also prove a weaker version of this conjecture for general rational surfaces.
	
	From a toric fibration $ \pi_0 $, we could perform the following series of operations to get a path $ \pi_t $ of almost toric fibrations such that $ \pi_1 $ is another toric fibration. We perform first a nodal trade at a vertex to get a node, slide the node along the eigenray further away from the original vertex, then perform a mutation along this eigenray, which creates a new vertex where the eigenray intersects some edge of the base diagram. If the new vertex is smooth (i.e. satisfies the smoothness condition in Definition \ref{def:delzant}), then one can slide the node towards this vertex and use a nodal trade to get back to a toric fibration. The process is illustrated in Figure \ref{fig:toricmutation}, which we call a \textbf{toric mutation} at the chosen vertex.
	In effect, it cuts the moment polygon into two parts along the eigenray at this vertex, applies an $ AGL(2,\Z) $ transformation to one part and glues it back to form a new moment polygon.
	See Section 2 of \cite{CasVia-embedding} for a more detailed account.
	
	\begin{figure}[h]
		\includegraphics*[width=\linewidth]{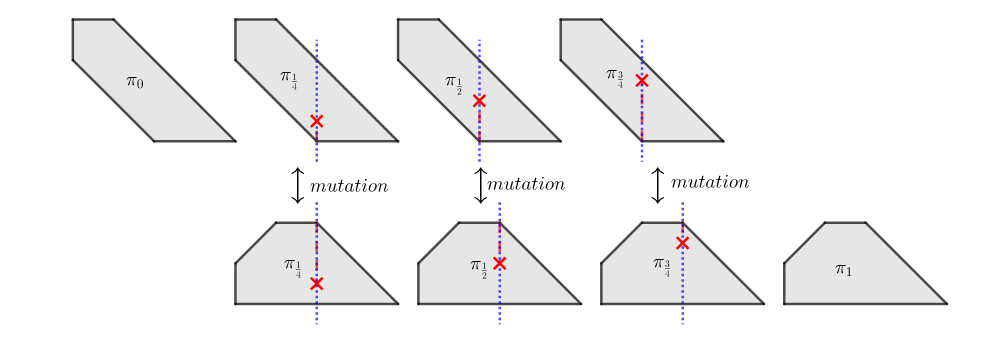}
		\caption{\label{fig:toricmutation}Toric mutation (= nodal trade + mutation + nodal trade) gives a path of almost toric fibrations.}
	\end{figure}
	
	
	\begin{prop}\label{prop:symington minimal}
		Symington's conjecture is true for  $ \CP^2,S^2\times S^2,\CP^2\#\overline{\CP}^2,\CP^2\# 2\overline{\CP}^2 $ and $ \CP^2\# 3\overline{\CP}^2 $.
	\end{prop}
	\begin{proof}
		It suffices to consider normalized reduced symplectic forms on these rational surfaces.
		
		(1) For $ X=\CP^2 $, there is only one toric fibration.
		
		(2) For $ X=S^2\times S^2 $ with standard basis $ \{ B,F\} $ and a normalized reduced symplectic class $ [\omega]=B+\mu F $, the possible Delzant polygons are $ \beta(2n) $ as in Figure \ref{fig:betan}, where $n\in\ZZ_{\geq 0}$. By toric mutation along the eigenray at the lower right vertex, we could always connect $ \beta(2n) $ and $ \beta(2n+2) $.
		
		\begin{figure}[h]
			\includegraphics*[width=\linewidth]{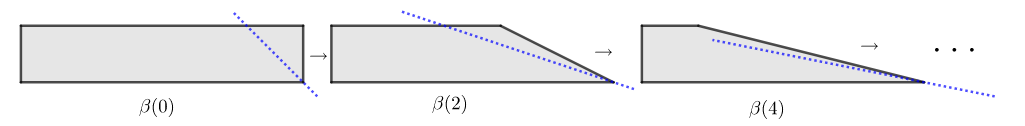}
			\caption{Toric mutations for Delzant polygons of $S^2\times S^2$.\label{fig:betan}}
		\end{figure}
		(3) For $ \CP^2\# \overline{\CP}^2 $ with standard basis $ \{ H,E\} $ and a normalized reduced symplectic class $ [\omega]=H-\mu E $, the possible Delzant polygons are $ \alpha(2n+1) $ as in Figure \ref{fig:alphan}, where $n\in\ZZ_{\geq 0}$. Again, by toric mutation we can connect $ \alpha(2n+1) $ and $ \alpha(2n+3) $.
		\begin{figure}[h]
			\includegraphics*[width=\linewidth]{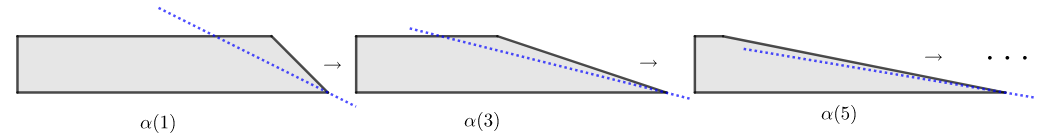}
			\caption{Toric mutations for Delzant polygons of $\CP^2\# \overline{\CP}^2$.\label{fig:alphan}}
		\end{figure}
		
		(4) For $ X= \CP^2\# 2\overline{\CP}^2 $ with a normalized reduced symplectic class $ [\omega]=H-\delta_1E_1-\delta_2E_2 $. By Corollary \ref{cor:toric count M2}, the toric fibrations are given by the family $\gamma (n)$ for $n\in\ZZ_{\geq 0}$ shown in Figure \ref{fig:gamman}. When $n=2k$, $\gamma(2k)$ denotes $((k+1)H-kE_1-E_2,H_1,-kH+(k+1)E_1,H_{12},E_2)$; when $n=2k+1$, $\gamma(2k+1)$ denotes $((k+1)H-kE_1,H_1,-kH+(k+1)E_1-E_2,E_2,H_{12})$. We see that by toric mutation along the eigenray at the lower right vertex, we can connect $\gamma(2k)$ and $\gamma(2k+2)$ or $\gamma(2k+1)$ and $\gamma(2k+3)$. Moreover, by toric mutation at the lower left vertex, we can connect $\gamma(0)$ and $\gamma(1)$.
		\begin{figure}[h]
			\includegraphics*[width=\linewidth]{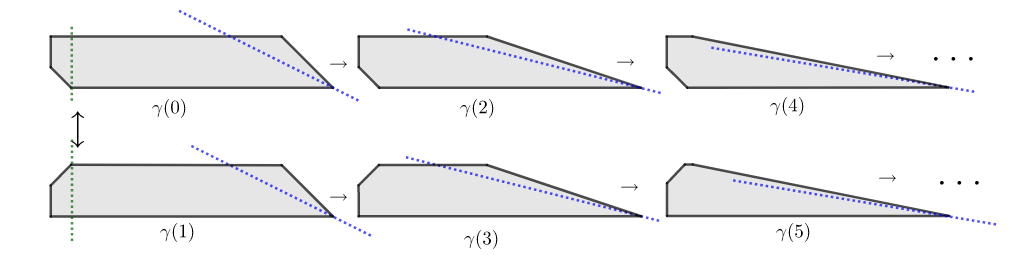}
			\caption{Toric mutations for Delzant polygons of $\CP^2\# 2\overline{\CP}^2$.\label{fig:gamman}}
		\end{figure}
		
		(5) For $ X= \CP^2\# 3\overline{\CP}^2 $, we claim if the toric fibrations come from chopping different corners of the same pentagon $\gamma(n)$, then by toric mutation, they can be connected. If the corners are adjacent, this is shown in Figure \ref{fig:adjacentmutation} in the next section; if not, this is also true simply because in a pentagon whenenver we can perform corner chopping at two non-adjacent corners, we can also perform corner chopping at another corner adjacent to both of them.
		
		Therefore, it suffices to consider the family $\theta (n)$ for $n\in\ZZ_{\geq 0}$, where $\theta(2k)$ denotes $((k+1)H-kE_1-E_2-E_3,E_3,H_{13},-kH+(k+1)E_1,H_{12},E_2)$ and  $\theta(2k+1)$ denotes $((k+1)H-kE_1-E_3,E_3,H_{13},-kH+(k+1)E_1-E_2,E_2,H_{12})$. Thus the Delzant polygons of $\theta (n)$ is obtained by chopping the lower right corner of the Delzant polygons of $\gamma(n)$ in Figure \ref{fig:gamman}. The reason why $\theta(0)$ and $\theta(1)$ can be connected is the same as (4). Moreover, note that if there exists toric fibration $\theta(2k+2)$, the toric fibration $\theta'(2k)$ given by  $((k+1)H-kE_1-E_2,H_{13},E_3,-kH+(k+1)E_1-E_3,H_{12},E_2)$ should also exist by the assumption that $\omega$ is reduced. In the Figure \ref{fig:thetan} we show that $\theta'(2k)$ can be connected to $\theta(2k+2)$ and $\theta'(2k)$ can be connected to $\theta(2k)$ by performing mutations at the lower right vertex. It follows that we can connect all $\theta(n)$ for even $n$ whenever they exist. By the similar observation, we can also check $\theta(2k+3)$ and $\theta(2k+1)$ can be connected. Consequently all the possible toric fibrations $\theta (n)$ can all be connected.
		\begin{figure}[h]
			\includegraphics*[width=\linewidth]{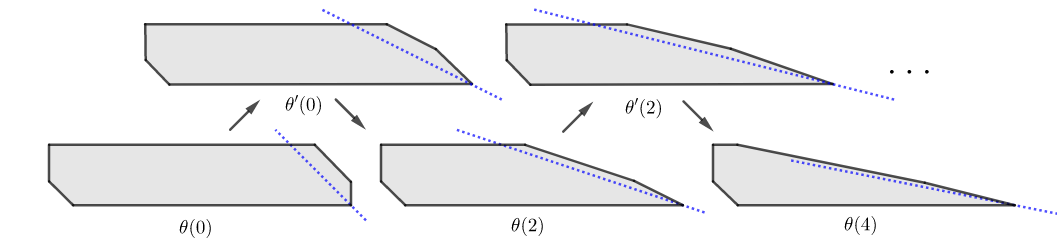}
			\caption{Toric mutations for Delzant polygons of $\CP^2\# 3\overline{\CP}^2$.\label{fig:thetan}}
		\end{figure}
	\end{proof}

	\subsection{Realization and connectedness in the restrictive region}\label{section:verify}
	Recall that the restrictive reduced symplectic classes defined in Section \ref{section:general counting} form a region in the normalized reduced symplectic cone, which is cut out by the inequalities:
	\[1>\delta_1+\cdots+\delta_l \]
	\[\delta_1>\delta_2\]
	\[\delta_k>\delta_{k+1}+\cdots+\delta_l\]
	for all $2\leq k\leq l-1$. One can think of this region as an explicit description of blow up with sufficiently small sizes which has many good properties. In this section we will prove Conjecture \ref{conj:ATF} and \ref{conj:Symington} hold in this region.
	
	Let's firstly discuss Conjecture \ref{conj:ATF}. From the counting result Proposition \ref{prop:general count} we see that in the restrictive region, all the divisors come from the toric and non-toric blow ups of some germs. Our strategy is to choose some appropriate ATF realizations of those germs and then realize each toric or non-toric blow up by toric or almost toric blow up on the almost toric base diagram. In the proof of Proposition \ref{prop:surjective M3} we need to apply toric mutations to overcome the issue that there might be the edge whose length is less than the blow up size (this happens only when $\delta_2=\delta_3$ or $1=\delta_1+\delta_2+\delta_3$). However, when the restrictive condition is added, we will see there is no need to do toric mutations if we select almost toric base diagrams for germs wisely. So we next only need to check whether there is enough space on the diagram to perform toric or almost toric blow up.
	
	It's worth noting that for an almost toric blow up, there are different presentations of excising an affine triangle of on the edge which differ by $AGL(2,\ZZ)$ transformations (See Figure \ref{fig:differentatfblowup}). The only requirement is the directions of any two edges form a $\ZZ$-basis. We say that an almost toric triangle of size $\delta$ can be embedded along an edge if it is possible to excise a size $\delta$ affine triangle in one of such forms.
	
	\begin{figure}[h]
		\includegraphics*[width=\linewidth]{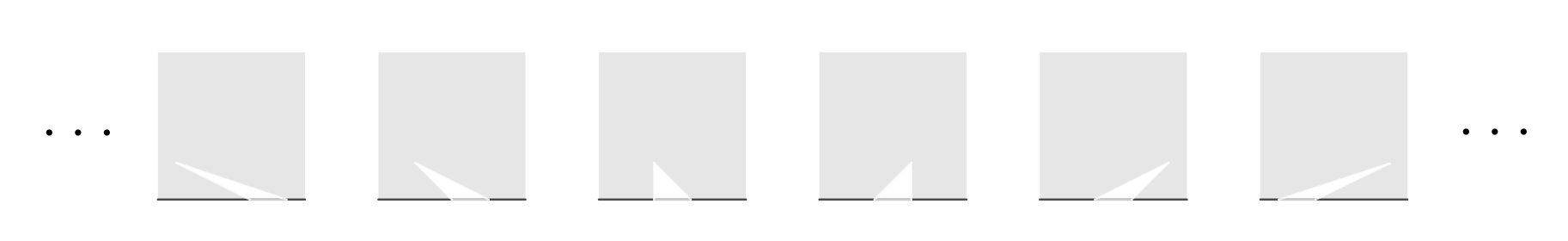}
		\caption{Different embeddings of almost toric triangles.\label{fig:differentatfblowup}}
	\end{figure}
	
	\begin{lemma}\label{lemma:ATFcorner}
		Suppose there are distinct positive numbers $\delta_1>\cdots>\delta_k, \delta_1>\delta_{1,1}>\cdots>\delta_{1,\alpha_1},\cdots,\delta_k>\delta_{k,1}>\cdots>\delta_{k,\alpha_k}$ such that if we order them in the decreasing manner, they satisfy the restrictive condition. Let $\Delta$ be the region obtained by performing corner choppings of sizes  $\delta_1,\cdots,\delta_k$ on the moment map image $\Delta_0$ of the unit open ball in $\CC^2$. Then there is a way to embed almost toric triangles of sizes $\delta_{i,1},\cdots,\delta_{i,\alpha_i}$ along the edge coming from corner chopping of size $\delta_i$ for all $1\leq i\leq k$ such that they are pairwisely disjoint.
	\end{lemma}

	\begin{proof}
		Assume $\Delta_0$ is the region given by $\{(x,y)\,|\,x\geq0,y\geq0,x+y<1\}$. Let $\vec{u}_0=(1,0)$ be a vector. After chopping the corner of size $\delta_1$, we will get a new region $\Delta_1=\{(x,y)\,|\, x\geq0,y\geq0,\delta_1\leq x+y<1\}$, two vertices $A_1=(0,\delta_1),B_1=(\delta_1,0)$ and a unit affine length vector $\vec u_1=\delta_1^{-1}\overrightarrow{A_1B_1}$. For the remaining corner chopping, it makes sense to say whether it is on the left or right of the first corner ($A_1$ side or $B_1$ side). Let $L_1, R_1\subset\{2,\cdots,k\}$ denote the left and right ones respectively. Now by the restrictive condition, it's always possible to choose two points $P_1,Q_1$ on the segment $A_1B_1$ such that the affine lengths of $\overrightarrow{A_1P_1},\overrightarrow{P_1Q_1},\overrightarrow{Q_1B_1}$ are greater than $\sum_{i\in L_1}(\delta_i+\delta_{i,1}+\cdots+\delta_{i,\alpha_i}),\delta_{1,1}+\cdots+\delta_{1,\alpha_1},\sum_{i\in R_1}(\delta_i+\delta_{i,1}+\cdots+\delta_{i,\alpha_i})$ respectively. We then define the parallelogram region $\Theta_1$ attached to the segment $A_1B_1$ by $\{z\in\RR^2\,|\,\overrightarrow{Q_1z}=a\vec u_0+b\overrightarrow{Q_1P_1}, a\in[0,\delta_{1,1}],b\in[0,1]\}$. Note that by our choice of $P_1,Q_1$, $\Theta_1$ will be inside the region $\Delta_1$ and we can embed the pairwisely disjoint almost toric triangles of sizes $\delta_{1,1},\cdots,\delta_{1,\alpha_1}$ into $\Theta_1$. See Figure \ref{fig:atfcorner1}.
		
		After performing the second corner chopping of size $\delta_2$, we will get the region $\Delta_2$. Similarly as above, we can define two new vertices $A_2,B_2$ ($A_2$ is on the left of $B_2$), unit affine length vector $\vec u_2=\delta_2^{-1}\overrightarrow{A_2B_2}$, left and right remaining corner choppoing indices $L_2,R_2\subset\{3,\cdots,k\}$, $P_2,Q_2$ on the segment $A_2B_2$ satisfying the affine lengths of $\overrightarrow{A_2P_2},\overrightarrow{P_2Q_2},\overrightarrow{Q_2B_2}$ are greater than $\sum_{i\in L_2}(\delta_i+\delta_{i,1}+\cdots+\delta_{i,\alpha_i}),\delta_{2,1}+\cdots+\delta_{2,\alpha_2},\sum_{i\in R_2}(\delta_i+\delta_{i,1}+\cdots+\delta_{i,\alpha_i})$ respectively. The selection of the parallelogram region $\Theta_2$ attached to the segment $A_2B_2$ needs more words. When $2\in L_1$, that is, the second corner chopping is at $A_1$, we define $\Theta_2$ to be $\{z\in\RR^2\,|\,\overrightarrow{Q_2z}=a\vec u_1+b\overrightarrow{Q_2P_2}, a\in[0,\delta_{2,1}] ,b\in[0,1]\}$; when $2\in R_1$, $\Theta_2$ will be $\{z\in\RR^2\,|\,\overrightarrow{Q_2z}=a\vec u_0+b\overrightarrow{Q_2P_2}, a\in[0,\delta_{2,1}],b\in[0,1]\}$. See Figure \ref{fig:atfcorner2}.
		
		\begin{figure}[h]
			\centering
			\begin{minipage}{.5\textwidth}
				\centering
				\includegraphics[width=\linewidth]{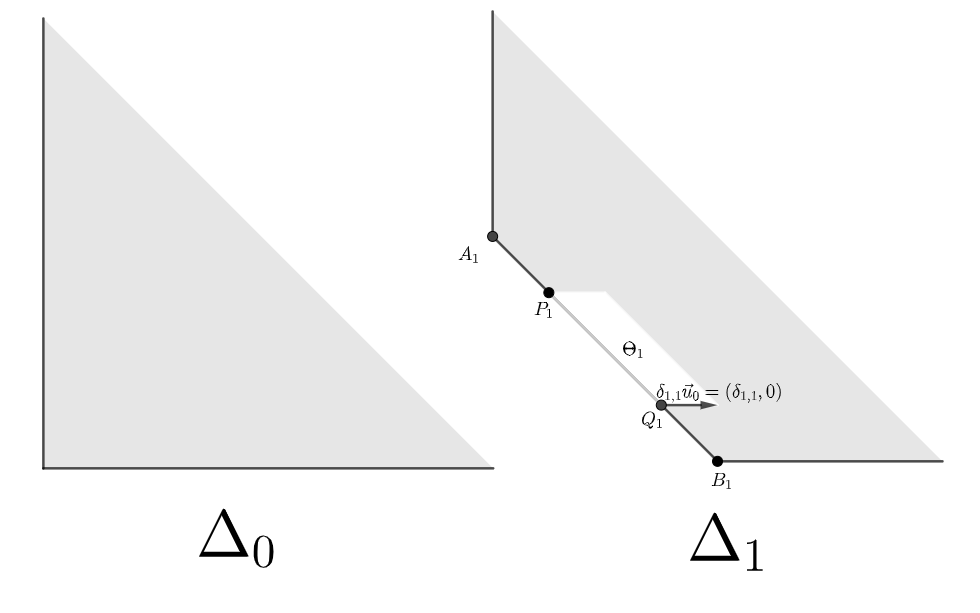}
				\caption{Embedding the first parallelogram region $\Theta_1$. \label{fig:atfcorner1}}
			\end{minipage}%
			\begin{minipage}{.5\textwidth}
				\centering
				\includegraphics[width=\linewidth]{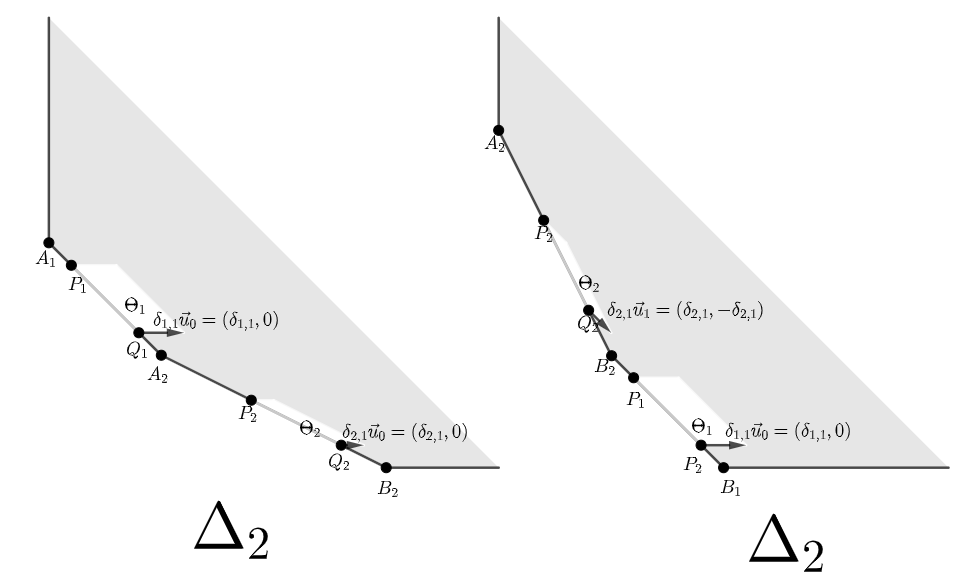}
				\caption{Two cases of embedding the second parallelogram region $\Theta_2$.\label{fig:atfcorner2}}
			\end{minipage}
		\end{figure}

		We could repeat this procedure for all the regions $\Delta_3,\cdots,\Delta_k=\Delta$ obtained by the remaining corner choppings. The only thing that needs to be mentioned is the choice of $\Theta_j$. Assume from $\Delta_{j-1}$ to $\Delta_j$, the new edge $A_jB_j$ come from the chopping the corner of $\Delta_{j-1}$ whose right edge has the unit affine length direction $\vec u_{\phi(j)}$ ($\phi(j)$ can be an arbitrary positive integer less than $j$). We will let $\Theta_j= \{z\in\RR^2\,|\,\overrightarrow{Q_jz}=a\vec u_{\phi(j)}+b\overrightarrow{Q_jP_j}, a,b\in[0,1]\}$. Note that this choice is consistent with the requirement of $\Delta_2$ in the above paragraph.
		
		Therefore we get these  parallelogram regions $\Theta_1,\cdots,\Theta_k$. It's easy to see from the restrictive condition and our choice that $\Theta_j$ will not intersect the corner we will chop away from $\Delta_j$ to get $\Delta_{j+1}$. So these parallelogram regions actually all live in the final region $\Delta_k=\Delta$. To finish the proof of this lemma, it suffices to show these $\Theta_j$'s are pairwisely disjoint by our construction. This can be verified by the observation that, for any $i<j$, $\Theta_i$ and $\Theta_j$ will be contained in two different half planes divided by the line $\ell$ of the opposite sides of $P_jQ_j$ in the parallelogram $\Theta_j$. Note that we can consider the line of the segment $\ell'$ and the intersection point $S$ between $\ell$ and $\ell'$. By our choice of $P_i,Q_i$, if $\Theta_i$ is on the left (resp. right) of $\Theta_j$, then $S$ must be on the left of $P_i$ (resp. right of $Q_i$). See Figure \ref{fig:Thetadisjoint}.
		
		\begin{figure}[h]
			\includegraphics*[width=\linewidth]{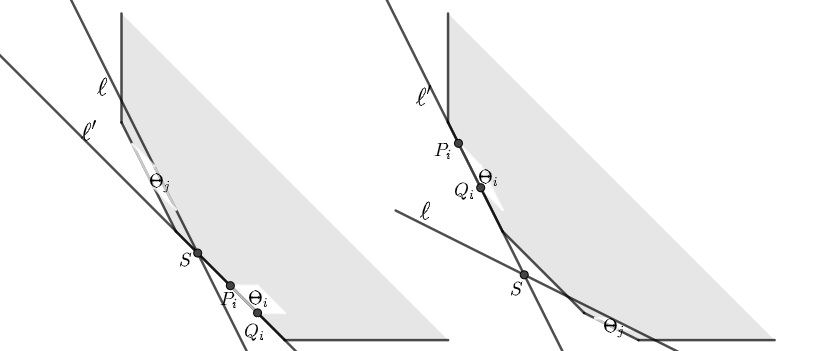}
			\caption{$\Theta_j$ is on the left or right of $\Theta_i$. In both cases the line $\ell$ will always seperate two parallelogram regions. \label{fig:Thetadisjoint}}
		\end{figure}
		
	\end{proof}

	Now let's choose the almost toric base diagrams for the germs in the proof of Proposition \ref{prop:general count}. We can put $((k+1)H-kE_1,H_1,-kH+(k+1)E_1,H_1),((k+2)H-(k+1)E_1,-kH+(k+1)E_1,H_1),((k+3)H-(k+2)E_1,-kH+(k+1)E_1)$ with $k\geq0$ together since their almost toric base diagrams only differ by nodal trades, see Figure \ref{fig:germrealization}. The remaining germs are $(H,H,H_1),(2H-E_1,H),(2H,H_1)$ and $(3H-E_1-\cdots-2E_j,E_j)$ whose almost toric base diagrams are shown in Figure \ref{fig:germrealizationn}. Note that we have given the ATF realizations we will use for all the germs.
	
	\begin{figure}[!]
		\centering
		\begin{minipage}{.5\textwidth}
			\centering
			\includegraphics[width=\linewidth]{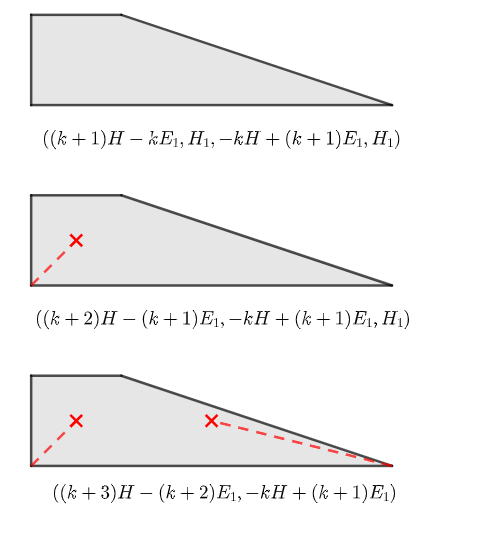}
			\caption{ATF realizations of some germs.\label{fig:germrealization}}
		\end{minipage}%
		\begin{minipage}{.5\textwidth}
			\centering
			\includegraphics[width=\linewidth]{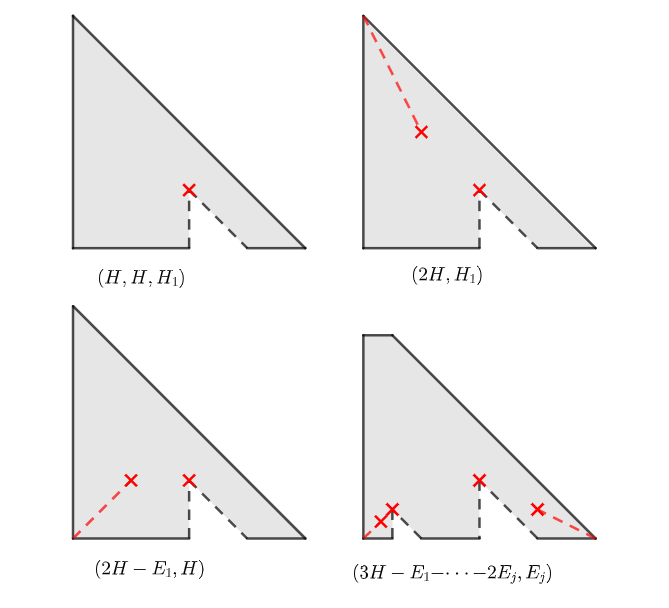}
			\caption{ATF realizations of the remaining germs.\label{fig:germrealizationn}}
		\end{minipage}
	\end{figure}

	\begin{proposition}\label{prop:proof of conjecture 5.3}
		For any rational surface $(M_l,\omega)$, if $[\omega]$ is in the restrictive region, then the map $ \Phi:\mc{ATF}(M_l,\omega)\to \mc{LCY}(M_l,\omega) $ is surjective.
	\end{proposition}
	\begin{proof}
		For the first diagram in Figure \ref{fig:germrealization}, let $X_1,X_2,X_3,X_4\subset\{2,\cdots,l\}$ be the set of indices of non-toric blow ups over the four edges and $Y_1,Y_2,Y_3,Y_4\subset\{2,\cdots,l\}$ be the set of indices of toric blow ups at four corners. Again, similar to the proof of Lemma \ref{lemma:ATFcorner}, by the restrictive condition we can embed pairwisely disjoint parallelogram regions $\Phi_p$ with affine width $\sum_{i\in X_p}\delta_i$, affine height $\max_{i\in X_p}\delta_i$ and triangles $\Psi_p$ of sizes $\sum_{i\in Y_p}\delta_i$ for all $1\leq p\leq 4$. Then we apply Lemma \ref{lemma:ATFcorner} to the regions inside $\Psi_p$ to realize all the remaining non-toric blow ups over the edges with indices in $Y_p$ by almost toric blow ups. It follows that we are able to create the almost toric base diagram for the ultimate divisor, see Figure \ref{fig:ATFrealizationgeneral}. The other two cases in Figure \ref{fig:germrealization} and the cases in \ref{fig:germrealizationn} can be settled in the same way. Observe that in each of those cases, the restrictive condition guarantees each edge has enough length to perform blow ups so that we don't need to do toric mutations. Also, by nodal slides, we can make the nodal points close enough the vertices so that the nodal rays don't intersect with those parallelogram or triangle regions.
	\end{proof}

	\begin{figure}[h]
		\includegraphics*[width=\linewidth]{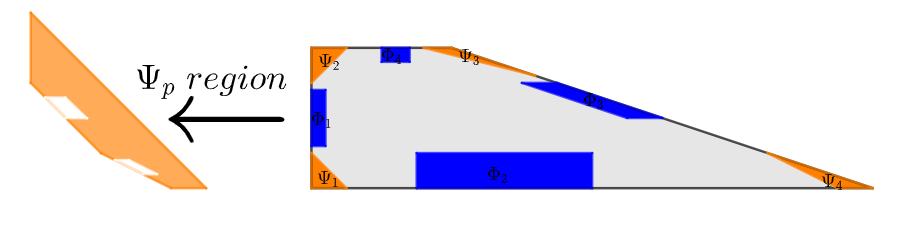}
		\caption{Arrangement of $\Psi_p,\Phi_p$ regions. When we zoom into the $\psi_p$ region we will get the diagram that appears in Lemma \ref{lemma:ATFcorner}. \label{fig:ATFrealizationgeneral}}
	\end{figure}

	Now let's prove Conjecture \ref{conj:Symington} in the restrictive region.
	
	\begin{lemma}\label{lemma:connecting different corner chopping}
		Suppose $\vec{\delta}'=(\delta_1,\cdots,\delta_k)$ and $\vec{\delta}=(\delta_1,\cdots,\delta_k,\delta_{k+1})$ are in the restrictive regions. Let $ \pi':(X,\omega_{\vec{\delta}'})\to \Delta' $ be a toric fibration. Then if $ \pi_i:(X\# \overline{\CP}^2,\omega_{\vec{\delta}})\to \Delta_i $ ($i=1,2$) are two toric fibrations obtained from performing corner chopping at different corners of $\Delta'$, then they can be connected by a path of almost toric fibrations.
	\end{lemma}
	\begin{proof}
		By the assumption of $\vec{\delta}=(\delta_1,\cdots,\delta_k,\delta_{k+1})$ being restrictive, it's immediate to see from the list in $t\mathcal{H}_k$ before Corollary \ref{cor:ingredients} that the only possible homology class of the boundary divisor's component of $\Delta'$ with symplectic area $\leq\delta_{k+1}$ must be
		\[-aH+(a-1)E_1-\sum_{i=2}^k \epsilon_i E_i, a\in\ZZ_{+},\epsilon_i\in\{0,1\}\] This implies there is at most one edge of $\Delta'$ whose length is not larger than $\delta_{k+1}$. Consequently it suffices to show that corner chopping at two adjacent corners can be connected. And this can be achieved by toric mutation as in Figure \ref{fig:adjacentmutation}.
	\end{proof}
	
	\begin{figure}[h]
		\centering
		\begin{minipage}{.5\textwidth}
			\centering
			\includegraphics[width=\linewidth]{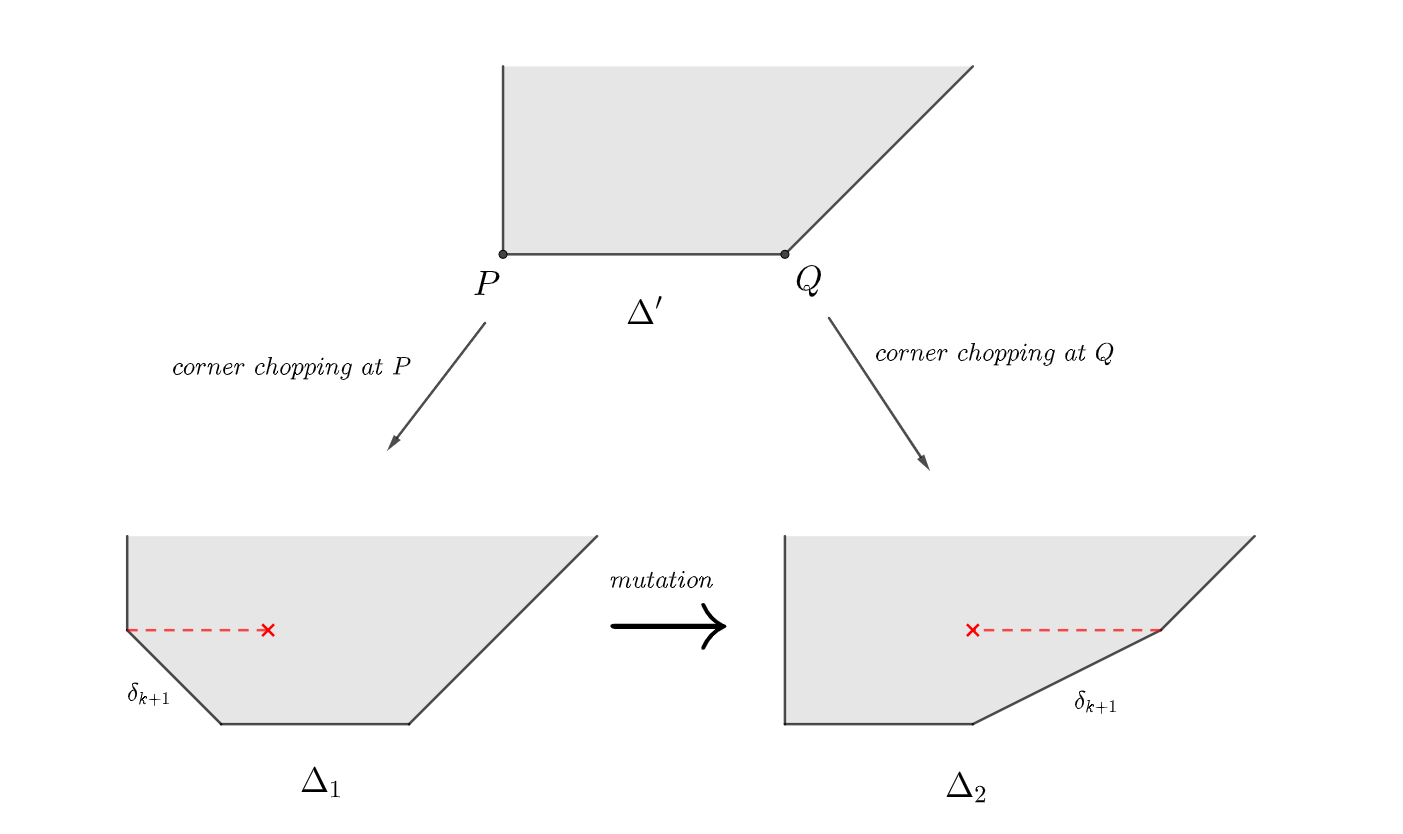}
			\caption{Adjacent corner choppings differ by a mutation.\label{fig:adjacentmutation}}
		\end{minipage}%
		\begin{minipage}{.5\textwidth}
			\centering
			\includegraphics[width=\linewidth]{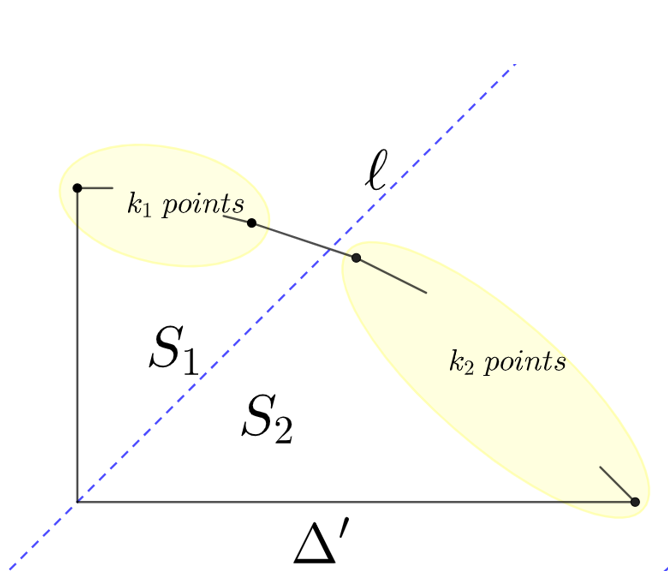}
			\caption{The eigenray $\ell$ divides the plane into two parts and must touch another edge.\label{fig:notintersectingl}}
		\end{minipage}
	\end{figure}

	\begin{lemma}\label{lemma:chopping corner not intersecting eigenray}
		Under the assumption of Lemma \ref{lemma:connecting different corner chopping} with additional hypothesis $k\geq3$, if $\pi':(X,\omega_{\vec{\delta}'})\to \Delta' $ can be connected to another toric fibration by mutation along an eigenray $\ell$, then it's always possible to chop some corner with size $\delta_{k+1}$ not intersecting $\ell$.
	\end{lemma}
	\begin{proof}
		Suppose $\ell$ divides the plane into two open half-planes $S_1,S_2$, which contains $k_1,k_2$ vertices of the Delzant polygon $\Delta'$ respectively. Then it's immediate to see that $k_1+k_2=k+2$ since $\ell$ does not pass through another vertex when the mutation is allowed. This is shown in Figure \ref{fig:notintersectingl}. Remember that the restrictive condition guarantees there is at most one edge of $\Delta'$ whose length is not larger than $\delta_{k+1}$. Assuming $\ell$ intersect the edge $L$, a triangle of size $\delta_{k+1}$ can always embedded at the corners of those $k_1+k_2$ vertices which are not the endpoints of $L$ or the shortest edges. Note that there will be at least $k+2-2-2=k-2\geq1$ such corners where we can perform corner chopping of size $\delta_{k+1}$ without intersecting $\ell$.
	\end{proof}

	\begin{proposition}\label{prop:proof of conjecture 5.7}
		For rational surface $(M_l,\omega_{(\delta_1,\cdots,\delta_{l})})$, if $(\delta_1,\cdots,\delta_{l})$ is in the restrictive region, then any two toric fibrations can be connected by a path of almost toric fibrations. As a result, on $ (X,\omega) $ with arbitrary $\omega$, any two toric fibrations are symplectic deformation equivalent to some toric fibrations on $(X,\omega')$ (in the sense of their boundary divisors) which can be connected through a path of almost toric fibrations.
	\end{proposition}
	\begin{proof}
		By Proposition \ref{prop:symington minimal}, we only need to look at the cases when $l\geq 4$. By induction we will assume any two toric fibrations on $(M_{l-1},\omega_{(\delta_1,\cdots,\delta_{l-1})})$ can be connected by toric mutations and prove the case for $l$. For any two toric fibrations $\pi,\pi'$ on $(M_l,\omega_{(\delta_1,\cdots,\delta_{l})})$ which come from performing the corner chopping of two toric fibrations $\overline{\pi},\overline{\pi}'$ on $(M_{l-1},\omega_{(\delta_1,\cdots,\delta_{l-1})})$ respectively, there will be a sequence of toric fibrations $\pi_1=\overline{\pi},\pi_2,\cdots,\pi_{n-1},\pi_n=\overline{\pi}'$ on $(M_{l-1},\omega_{(\delta_1,\cdots,\delta_{l-1})})$ such that $\pi_i$ and $\pi_{i+1}$ can be connected by performing one toric mutation for all $1\leq i\leq n-1$.  By Lemma \ref{lemma:chopping corner not intersecting eigenray}, there exist toric fibrations $\tilde{\pi}_i,\tilde{\tilde{\pi}}_{i}$ on $(M_l,\omega_{(\delta_1,\cdots,\delta_{l})})$ which are toric blow up of $\pi_{i},\pi_{i+1}$ and can be connected. As a result, by Lemma \ref{lemma:connecting different corner chopping} we see that $\tilde{\pi}_1,\tilde{\tilde{\pi}}_{n-1}$ can be connected and $\pi,\pi'$ can be connected to $\tilde{\pi}_1,\tilde{\tilde{\pi}}_{n-1}$  respectively. Therefore we finally obtain that $\pi$ can be connected to $\pi'$.
		
		Now given two toric fibrations on $(X,\omega)$ with toric divisors $D_1,D_2$, by Lemma \ref{lemma:star shaped}, there are always symplectic deformations $ (\omega_i^t,D^t_i) $ such that $ \omega_i^0=\omega $ and $ \omega_i^1=\omega' $ ($i=1,2$) with $\omega'$ satisfying the restrictive condition.
		Then the the toric fibrations on $ X $ corresponding to $ (\omega',D^1_1),(\omega',D^1_2) $ can be connected by a path of almost toric fibrations.
	\end{proof}
	
	
	\appendix
	\section{Appendix}
	\subsection{More discussions on tautness}\label{section:appendixtaut}
	
	In this subsection,  we study the def-tautness. Hopefully the result can be applied to study the problem of classifying symplectic fillings of torus bundles in the future. 
	
	Labeled holomorphic anticanonical pairs are in general not taut (\cite{Fr}). The following example of Golla and Lisca shows that the labeled symplectic log Calabi-Yau divisors are also in general not def-taut.
	\begin{example}[Proposition 3.6 of \cite{GoLi14}]
		The cycle of spheres with self-intersections \[
		(1,-2,-3,-3,-2,-3,-2)\]
		admits two different symplectic embeddings (with possibly different symplectic structures) in $ X=\CP^2\# 9\overline{\CP}^2 $ as symplectic log Calabi-Yau divisors $ (\omega_1,D_1) $ and $ (\omega_2,D_2) $. Golla and Lisca showed that their complements have different intersection forms and thus not homotopy equivalent. So the two divisors are not symplectic deformation equivalent (in both labeled and unlabeled senses).
	\end{example}

	The example given by Golla and Lisca is included in the classification of log Calabi-Yau divisors with $b^{+}=1$ up to toric equivalence in Theorem 1.4 (4) of \cite{LiMaMi-logCYcontact}. We can actually give a criterion of def-tautness for this type of divisors. Recall that in Theorem 1.4 (4) of \cite{LiMaMi-logCYcontact}, a blown-up $(1,1-p_1,-p_2,\cdots,-p_{l-1},1-p_l)$ means this self-intersection sequence is obtained by toric blow up of $(1,1,1)$ and then followed by non-toric blow up. We could give an equivalent description of this type of sequences following \cite{Lis08-lens}. Let
	\[\mathcal{C}=\{(n_1,\cdots,n_k,n_{k+1}):\sum_{i=1}^{k+1}n_i=3(k-1),[n_1,\cdots,n_k]=0,(n_1,\cdots,n_k)\text{ is admissible}\}\]
	\[\mathcal{C}'=\{(n_1,\cdots,n_k,n_{k+1},a_1,\cdots,a_{k+1}):(n_1,n_2,\cdots,n_k,n_{k+1})\in\mathcal{C},a_i\in\ZZ_{\geq 0},n_i+a_i\geq 2\}\]
	where $[n_1,\cdots,n_k]$ denotes the continued fraction and 'admissible' means each denominator appearing in the continued fraction is positive.
	By Lemma 2.2 in \cite{Lis08-lens}, there is a surjective map \[\Phi: \mathcal{C}'\rightarrow \{\text{blown-up }(1,1-p_1,-p_2,\cdots,-p_{k},1-p_{k+1}),p_i\geq 2\}\] sending $(n_1,\cdots,n_k,n_{k+1},a_1,\cdots,a_{k+1})$ to $(1,1-(n_1+a_1),-(n_2+a_2),\cdots,-(n_k+a_k),1-(n_{k+1}+a_{k+1}))$. For instance, $(1,-2,-3,-3,-2,-3,-2)$ in the above example is equal to $\Phi(2,1,3,2,1,3,1,2,0,0,2,0)$ or $\Phi(3,1,3,1,3,1,0,2,0,1,0,2)$.
	
	\begin{proposition}\label{prop:criterion of tautness}
		A labeled symplectic log Calabi-Yau divisor $(D=(C_0,C_1,\cdots,C_l),\omega)$ with blown-up self-intersection sequence $(1,1-p_1,-p_2,\cdots,-p_{l-1},1-p_l)$ such that all $p_i\geq 2$ is def-taut if and only if the preimage of the self-intersection sequence under $\Phi$ has only one element.
	\end{proposition}
	\begin{proof}
		Firstly we show that the homology class of any component of self intersection $1$ in a symplectic log Calabi-Yau divisor with length $\geq 2$ can be mapped to the standard line class $H$ by some $\gamma\in D(X)$, where $D(X)$ is the image of $Diff^+(X)\rightarrow Aut(H^2(X;\ZZ))$. Note that with standard basis $H,E_1,\cdots,E_l$, the reflections along $E_i-E_j,H-E_i-E_j-E_k$ are in $D(X)$. By Corollary \ref{cor:ingredients}, we may assume the self intersection $1$-component is one of the following
		\[kH-(k-1)E_1-\sum_{i=2}^l \epsilon_i E_i, k\in\ZZ,\epsilon_i\in\{0,1\}\]
		\[2H-\sum_{i=2}^l \epsilon _i E_i, \epsilon_i\in\{0,1\}\]
		\[3H-E_1-\cdots-E_{p-1}-2E_p-\sum_{i=p+1}^l\epsilon _i E_i, 2\leq p\leq l,\epsilon_i\in\{0,1\}\]
		For the second and the third cases, we can firstly perform the reflections along some $E_1-E_i$ for the $i$ such that $\epsilon_i=1$ or $E_1-E_p$, so that these homology classes can be transformed into the first type. For the first type, we could perform the reflections along $H-E_1-E_i-E_j$ for $i,j$ such that $\epsilon_i=\epsilon_j=1$ to decrease $k$ by $1$. By repeating this procedure we will get $k=1$ and the homology class must be $H$.
		
		
		Since the symplectic deformation equivalent class will not changed by a diffeomorphism, by the above claim we may just assume $[C_0]=H$. By Proposition 4.4 in \cite{Lis08-lens}, it follows that $[C_1]$ has the form of $H-\sum_{i=2}^l \epsilon _i E_i, \epsilon_i\in\{0,1\}$, $[C_i]$ with $1< i<l$ has the form of $E_{\alpha_i}-\sum_{\beta>\alpha_i} \epsilon_{\beta}E_{\beta}, \epsilon_{\beta}\in\{0,1\}$. Define $a_i=\#\{E_j\, | \,E_j\cdot[C_i]=1,E_j\cdot[C_k]=0\text{ for other } k\}$. It follows that $\Phi(p_1-a_1,\cdots,p_{l}-a_{l},a_1,\cdots,a_l)=(1,1-p_1,-p_2,\cdots,-p_{l-1},1-p_l)$. Therefore each homological configuration will correspond to some element in the preimage of $\Phi$ according to the number of non-toric blow up times at each position. Conversely, if there are two different elements in the preimage, then the homological configurations corresponding to them can not be only differed by a homological action of diffeomorphism. Since otherwise, the diffeomorphism maps $H$ to $H$, which implies its homological action must be some permutations of $E_i$'s. It's easy to see the value $a_i$ will not change after the permutations of $E_i$'s. As a result, the two divisors corresponding to those two homological configurations are not symplectic deformation equivalent by Theorem \ref{thm: symplectic deformation class=homology classes} but have the same self-intersection sequence, which means the labeled divisor is not def-taut.
	\end{proof}
	
	For the completeness, we also determine the def-tautness for all the other self-intersection sequences with $b^+=1$ in Theorem 1.4 of \cite{LiMaMi-logCYcontact}. Those are $(1,p),(-1,-p)$ with $p=1,2,3$, $(1,1,p)$ with $p\leq1$, $(0,p)$ with $p\leq 4$ and $(1,p)$ with $p\leq 1$.
	
	\begin{proposition}\label{prop:criterion of tautness for others}
		All the labeled symplectic log Calabi-Yau divisors with the above self-intersection sequences are def-taut.
	\end{proposition}
	\begin{proof}
		One could use the same argument as the proof of Proposition \ref{prop:criterion of tautness} to show that if $A\in\mathcal{H}_l$ and $A\neq3H-E_1-\cdots-E_l$, then when $A^2=0$, it could be transformed to $H-E_1$ (or $F$ when the manifold is $S^2\times S^2$) by reflections and when $A^2=-1$ it could be transformed to $E_1$ by reflections unless $A=H-E_1-E_2$ with $l=2$. Note that those reflections along $E_i-E_j,H-E_i-E_j-E_k$ preserve $c_1(M_l)=3H-E_1-\cdots-E_l$. So when the length of the sequence is $2$, if we could transform one component by reflections, the other one will be transformed as well. This shows that labeled divisors of self-intersection sequences $(1,p),(0,p)$ are def-taut, and labeled divisors of self-intersection sequences $(-1,-p)$ are def-taut unless the manifold is $M_2$, in which case $p=-4$. Therefore we have checked all the length $2$ sequences above are def-taut. For the remaining case $(1,1,p)$, note that if we transform one of the self-intersection $1$-component into $H$, the other self-intersection $1$-component must also be $H$. This is simply because $A\cdot H=1,A^2=1$ has the only solution $A=H$. This shows the homological configurations of the labeled divisors of self-intersection sequences $(1,1,p)$ can always be transformed into $(H,H,H-E_1-\cdots-E_l)$ which implies labeled divisors of self-intersection sequences $(1,1,p)$ are def-taut.
	\end{proof}
	
	\begin{proposition}\label{prop:toricinvariancetaut}
		The labeled symplectic log Calabi-Yau divisors with self-intersection sequence $(s_1,\cdots,s_k)$ are def-taut if and only if the labeled symplectic log Calabi-Yau divisors with the toric blow up self-intersection sequences (which are $(-1,s_1-1,\cdots,s_k-1),(s_1-1,-1,s_2-1,\cdots,s_k),\cdots,(s_1,\cdots,s_{k-1}-1,-1,s_k-1),(s_1-1,\cdots,s_k-1,-1)$) are def-taut.
	\end{proposition}
	\begin{proof}
		We only need to consider $(-1,s_1-1,\cdots,s_k-1)$, the others are same. Assume the divisors of self-intersection sequences $(s_1,\cdots,s_k),(-1,s_1-1,\cdots,s_k-1)$ are in $M_{l-1},M_{l}$ with reduced blowup forms. We adopt the canonical identification $H_2(M_{l};\ZZ)=H_2(M_{l-1};\ZZ)\oplus \ZZ E_l$ and always identify symplectic deformation equivalence and homological equivalence by Theorem \ref{thm: symplectic deformation class=homology classes}. We will use the following convenient result, which is actually implied in Section 4.1.2 of \cite{LiWu12-lagrangian} : let $K_{l-1},K_l$ be the standard canonical classes, if $\gamma,\gamma'$ are integral isometries of $ H_2(M_{l-1};\ZZ),H_2(M_{l};\ZZ)$ respectively and $\gamma'=\gamma\oplus id_{E_l}$, then $\gamma\in D_{K_{l-1}}(M_{l-1})$ if and only if $\gamma'\in D_{K_l}(M_l)$.
		
		On the one hand, if the labeled divisors of self-intersection sequence $(-1,s_1-1,\cdots,s_k-1)$ are def-taut and $D=(C_1,\cdots,C_k),D'=(C_1',\cdots,C_k')$ are two labeled divisors of self-intersection sequence $(s_1,\cdots,s_k)$, then there must exist divisors on $M_{l}$ with homological configurations $(E_l,[C_1]-E_l,\cdots,[C_k]-E_l),(E_l,[C_1']-E_l,\cdots,[C_k']-E_l)$ by Corollary \ref{cor:ingredients}. And there is some homological action $\gamma\in D_{K_l}(M_l)$ such that $\gamma(E_l,[C_1]-E_l,\cdots,[C_k]-E_l)=(E_l,[C_1']-E_l,\cdots,[C_k']-E_l)$. $\gamma(E_l)=E_l$ implies $\gamma([C_1],\cdots,[C_k])=([C_1'],\cdots,[C_k'])$ and $\gamma|_{H_2(M_{l-1};\ZZ)}$ is an isomorphism of the lattice $H_2(M_{l-1};\ZZ)$ and thus belongs to $D_{K_{l-1}}(M_{l-1})$ by the claim. This shows $D,D'$ are symplectic deformation equivalent.
		
		On the other hand,  if the labeled divisors of self-intersection sequence $(s_1,\cdots,s_k)$ are def-taut and $D=(C_0,C_1,\cdots,C_k),D'=(C_0',C_1',\cdots,C_k')$ are two labeled divisors of self-intersection sequence $(-1,s_1-1,\cdots,s_k-1)$. As the proof of the previous Proposition, unless $l=2$ we may assume $[C_0]=[C_0']=E_l$. It then follows that $([C_1]+E_l,\cdots,[C_k]+E_l),([C_1']+E_l,\cdots,[C_k']+E_l)$ are homological configurations with classes in $H_2(M_{l-1};\ZZ)$. So there is some $\gamma\in D_{K_{l-1}}(M_{l-1})$ such that $\gamma([C_1]+E_l,\cdots,[C_k]+E_l)=([C_1']+E_l,\cdots,[C_k']+E_l)$. Thus by the claim $\gamma'=\gamma\oplus id_{E_l}\in D_{K_l}(M_l)$ will satisfy $\gamma'([C_0],[C_1],\cdots,[C_k])=([C_0'],[C_1'],\cdots,[C_k'])$. This shows that $D,D'$ are symplectic deformation equivalent.
	\end{proof}
	
	Combining the above Propositions with Theorem 1.4 of \cite{LiMaMi-logCYcontact}, we already determined all the def-taut labeled divisors with $b^+=1$.

	\begin{remark}
		There are some differences between the symplectic def-tautness and holomorphic tautness. From the definition it's easy to see the holomorphic tautness implies symplectic def-tautness, but the converse is not true. In \cite{Fr} the sequence $(4,0)$ is considered to be not taut, but in our notion it is def-taut. This is because $(4,0)$ is the (only) self-intersection sequence having embeddings into two different manifolds $S^2\times S^2$ or $\CC\PP^2\#\overline{\CC\PP}^2$ and the definition of tautness in the holomorphic category does not fix the ambient manifolds. Other than this, the self-intersection sequence $(2,2)$ is not holomorphic taut since it has realizations in $S^2\times S^2$ with two non-isomorphic complex structures (namely $\mathbb{F}_0$ and $\mathbb{F}_2$, see \cite{Fr}). However it's not hard to check the def-tautness for divisors with self-intersection sequence $(2,2)$.
	\end{remark}
	
	\begin{remark}
		For the purpose of completely determining the diffeomorphism types of symplectic fillings of torus bundles, one should define the def-tautness for unlabeled divisors. However it's difficult to obtain the toric blow up invariance for unlabeled def-tautness analogous to Proposition \ref{prop:toricinvariancetaut} due to the cyclic symmetry.
	\end{remark}

	\subsection{Details of counting toric divisors in $M_3$}\label{section:detail1}
	\begin{proof}[Proof of Proposition \ref{prop:M3}:]
		The proof is very similar to that of Proposition \ref{prop:general count M2}. We start by listing the possible self-intersection sequences to be \begin{enumerate}[label=$ (\arabic*) $]
			\item $ (-n,0,n-2,-1,-2,-1) $,
			\item $ (-n,0,n-1,-2,-1,-2) $,
			\item $ (-n,-1,-1,n-2,-1,-1) $,
			\item $ (-1,-1,-1,-1,-1,-1) $,
		\end{enumerate}
		where $ n\ge 2 $. Also, to make the presentation easier, we write $ H_i=H-E_i $, $ H_{ij}=H-E_i-E_j $ and $ H_{ijk}=H-E_i-E_j-E_k $, for distinct $ i,j,k $.
		
		We denote by $ tF_n(3;1,\delta_1,\delta_2,\delta_3) $ the number of toric log Calabi-Yau divisors in $ (1) $-$ (3) $ and by $ tF_1(3;1,\delta_1,\delta_2,\delta_3) $ the number of toric log Calabi-Yau divisors in $ (4) $. We could make use of Proposition \ref{prop:spherical classes} to enumerate the homology classes with self-intersections listed above.
		
		Let's first consider $ (1)-(3) $. If $ n=2k+1 $ with $ k\ge 1 $, then the possible homology types are \begin{enumerate}
			\item $ (-(k-1)H+kE_1-E_2-E_3,H_1,kH-(k-1)E_1,H_{12},E_2-E_3,E_3),\\ (-kH+(k+1)E_1,H_1,(k+1)H-kE_1-E_2-E_3,E_3,E_2-E_3,H_{12}), $
			\item $ (-kH+(k+1)E_1,H_1,(k+1)H-kE_1-E_2,E_2-E_3,E_3,H_{123}),$
			\item $ (-(k-1)H+kE_1-E_2-E_3,E_2,H_{12},kH-(k-1)E_1,H_{13},E_3),\\ (-kH+(k+1)E_1,H_{12},E_2,(k+1)H-kE_1-E_2-E_3,E_3,H_{13}) $.
		\end{enumerate}
		Then we have \[tF_{2k+1}(3;1,\delta_1,\delta_2,\delta_3)=\begin{cases}
		5 \quad \text{if } -k+(k+1)\delta_1>0, \delta_2>\delta_3,\delta_1+\delta_2+\delta_3<1 \\
		2 \quad \text{if } -k+(k+1)\delta_1>0, \delta_2>\delta_3,\delta_1+\delta_2+\delta_3=1 \\
		2 \quad \text{if } -k+(k+1)\delta_1>0, \delta_2=\delta_3,\delta_1+\delta_2+\delta_3<1 \\
		1 \quad \text{if } -k+(k+1)\delta_1>0, \delta_2=\delta_3,\delta_1+\delta_2+\delta_3=1 \\
		2 \quad \text{if } \delta_1+\delta_2+\delta_3-1<-k+(k+1)\delta_1\le 0, \delta_2>\delta_3 \\
		1 \quad \text{if } \delta_1+\delta_2+\delta_3-1<-k+(k+1)\delta_1\le 0, \delta_2=\delta_3 \\
		0 \quad \text{if } -k+(k+1)\delta_1\le \delta_1+\delta_2+\delta_3-1
		\end{cases}\]
		If $ n=2k+2 $ with $ k\ge 1 $, then the possible homology types are \begin{enumerate}
			\item $ (-kH+(k+1)E_1-E_i,H_1,(k+1)H-kE_1-E_j,E_j,H_{123},E_i) $, with $ \{ i,j \}=\{ 2,3 \} $
			\item $ (-kH+(k+1)E_1-E_2,H_1,(k+1)H-kE_1,H_{123},E_3,E_2-E_3) $
			\item $ (-kH+(k+1)E_1-E_i,E_i,H_{1i},(k+1)H-kE_1-E_j,E_j,H_{1j}) $, with $ \{ i,j \}=\{ 2,3 \} $
		\end{enumerate}
		Then we have \[tF_{2k+2}(3;1,\delta_1,\delta_2,\delta_3)=\begin{cases}
		5 \quad \text{if } -k+(k+1)\delta_1-\delta_2>0, \delta_2>\delta_3,\delta_1+\delta_2+\delta_3<1 \\
		2 \quad \text{if } -k+(k+1)\delta_1-\delta_2>0, \delta_2>\delta_3,\delta_1+\delta_2+\delta_3=1 \\
		2 \quad \text{if } -k+(k+1)\delta_1-\delta_2>0, \delta_2=\delta_3, \delta_1+\delta_2+\delta_3<1 \\
		1 \quad \text{if } -k+(k+1)\delta_1-\delta_2>0, \delta_2=\delta_3, \delta_1+\delta_2+\delta_3=1 \\
		2 \quad \text{if } \delta_3-\delta_2<-k+(k+1)\delta_1-\delta_2\le 0, \delta_1+\delta_2+\delta_3<1 \\
		1 \quad \text{if } \delta_3-\delta_2<-k+(k+1)\delta_1-\delta_2\le 0, \delta_1+\delta_2+\delta_3=1 \\
		0 \quad \text{if } -k+(k+1)\delta_1-\delta_2\le \delta_3-\delta_2
		\end{cases}\]
		
		If $ n=2 $, the possible homology types are \begin{enumerate}
			\item $ (E_i-E_j,H_i,H_k,E_k,H_{123},E_j), \{ i,j,k\}=\{ 1,2,3\},i<j$
			\item $ (E_1-E_2,H_1,H,H_{123},E_3,E_2-E_3) $
			\item $ (E_i-E_j,E_j,H_{ij},H_k,E_k,H_{ik}),\{i,j,k \}=\{1,2,3 \},i<j $
		\end{enumerate}
		Then we have \[tF_{2}(3;1,\delta_1,\delta_2,\delta_3)=\begin{cases}
		7 \quad \text{if } \delta_1>\delta_2>\delta_3,\delta_1+\delta_2+\delta_3<1 \\
		3 \quad \text{if } \delta_1>\delta_2>\delta_3,\delta_1+\delta_2+\delta_3=1 \\
		2 \quad \text{if } \delta_1>\delta_2=\delta_3,\delta_1+\delta_2+\delta_3<1 \\
		1 \quad \text{if } \delta_1>\delta_2=\delta_3,\delta_1+\delta_2+\delta_3=1 \\
		2 \quad \text{if } \delta_1=\delta_2>\delta_3,\delta_1+\delta_2+\delta_3<1 \\
		1 \quad \text{if } \delta_1=\delta_2>\delta_3,\delta_1+\delta_2+\delta_3=1 \\
		0 \quad \text{if } \delta_1=\delta_2=\delta_3\\
		\end{cases}\]
		If $ n=1 $, there is only one possible homology type $ (H_{12},E_1,H_{13},E_3,H_{23},E_2) $ and is always realized. So we have $$ tF_1(3;1,\delta_1,\delta_2,\delta_3)=1 .$$
		Now the results follow from \[
		|t\mc{LCY}(3;1,\delta_1,\delta_2, \delta_3)|=\sum_{n=1}^{\infty } tF_n(3;1,\delta_1,\delta_2, \delta_3),
		\]
		and analyzing the bounds on $ \delta_1,\delta_2 , \delta_3 $ in each region. For example, in region $Q_{i-1}R_{i-1}P_{2i-1}P_{2i-2}$, one just have to check the followings:
		\begin{enumerate}
			\item $-k+(k+1)\delta_1>0 \quad \text{if and only if} \quad k< i$
			\item $ \delta_1+\delta_2+\delta_3-1<-k+(k+1)\delta_1\leq 0 \quad \text{will never happen for any} \,\, k $
			\item $-k+(k+1)\delta_1-\delta_2>0 \quad \text{if and only if} \quad k<i$
			\item $\delta_3-\delta_2<-k+(k+1)\delta_1-\delta_2 \leq 0 \quad \text{will never happen for any} \,\, k $
			\item $\delta_1>\delta_2>\delta_3,\delta_1+\delta_2+\delta_3<1$
		\end{enumerate}
		In the region $Q_{i-1}Q_{i}P_{2i-1}$, one can check that:
		\begin{enumerate}
			\item $-k+(k+1)\delta_1>0 \quad \text{if and only if} \quad k< i$
			\item $ \delta_1+\delta_2+\delta_3-1<-k+(k+1)\delta_1\leq 0 \quad \text{if and only if}\quad k=i $
			\item $-k+(k+1)\delta_1-\delta_2>0 \quad \text{if and only if} \quad k<i$
			\item $\delta_3-\delta_2<-k+(k+1)\delta_1-\delta_2 \leq 0 \quad \text{will never happen for any} \,\, k $
			\item $\delta_1>\delta_2=\delta_3,\delta_1+\delta_2+\delta_3<1$
		\end{enumerate}
		Therefore, the count in the region $Q_{i-1}R_{i-1}P_{2i-1}P_{2i-2}$ is given by
		\[
		|t\mc{LCY}(3;1,\delta_1,\delta_2, \delta_3)|=\sum_{n=1}^{\infty } tF_n(3;1,\delta_1,\delta_2, \delta_3)=1+7+\sum_{k=1}^{i-1}5+\sum_{k=1}^{i-1}5=10i-2
		\]
		And the count in the region $Q_{i-1}Q_{i}P_{2i-1}$ is given by
		\[
		|t\mc{LCY}(3;1,\delta_1,\delta_2, \delta_3)|=\sum_{n=1}^{\infty } tF_n(3;1,\delta_1,\delta_2, \delta_3)=1+2+\sum_{k=1}^{i-1}2+1+\sum_{k=1}^{i-1}2=4i
		\]
		These two examples just correspond to the case (1) and (6) in our statement.
		
	\end{proof}

	\subsection{Details of computing toric regions for $M_5$}\label{section:M5region}
	
	\begin{proof}[Proof of Proposition \ref{prop:toricconeM5}:]
		Firstly note that if the homological sequence for $M_5$ in Fact \ref{fact} can not be realized as a toric symplectic log Calabi-Yau divisor, then one of $\omega(H_{134}),\omega(E_1-E_4),\omega(E_2-E_5),\omega(H_{235})$ must be zero. Thus we only need to focus on the regions $MAD,MOD,MOA$ where the toric divisors might not exist.
		
		Secondly, we consider the following toric homological sequence:
		$$(E_2,E_1-E_2-E_3,E_3,H_{135},E_5,H_{45},E_4,H_{124})$$
		In the interior of region $XOA$, we see that $\omega$ has positive pairing with all the homology classes in the above sequence. Again by Proposition \ref{prop:realization}, the homological sequence can be realized as a toric symplectic Calabi-Yau divisor. So the interior of region $XOA$ is also contained in the toric cone.
		
		Next we explain the nonexistence in regions $MOD$, $MAD$ and $MOX$. In the $M_5$ case we have the following more complicated classification for self-intersection sequences:
		\begin{align*}
		&A^{1}_n=(-1,-2,-2,-1,n,-1,-1,-n-4),A^{2}_n=(-1,-3,-1,-2,n,-1,-1,-n-3)\\
		&A^{3}_n=(-3,-1,-2,-2,n,-1,-1,-n-2),B^{1}_n=(-1,-2,-1,n,-1,-2,-1,-n-4)\\
		&B^{2}_n=(-2,-1,-2,n,-1,-2,-1,-n-3),B^{3}_n=(-2,-1,-2,n,-2,-1,-2,-n-2)\\
		&C^{1}_n=(-4,-1,-2,-2,-2,n,0,-n-1),C^{2}_n=(-3,-2,-1,-3,-2,n,0,-n-1)\\
		&C^{3}_n=(-3,-1,-3,-1,-3,n,0,-n-1),D^{1}_n=(-2,-1,-4,-1,-2,n,0,-n-2)\\
		&D^{2}_n=(-1,-4,-1,-2,-2,n,0,-n-2),D^{3}_n=(-3,-1,-2,-3,-1,n,0,-n-2)\\
		&E^{1}_n=(-2,-1,-3,-2,-1,n,0,-n-3),E^{2}_n=(-1,-3,-1,-3,-1,n,0,-n-3)\\
		&F^{1}_n=(-1,-2,-2,-2,-1,n,0,-n-4)
		\end{align*}
		
		Similarly as in the case for $M_4$, we can apply the following information about the negative symplectic sphere classes by Proposition 3.4 of \cite{LiLi20-pacific}.
		
		In region $MOD$:
		\begin{align*}
		\mc{S}_{\omega}^{-2}\subset &\{E_p-E_5,H_{ijk}|1\leq p\leq 4,1\leq i<j<k\leq 5\}\\
		\mc{S}_{\omega}^{-3} \subset &\{H_{ijkl}|1\leq i<j<k<l\leq 5\},\mc{S}_{\omega}^{-4} \subset \{H_{12345}\}, \mc{S}_{\omega}^{\leq -5}=\emptyset
		\end{align*}
		In region $MAD$:
		\begin{align*}
		&\mc{S}_{\omega}^{-2}\subset \{E_1-E_p,E_q-E_5,H_{234},H_{235},H_{345},H_{125},H_{135},H_{145},H_{245}|2\leq p\leq 5,1\leq q\leq 4\}\\
		&\mc{S}_{\omega}^{-3} \subset \{-H+2E_1,H_{2345},E_1-E_p-E_q|2\leq p<q\leq 5\}\\
		&\mc{S}_{\omega}^{-4} \subset \{E_1-E_i-E_j-E_k,-H+2E_1-E_p|2\leq i<j<k\leq 5,2\leq p\leq 5\}\\
		&\text{For any $k\geq 1$:}\\
		&\mc{S}_{\omega}^{-2k-3}\subset \{ -(k-1)H+kE_1-E_2-E_3-E_4-E_5, -kH+(k+1)E_1-E_2-E_i,\\
		&-(k+1)H+(k+2)E_1,-kH+(k+1)E_1-E_p-E_q|3\leq i,p,q\leq 5,p\neq q\}\\
		&\mc{S}_{\omega}^{-2k-4} \subset \{-(k+1)H+(k+2)E_1-E_2, -kH+(k+1)E_1-E_2-E_p-E_q,\\
		&-kH+(k+1)E_1-E_3-E_4-E_5, -(k+1)H+(k+2)E_1-E_i|3\leq i,p,q\leq 5,p\neq q\}
		\end{align*}
		In region $MOX$:
		\begin{align*}
		&\mc{S}_{\omega}^{-2}\subset \{E_1-E_p,H_{ijk}|2\leq p\leq 5,1\leq i<j<k\leq 5\}\\
		&\mc{S}_{\omega}^{-3} \subset \{-H+2E_1,H_{ijkl},E_1-E_p-E_q|1\leq i<j<k<l\leq 5,2\leq p<q\leq 5\}\\
		&\mc{S}_{\omega}^{-4} \subset \{H_{12345},E_1-E_i-E_j-E_k,-H+2E_1-E_p|2\leq i<j<k\leq 5,2\leq p\leq 5\}\\
		&\mc{S}_{\omega}^{\leq -5}=\emptyset
		\end{align*}
		
		Finally by checking all the possible self-intersection sequences, one can carefully exclude all the possibilities by obeying the rules that any two adjacent homology classes have intersection number $1$ and any two nonadjacent ones have intersection number $0$.
	\end{proof}

	\bibliographystyle{amsalpha}
	\bibliography{mybib}{}

\end{document}